\documentclass[11pt]{amsart}
\usepackage{amsmath,amsfonts,amsthm,amscd,amssymb,graphicx}
\usepackage{fullpage}

\numberwithin{equation}{section}

\newtheorem{theorem}{Theorem}[section]
\newtheorem{lemma}[theorem]{Lemma}
\newtheorem{proposition}[theorem]{Proposition}
\newtheorem{corollary}[theorem]{Corollary}

\newtheorem{remark}[theorem]{Remark}
\newtheorem{definition}[theorem]{Definition}

\def\cW{\mathcal{W}}

\def\bb1{{1\!\!1}}
\def\eps{{\epsilon}}

\def\CalB{\mathcal{B}}
\def\CalC{\mathcal{C}}
\def\CalF{\mathcal{F}}
\def\CalL{\mathcal{L}}
\def\CalO{\mathcal{O}}

\def\CalU{\mathcal{U}}
\def\CalR{\mathcal{R}}
\def\CalS{\mathcal{S}}

\def\R{\Re e}

\newcommand{\Trace}{{\text{\rm Trace}}}

\newcommand{\RR}{{\mathbb R}}
\newcommand{\btheta}{{\bar \theta}}

\newcommand{\Span}{{\rm Span }}
\newcommand{\sgn}{{\rm sgn }}

\newcommand\br{\begin{remark}}
\newcommand\er{\end{remark}}
\newcommand\bp{\begin{pmatrix}}
\newcommand\ep{\end{pmatrix}}
\newcommand\be{\begin{equation}}
\newcommand\ee{\end{equation}}
\newcommand\ba{\begin{equation}\begin{aligned}}
\newcommand\ea{\end{aligned}\end{equation}}

\date{Last Updated:  December 11, 2009}

\title{Existence and stability of viscous shock profiles for 2-D isentropic 
MHD with infinite electrical resistivity}
\author[Barker, Lafitte, and Zumbrun]{Blake Barker, Olivier Lafitte, and Kevin Zumbrun}

\thanks{ This work was supported in part by the National Science Foundation award numbers DMS-0607721 and DMS-0300487.}

\address{Department of Mathematics, Indiana University, Bloomington, IN 47402}
\email{bhbarker@gmail.com}

\address{LAGA, Institut Galilee, Universite Paris 13, 93 430 Villetaneuse and CEA Saclay, DM2S/DIR, 91 191 Gif sur Yvette Cedex}
\email{lafitte at math.univ-paris13.fr}

\address{Department of Mathematics, Indiana University, Bloomington, IN 47402}
\email{kzumbrun at indiana.edu}

\begin{document}
\begin{abstract}
For the two-dimensional Navier--Stokes equations of
isentropic magnetohydrodynamics (MHD) 
with $\gamma$-law gas equation of state, $\gamma\ge 1$,
and infinite electrical resistivity, we carry out a global
analysis categorizing all possible viscous shock profiles.
Precisely, we show that the phase portrait of the
traveling-wave ODE generically consists of
either two rest points connected by a viscous Lax profile,
or else four rest points, two saddles and two nodes.
In the latter configuration,
which rest points are connected by profiles depends on
the ratio of viscosities, and can involve Lax, overcompressive,
or undercompressive shock profiles. 
For the monatomic and diatomic cases $\gamma=5/3$ and $\gamma= 7/5$,
with standard viscosity ratio for a nonmagnetic gas, 
we find numerically that the
the nodes are connected by a family of overcompressive
profiles bounded by Lax profiles connecting saddles to nodes,
with no undercompressive shocks occurring.
We carry out a systematic numerical Evans function analysis
indicating that all of these two-dimensional shock
 profiles are linearly and nonlinearly
stable, both with respect to two- and three-dimensional perturbations.
For the same gas constants, but different viscosity ratios, we investigate
also cases for which undercompressive shocks appear; these are seen
numerically to be stable as well.
\end{abstract}
\maketitle


\tableofcontents

\section{Introduction}\label{sec:intro}

In this paper, we continue the investigations of
\cite{GZ,ZH,MaZ3,MaZ4,Ra,RZ,Z5,TZ,BeSZ,Br1,Br2,BrZ,HuZ,BHRZ,HLZ,HLyZ1,HLyZ2,BHZ}
on stability and dynamics of large-amplitude viscous shock profiles,
examining 
%
%
classical Lax-type and nonclassical overcompressive 
and undercompressive shocks occurring
in isentropic magnetohydrodynamics (MHD) with infinite electrical
resistivity.

Existence of large-amplitude profiles for full (nonisentropic)
magnetodydrodynamics was 
studied in pioneering works of Germain
and Conley--Smoller \cite{G,CS1,CS2}, making use of properties of
the traveling-wave ODE as a gradient system and of Conley index techniques.
Further investigations have been carried out by
Freist\"uhler--Szmolyan \cite{FS} using geometric singular
perturbation techniques and by Freist\"uhler--Roehde \cite{FR1,FR2}
using a combination of bifurcation analysis and numerical approximation.
In this generality, the traveling-wave ODE for MHD profiles is a
six-variable dynamical system, with up to four rest points corresponding to
endstates of various inviscid shock waves.
For an ideal gas law, it is known that fast and slow Lax shocks
always possess a viscous profile.
In certain special cases, or in certain limiting ratios of viscosity,
heat conduction, etc., it is known that intermediate shocks do or
do not possess profiles; 
however, in general, the profile existence problem for the full
nonisentropic case is accessible at present only numerically.
For further discussion, see \cite{FS,FR1,FR2} and references therein.

In the present work, we examine in detail the restricted case
of isentropic flow 
with infinite electrical resistivity, in 
two dimensions, for which the traveling-wave ODE becomes a
{\it planar dynamical system}.
This example exhibits 
the main features of the general case, 
in a simpler setting conducive to systematic numerical investigation.

Specifically, for a rather general equation of state (convex,
decreasing in specific volume, and blowing up at least linearly
with density as density goes to infinity)
we show in Sections \ref{RH} and \eqref{s:Lax}
that the phase portrait of the traveling-wave ODE generically consists of
either two rest points connected by a viscous Lax profile,
or else four rest points, two saddles and two nodes.
In the latter, four rest point configuration, the Lax shocks
involving consecutive rest points ordered by specific volume always
have connecting profiles.
The remaining, ``intermediate'' shocks may or may not admit profiles,
depending on the ratio of parallel to transverse viscosity. 
Specifically, we show in Section \ref{s:sing}
by phase plane (and, separately, by
singular perturbation) analysis that,
similarly as in the nonisentropic case \cite{FS,G}, any intermediate
shock with decreasing specific volume permits a connection for some
viscosity ratios
and not for others.
By entropy considerations,
shocks with increasing specific volume never have connecting profiles.
Here, and elsewhere, we without loss of generality
restrict discussion to the case of a {\it left-moving shock}.
(For right-going shocks, the ordering would be reversed.)

We supplement this abstract existence discussion by
a systematic numerical existence study for specific parameter values
in physical range.
For the most common cases of monatomic or diatomic gas, 
$\gamma=5/3$ or $\gamma= 7/5$, with standard viscosity ratio
for a nonmagnetic gas (see \eqref{etaform}), 
we find that there occurs only one profile configuration,
with the nodes connected by a family of overcompressive
profiles (intermediate shocks)
bounded by Lax profiles connecting saddles to nodes
in a four-sided configuration (one pair of opposing sides
corresponding to slow and fast Lax connections,
the other to intermediate Lax connections).
Undercompressive profiles do not seem to occur in this parameter range.

Next, restricting to the same parameters $\gamma=5/3, 7/5$,
and standard viscosity ratio,
we carry out numerically a systematic stability analysis of these waves,
using the general numerical Evans function techniques developed in
\cite{Br2,BrZ,HuZ,HLZ,HLyZ1,HLyZ2,BHZ,Z5}.
Our results, carried out up to extremely high Mach number (typically
Mach $20-40$, but in some cases up to Mach $10,000$),
indicate that {\it all of the above profiles}, 
both Lax- and overcompressive type,
{\it are spectrally stable} in the generalized Evans function sense
defined in \cite{ZH,MaZ3}, both with respect to two-dimensional
and three-dimensional perturbations.
These results are described in Sections \ref{s:lin} and \ref{s:num}.
By the abstract framework established in \cite{MaZ3,MaZ4,Z1,Ra,RZ},
this implies linearized and nonlinear time-asymptotic orbital
stability, as described for completeness in Section \ref{s:stabtheory}.
Varying the viscosity ratio, we carry out case studies also for
examples of undercompressive profiles.
Numerically, these are seen to be (Evans, hence linearly
and nonlinearly) stable as well.

Finally, in Section \ref{s:discussion} we discuss our results
and suggest directions for further study.

\section{Preliminaries}\label{prelim}

\subsection{Equations and assumptions}
In Lagrangian coordinates, the equations for compressible isentropic
magnetohydrodynamics (MHD) take the form 
\begin{equation}
\left\{\begin{array}{l}
v_t -u_{1x} = 0,\\
 u_{1t} + (p+ (1/2\mu_0)(B_2^2+B_3^2))_x =(((2\mu+\eta)/v) u_{1x})_x,\\
 u_{2t}  - ((1/\mu_0)I B_2)_x =((\mu/v) u_{2x})_x,\\
 u_{3t}  - ((1/\mu_0)I B_3)_x =((\mu/v) u_{3x})_x,\\
 (vB_2)_{t}  - (I u_2)_x =((1/\sigma\mu_0 v) B_{2x})_x,\\
 (vB_3)_{t}  - (I u_3)_x =((1/\sigma\mu_0 v) B_{3x})_x,\\
\end{array}\right.
\label{MHD}
\end{equation}
where $v$ denotes specific volume, $u=(u_1,u_2,u_3)$ velocity,
$p=p(v)$ pressure, $B=(I,B_2,B_3)$ magnetic induction,
$I$ constant, and
$\mu>0$ and $\eta>0$ the two coefficients of viscosity, 
$\mu_0>0$ the magnetic permeability, and $\sigma>0$ the electrical resistivity;
see \cite{A,C,J} for further discussion.

We restrict mainly to the case of 
an ideal polytropic gas, in which case the pressure 
function takes form
\begin{equation}
p(v)=a v^{-\gamma}
\label{eq:ideal_gas}
\end{equation}
where $a > 0$ and $\gamma \ge 1$ are constants that characterize the gas,
the limiting case $\gamma=1$ corresponding to the barotropic, or
constant-temperature approximation and $\gamma>1$ corresponding
to the isentropic, or constant-entropy approximation, of the ideal
pressure law $p(v,e)= \Gamma v^{-1}e$.
Though we do not specify $\eta$, we have in mind mainly
%
the ratio
\be\label{etaform}
\eta=-2\mu/3
\ee
typically prescribed for (nonmagnetic) gas dynamics \cite{Ba}.
(By rescaling space and time, we can rescale all transport coefficients
 by a common factor; the ratio $\eta/\mu$, however, is invariant.)

In the thermodynamical rarified gas approximation, $\gamma>1$ is the
average over constituent particles of $\gamma=(N+2)/N$, where $N$ is the
number of internal degrees of freedom of an individual particle,
or, for molecules with ``tree'' (as opposed to ring, or other
more complicated) structure,
\begin{equation}\label{gammaformula}
\gamma=\frac{2n+3}{2n + 1},
\end{equation}
where $n$ is the number of constituent atoms \cite{Ba}:
$\gamma= 5/3 \approx 1.66$ for monatomic, $\gamma= 7/5=1.4$ for  
diatomic gas.

An interesting subcase is the limit of infinite electrical resistivity
$\sigma=0$, in which the last two equations of \eqref{MHD} are replaced by
\ba\label{sigreplace}
 (vB_2)_{t}  - (I u_2)_x &=0,\qquad (vB_3)_{t}  - (I u_3)_x =0.,
\ea
and only the velocity variables $u=(u_1,u_2,u_3)$ experience 
parabolic smoothing, through viscosity.
We can restrict further to the two-dimensional case, setting
$u_3, B_3\equiv 0$ and dropping these variables from consideration,
as we shall do for most of our investigations.

\subsubsection{Eigenvalues of the $2$-d inviscid system}
The inviscid version of system \eqref{MHD} in dimension two, is,
introducing the scalar quantities $B=B_2$, $w=u_2$, 
\begin{equation}
\begin{array}{l}
v_t-u_x=0,\\
u_t+p_x+\frac{B}{\mu_0}B_x=0,\\
w_t-\frac{I}{\mu_0}B_x=0,\\
B_t+\frac{B}{v}u_x-\frac{I}{v}w_x=0,\\
\end{array}
\end{equation}
or, in quasilinear form,
\begin{equation}
\left(\begin{array}{l}v\\ B\\ u\\w\end{array}\right)_t
+\left(\begin{array}{ccccc}0&0&-1&0\\
0&0&\frac{B}{v}&-\frac{I}{v}\\
-c^2&\frac{B}{\mu_0}&0&0\\
0&-\frac{I}{\mu_0}&0&0\end{array}\right)
\left(\begin{array}{l}v\\ B\\ u\\w\end{array}\right)_x=0,
\end{equation}
where $-c^2:=p'(v)$.
This system has 
four eigenvalues of the form $\pm\sqrt{r_{\pm}}$, where $r_{\pm}$ are the roots of
\begin{equation}\label{phi}
\phi(r):=r^2-\Big( c^2+\frac{I^2+B^2}{\mu_0 v}\Big)r+\frac{I^2}{\mu_0 v}(
c^2)=0.
\end{equation}
As the discriminant of \eqref{phi}
is positive for $B\ne0$, the two roots $r_+$ and $r_-$ 
are positive real, verifying hyperbolicity.
When $B=0$, the discriminant can be zero for $c^2=\frac{I^2}{\mu_0v}$.
We do not explicitly require this computation in our analysis, but include it
for general interest/orientation.

\subsection{Viscous shock profiles and the rescaled equations}\label{sec:viscous}
A \emph{viscous shock profile} of \eqref{MHD} is a traveling-wave solution,
\begin{equation}
(v,u,B)(x,t)=(\hat v,\hat u, \hat B) (x-st),
\label{eq:tw_ansatz}
\end{equation}
moving with speed $s$ and connecting constant states 
\be\label{endstates}
(v_\pm,u_\pm,B_\pm)=\lim_{z\to \pm \infty} 
(\hat v,\hat u,\hat B)(z).
\ee
Such a solution is a stationary solution of the system of PDEs
\begin{equation}
\left\{\begin{array}{l}
v_t-sv_x -u_{1x} = 0,\\
 u_{1t} -s u_{1x} + (p+ (1/2\mu_0)|\tilde B|^2)_x =(((2\mu+ \eta)/v) 
u_{1x})_x,\\
 \tilde u_{t} -s\tilde u_{x}  - ((1/\mu_0)I\tilde B)_x =
((\mu/v) \tilde u_{x})_x,\\
 (v\tilde B)_{t} -s(v\tilde B)_{x}  - (I\tilde u)_x =
((1/\sigma\mu_0 v) \tilde B_{x})_x,
\end{array}\right.
\label{2dMHDmoving}
\end{equation}
where we have denoted $\tilde u:=(u_2,u_3)$, 
$\tilde B:=(B_2,B_3)$, i.e., a solution of the system of ODEs
\begin{equation}
\left\{\begin{array}{l}
-sv' -u_{1}' = 0,\\
  -s u_{1}' + (p+ (1/2\mu_0)|\tilde B|^2)' =(((2\mu+ \eta)/v) 
u_{1}')',\\
  -s\tilde u'  - ((1/\mu_0)I\tilde B)' =
((\mu/v) \tilde u')',\\
-s(v\tilde B)'  - (I\tilde u)' =
((1/\sigma\mu_0 v) \tilde B')'.
\end{array}\right.
\label{oprof}
\end{equation}
Integrating, we obtain 
\begin{equation}
\left\{\begin{array}{l}
-sv -u_{1} = C_1,\\
  -s u_{1} + (p+ (1/2\mu_0)|\tilde B|^2) =(((2\mu+ \eta)/v) 
u_{1}')+C_2,\\
  -s\tilde u  - ((1/\mu_0)I\tilde B) =
((\mu/v) \tilde u') +C_3,\\
-s(v\tilde B)  - (I\tilde u) =
((1/\sigma\mu_0 v) \tilde B')+C_4
\end{array}\right.
\label{iprof}
\end{equation}
for some constants of integration $C:=(C_1,\dots,C_4)$.

For fixed $C$, the rest points of \eqref{iprof}
comprise the possible endstates $(v_\pm,u_\pm,B_\pm)$
that can be connected by a viscous profile with speed $s$,
which necessarily satisfy the {\it Rankine--Hugoniot conditions}
\ba\label{oRH}
-s[v] &=[u],\\
-s[u_1]&=-\left[p+ \frac{\tilde B^2}{2\mu_0} \right],\\
-s[\tilde u]&=I\left[\frac{\tilde B}{\mu_0} \right],\\
-s[v \tilde B]&=I[\tilde u]\\
\ea
determining pairs of states connected by an inviscid shock wave,
where 
$$
[h]:=h(v_+,u_+,B_+)- h(v_-,u_-,B_-)
$$
 denotes jump in the quantity $h$ across the shock.

\subsubsection{Rescaled evolution equations}\label{rescaled}

Following \cite{HLZ,HLyZ1,HLyZ2,BHZ}, we now rescale
\be\label{scaling}
 (v, u, \mu_0, x, t, B,a)\to
\Big(\frac{v}{\epsilon}, -\frac{u}{\epsilon s}, 
\epsilon \mu_0, -\epsilon s( x-st), \epsilon s^2t,-\frac{B}{s},
\frac{a\epsilon^{-\gamma-1}}{s^2} \Big)
\ee
holding $\mu$, $\sigma$ fixed,
where $\eps:=v_-$, transforming \eqref{MHD}, \eqref{eq:ideal_gas} 
to the form
\begin{small}
\begin{equation}\label{redeqs}
\left\{
\begin{aligned}
    v_t+v_x-u_{1x}&=0\\
    u_{1t}+u_{1x}+\left(av^{-\gamma}+\left(\frac{1}{2\mu_0}\right)\left(B_2^2+B_3^2\right)\right)_x&=(2\mu +\eta)\left(\frac{u_{1x}}{v}\right)_x\\
    u_{2t}+u_{2x}-\left(\frac{1}{\mu_0}I B_2\right)_x&=\ \mu \left(\frac{u_{2x}}{v}\right)_x  \\
    u_{3t}+u_{3x}-\left(\frac{1}{\mu_0}I B_3\right)_x&=\mu \left(\frac{u_{3x}}{v}\right)_x\\
    \left(vB_2\right)_t+\left(vB_2\right)_x-\left(I u_2\right)_x&=\left(\left(\frac{1}{\sigma \mu_0 v}\right)B_{2x}\right)_x\\
    \left(vB_3\right)_t+\left(vB_3\right)_x-\left(I u_3\right)_x&=\left(\left(\frac{1}{\sigma \mu_0 v}\right)B_{3x}\right)_x
\end{aligned}
\right.
\end{equation}
\end{small}
where $p(v)=a v^{-\gamma}$.
There is no change in $\mu$ or $\eta$.

\subsubsection{Rescaled profile equations}\label{rprof}
Viscous shock profiles of \eqref{redeqs}
must satisfy the system of ordinary differential equations
\begin{equation}
\left\{\begin{array}{l}
v' -u_{1}' = 0,\\
 u_{1}' + (p+ (1/2\mu_0)|\tilde B|^2)' =(((2\mu+\eta)/v) u_{1}')',\\
 \tilde u'  - ((1/\mu_0)I\tilde B)' =((\mu/v) \tilde u')',\\
 (v\tilde B)'  - (I\tilde u)' =((1/\sigma\mu_0 v) \tilde B')',\\
\end{array}\right.
\label{2dMHDprof}
\end{equation}
together with the boundary conditions 
\[
(v,u_1,\tilde u,\tilde B) (\pm\infty)
= (v,u_1,\tilde u_2,\tilde B)_\pm.
\]
Evidently, we can integrate each of the differential equations from
$-\infty$ to $x$, and using the boundary conditions (in particular
$v_-=1$ and $u_-=0$), we find, after some elementary manipulations,
the profile equations (after having introduced shorthand 
notation  $u=u_1$, $w:=\tilde u$, $B=\tilde B$):
\begin{align}
(2\mu+\eta) v'&=v(v-1)+
v(p-p_-)+\frac{v}{2\mu_0}(B^2-B_-^2),\label{eq:profile1}\\
\mu w'&=v
w-\frac{vI}{\mu_0}(B-B_-),\label{eq:profile2}\\
\frac{1}{\sigma \mu_0} B'&=v^2B-vB_--Ivw,\label{eq:profile3}
\end{align}
with $u\equiv v-1$.
%

\subsubsection{The case $\sigma=\infty$}
When $\sigma=\infty$, we obtain in place of the final equation of 
\eqref{2dMHDprof},
$$
 (v\tilde B)'  - (I\tilde u)' =0,
$$
 or $(v B)'  - (Iw)' =0$,
yielding after integration the relation
\be\label{Brel}
B=\frac{B_-+Iw}{v}.
\ee
Substituting in \eqref{eq:profile1}--\eqref{eq:profile3},
we obtain a reduced, planar, ODE in $(v,w)$:
\ba\label{redode}
(2\mu+\eta) v'&=v(v-1)+
v(p-p_-)+\frac{1}{2\mu_0 v}((B_-+Iw)^2-v^2 B_-^2),\\
\mu w'&=v w-\frac{I}{\mu_0}\big(B_-(1-v)+Iw\big).
\ea

\subsection{The profile ODE as generalized gradient flow}
We now recall the general fact \cite{G,CS1,CS2,FR1,FR2} concerning a 
hyperbolic--parabolic conservation law
\be\label{visc}
U_t+\CalF(U)_x=(\CalB(U)U_x)_x,
\qquad
U=\bp U_1\\U_2\ep, \,
F=\bp F_1\\F_2\ep, \,
\CalB=\bp 0 & 0 \\ \CalB_{21}&\CalB_{22}\ep,
\ee
$\det \CalB_{22}\ne 0$,
possessing a convex entropy/entropy flux pair
\be\label{pair}
\nonumber
\eta:\, d^2\eta >0;
\qquad
q:\, dq= d\eta dF
\ee
that is viscosity-compatible in the sense that
\be\label{vc}
\R (d^2\eta \CalB)\ge 0,
\ee
that the associated traveling wave ODE
\be\label{twode}
\CalB(U)U'= F(U)-F(U_-) -s(U-U_-)
\ee
may be written always in the form of a generalized gradient
flow
\be\label{gradflowU}
d^2\eta \CalB(U)U'= \nabla_U \phi(U)
\ee
serving to increase $\phi(U)$ in the direction of positive $x$.
Here,
\ba\label{phidef}
\phi(U)&:=s \eta  - q + d\eta \big( F(U)-F(U_-) -s(U-U_-)\big),\\
\ea
so that
\ba\label{dphi}
\nabla_U \phi(U)&= d^2 \eta(U) \big( F(U)-F(U_-) -s(U-U_-)\big)
 = d^2 \eta \CalB(U)U'
\ea
by direct computation.
Likewise, $\phi(U)'=\nabla_U\phi \cdot U'=
  (d^2 \eta \CalB(U))U'\cdot U'\ge 0$ by \eqref{vc}.
We will refer to potential $\phi$ as the {\it relative entropy}
(more properly speaking, entropy production).

\br\label{Hrmk}
\textup{
Evidently, rest points $F(U)-F(U_-) -s(U-U_-)=0$
of the traveling-wave ODE correspond to critical points $\nabla_U \phi=0$ of
the relative entropy.
At a rest point $U_*$, the Hessian is given by
\ba\label{Hphi}
\nabla^2_U \phi(U)&= d^2 \eta(U) \big( dF(U_*) -sI\big),
\ea
so that, in particular, $\sgn \det \nabla_U^2 \phi =\sgn \det (dF-sI)$.
This gives a connection between the number of positive and negative
characteristics $a_j\in \sigma(dF-sI)$ and the type of the 
critical point of $\phi$.
}
\er

\subsubsection{Reduced gradient flow}
Using the assumed structure that
$\CalB$ has constant left kernel, \eqref{visc}, as is often
the case in applications, we may make a further simplification
by the use of entropy coordinates.
Introducing the entropy variable
\be\label{entvar}
W(U)=\bp W_1\\W_2\ep(U):=\nabla_U \eta(U),
\ee
globally invertible, by $d^2\eta>0$, and noting that
$dW/dU=d^2\eta$, we obtain from \eqref{gradflowU} the more
useful version
\be\label{gradflowW}
\tilde \CalB W'= \nabla_W \phi,
\qquad
\tilde \CalB= \CalB(d^2\eta)^{-1} =
\bp 0 & 0\\ 0 & \tilde b\ep,
\ee
\be\label{phicalc}
\nabla_W \phi= \big( F(U)-F(U_-) -s(U-U_-)\big),
\ee
where the block-diagonal form of $\tilde \CalB$ follows
from vanishing of the first row (inherited from left factor $\CalB$)
and compatibility assumption \eqref{vc}, which implies also $\Re \tilde \CalB\ge 0$.

Make, finally, the standard assumption (see, e.g.,
\cite{MaZ3,Z1,Z2}) that relation
\be\label{rel}
F_1(U)-F_1(U_-)-s(U_1-U_{1-})=0
\ee
coming from the traveling-wave
ODE may be solved for $W_1$ as a function $W_1=\cW(W_2)$ of $W_2$,
i.e.
\be\label{astar}
\det \Big(\partial_{W_1}F_1(U(W))-sI\Big) \ne 0.
\ee
Then, defining the reduced potential
\be\label{redpot}
\check \phi(W_2):=\phi(\cW(W_2), W_2),
\ee
and noting from \eqref{phicalc} that $\nabla_{W_1}\phi=0$
for $W_1=\cW(W_2)$, we obtain the relations
\be\label{checkphiprop}
\nabla_{W_2}\check \phi=\nabla_{W_2} \phi=
\Big(F_2(U)-F_2(U_-)-s(U_2-U_{2-})\Big)
\ee
and
\be\label{redflow}
\tilde b W_2'= \nabla_{W_2}\check \phi,
\qquad
\Re \tilde b>0
\ee
expressing \eqref{twode} as a {\it reduced} generalized
gradient flow in the parabolic entropy coordinates $W_2$ alone.
We remark in passing that this implies that $\phi=\check \phi$ is
{\it strictly} increasing in positive $x$ and not only nondecreasing
as shown above.
Moreover, we may find $\check \phi$ directly from \eqref{checkphiprop}
without computing either the full potential $\phi$ or
the entropy flux $q$, which in practice is a great simplification.

\br\label{entrmk}
\textup{
In particular, if there is a viscous profile connecting $U_\pm$,
we have
\be\label{strict}
\phi(U_+)=
\check \phi(U_+)>\check \phi(U_-)=\phi(U_-).
\ee
}
\er

\subsubsection{Application to MHD}
For MHD, we have a viscosity-compatible convex entropy
$$
\eta= \int_v^{+\infty} p(z)dz + |u|^2/2 + v|B|^2/2\mu_0\\
=\int_v^{+\infty} p(z)dz + |u|^2/2+|vB|^2/2\mu_0 v
$$
associated with entropy variables $(-p-|B|^2/2\mu_0, u_1, u_2, u_3, B_1/\mu_0,
 B_2/\mu_0, B_3/\mu_0)$ and denoting by $u=(u_1,u_2,u_3)$ and $B=(B_1,B_2,B_3)$ as in \eqref{MHD}.\footnote{
See \cite{Kaw} for related computations in the nonisentropic case.} The associated flux is $F(U)=(-u, -Iw, p(v)+\frac{|B|^2}{2\mu_0},-\frac{IB}{\mu_0})$.
The entropy variable is thus
$$
W:=d_U\eta= (-p-|B|^2/2\mu_0, u,w,B/\mu_0),
$$
of which the parabolic coordinates are $(u,w,B/\mu_0)$,
exactly the ones appearing already in \eqref{eq:profile1}--\eqref{eq:profile3}.
The corresponding flux density, though we do not need it, 
is $ q=|B|^2u/\mu_0 + IBw/\mu_0 + pu  $. 

Substituting into \eqref{checkphiprop}
 the relation $v=u+1$ obtained by integrating the $v$-equation
in the traveling-wave ODE, we obtain
\be\label{gradcheckphi}
\nabla_{v,w,B/\mu_0}\check\phi=\nabla_{u,w,B/\mu_0}\check\phi=
\bp
p(v)-p_- + \frac{B^2-B_-^2}{2\mu_0}+ v-1\\
-\frac{I(B-B_-)}{\mu_0} + w\\
vB-B_- -Iw
\ep,
\ee
which readily yields 
\ba\label{ischeckphi}
\check\phi(u,w,B)&= \int_1^vp(z)dz-p_-(v-1)
+\frac12
\Big(w^2+(v-1)^2+v\frac{B^2-B_-^2}{\mu_0} \Big)
\\&\quad
-\frac{I}{\mu_0}(B-B_-)w-\frac{BB_-}{\mu_0}.
\ea
One checks that 
$\partial_v\check\phi=\partial_u\check\phi=p(v)-p_-+u+\frac{B^2-B_-^2}{2\mu_0}$, $\partial_B\check\phi= \frac{1}{\mu_0}(vB-Iw-B_-)$, 
and
$\partial_w\check\phi=w-\frac{I}{\mu_0}(B-B_-)$.

\subsubsection{The case $\sigma=+\infty$}
Substituting into \eqref{checkphiprop} \eqref{ischeckphi}
 the relation $v=u+1$ obtained by integrating the $v$-equation
in the traveling-wave ODE, 
and the relation $ B=\frac{B_-+Iw}{v} $
obtained by integrating the $vB$-equation,
we obtain, denoting $\hat \phi(u,w):= \check\phi\Big(u, w, \frac{B_-+Iw}{v}\Big)$,
\be\label{ncheck}
\nabla_{v,w}\hat\phi=\nabla_{u,w} \hat \phi =
\bp
p(v)-p_- + \frac{(B_-+Iw)^2-v^2B_-^2}{2\mu_0v^2}+ v-1\\
-\frac{I(B_-+Iw-B_-v)}{\mu_0v} + w
\ep,
\ee
which readily yields
\be\label{sigphi}
\hat \phi(v,w)=\int_1^vp(z)dz-p_-(v-1)
+\frac12(v-1)^2+\frac{w^2}{2}\Big(1-\frac{I^2}{\mu_0v}\Big)
+\frac{B_-Iw}{\mu_0}\Big(1-\frac{1}{v}\Big)-\frac{B_-^2}{2\mu_0}\Big(v+\frac{1}{v}\Big).
\ee
Alternatively, this may be obtained from the definition, substituting
into \eqref{ischeckphi} the value $ \frac{B_-+Iw}{v} $ for $B$; however,
we wish to point out the simplification afforded by working with the
reduced problem, that is, to emphasize that one need not solve for
$\check \phi$ in order to find $\hat \phi$, or for $\phi$ in order
to find $\check \phi$.

\br\label{fullphirmk}
\textup{
Note that in the above we did not need to compute $q$ or
even $\eta$, but only to know the entropy variable $W$,
in order to determine the reduced potential by integration
of \eqref{checkphiprop}.
Likewise, computing the full potential $\phi$ by integration
of \eqref{phicalc}, we obtain 
$$
\phi(v, u, w, B)=u\Big(p(v)-p(1)+ 
\frac{B^2-B_-^2}{2\mu_0}\Big)+r(v)+\frac{(B-B_-)^2}{2\mu_0}
+\frac{u^2+w^2}{2}-\frac{I(B-B_-)}{\mu_0}w,
$$
where $r$ satisfies $r'(v)=p'(v)(1-v)$, hence 
$(r(v)-\int_1^vp(z)dz)'=r'-p=((1-v)p(v))'$, or
$$
\begin{aligned}
\phi(v, u, w, B)&=u\Big(p(v)-p(1)+ 
\frac{B^2-B_-^2}{2\mu_0}\Big)+
\int_1^vp(z)dz) + (1-v)p(v))\\
&\quad +\frac{(B-B_-)^2}{2\mu_0}
+\frac{u^2+w^2}{2}-\frac{I(B-B_-)}{\mu_0}w
\end{aligned}
$$
in agreement up to constant of integration
with the formula obtained by direct
substitution of $\eta$ and $q$ into \eqref{phidef}. 
A further substitution yields
$$
\begin{aligned}
\phi(u+1, u, w, B)&=u(p(v)-p(1))+\int_1^{v}p(z)dz -u p(v)+ 
\frac12 ((v-1)^2+w^2)\\
&\quad + \frac{B^2}{2\mu_0}(u+1)+\frac{B_-^2}{2\mu_0}(1-u)-\frac{BB_-}{\mu_0}-\frac{I}{\mu_0}(B-B_-)w,
\end{aligned}
$$
directly verifying the relation $\phi(u+1, u, w, B)=\check\phi(u, w, B)$.
}
\er

\subsection{Types of shocks vs. connections}\label{shocktype}
Consider a general system of conservation laws 
$$
U_t+F(U)_x= (\CalB(U)U_x)_x, \quad U\in R^n
$$
as in \eqref{visc}.
Inviscid shock waves correspond to 
triples $(U_-,U_+,s)$ satisfying the Rankine--Hugoniot conditions
\be\label{RH}
[F(U)]-s[U]=0,
\ee
where $[h]:=h(U_+)-h(U_-)$ denotes the jump
in quantity $h$ across the shock.
The {\it type} of the shock wave is defined by the degree of
compressivity
\be\label{shockindex}
\ell:= \dim \CalU(dF(U_-)-sI)+ \dim \CalS(dF(U_+)- sI) - n,
\ee
measuring the number of incoming characteristic modes relative
to the shock,
where $\CalU(M)$ and $\CalS(M)$ denote unstable and stable subspaces
of a matrix $M$, with $\ell=1$ corresponding to the classical
{\it Lax type}, $\ell>1$ nonclassical {\it overcompressive type},
and $\ell\le 0$ corresponding to nonclassical 
{\it undercompressive type}.  See \cite{ZH,MaZ3,Z1} for further discussion.

At a slightly more detailed level, we define a $j$-$k$ shock
as a shock for which $j=n-i_-+1$ and $k=i_+$ are the
indices of the largest positive characteristic speed $a_j^-$
at $U_-$ and the smallest negative characteristic speed $a_k^+$
at $U_+$, where $a_1<\dots<a_n$ denote the eigenvalues of $dF(U)$.
Lax shocks are associated with a single characteristic family
$j=k$, and we refer to them simply as Lax $k$-shocks.
For overcompressive shocks, $j<k$, and for undercompressive shocks,
$j>k$, with the degree of compressivity $\ell=k-j$ measuring the
difference between $j$ and $k$.

Now suppose (as in the present case) that $\CalB$ has a constant left
kernel and constant rank, 
without loss of generality
\be\label{B}
\CalB=\begin{pmatrix} 0 & 0\\ b_1 & b_2\end{pmatrix}
\quad 
\hbox{ \rm with $b_2$ nonsingular},
\ee
and that, if we denote by $A$ the Jacobian matrix of the flux $F$,
\be\label{anon}
A_*:=A_{11}- A_{12}b_1 b_2^{-1}
\quad \hbox{\rm  is nonsingular 
with real eigenvalues}. 
\ee
(In the case that there exist a viscosity compatible convex entropy, 
it may be checked \cite{MaZ3} that $A_*$ necessarily has
real eigenvalues, and $\det A_*\ne 0$ is equivalent to \eqref{astar}.)
It follows that traveling wave ODE \eqref{twode} can be expressed
as a nondegenerate {\it reduced ODE} on a manifold of dimension $r:=\dim U_2$;
in the case that there exist a compatible convex entropy,
it can simply be expressed as the reduced ODE \eqref{redflow} in
$W_2$.

Suppose further that the shock is noncharacteristic,
\be\label{nonchar}
\det (dF(U_\pm -s))\ne 0,
\ee
and the endstates satisfy the dissipativity condition
\be\label{diss}
\Re \sigma \big( -i\xi A -\xi^2 \CalB\big)_\pm \le \frac{-\theta \xi^2}{1+\xi^2},
\; \theta>0,
\; \hbox{ \rm for all }\; \xi \in R, 
\ee
where, here and elsewhere,
$\sigma(M)$ denotes spectrum of a matrix or linearized operator $M$.
In the case that there exist a viscosity-compatible convex entropy
in the vicinity of $U_\pm$, \eqref{diss} is equivalent to the
{\it genuine coupling condition} of Kawashima \cite{Kaw} that
no eigenvector of $A_+$ lie in the kernel of $\CalB_+$, and likewise
for $A_-$ and $\CalB_-$. 

All of these assumptions are satisfied quite generally in applications,
in particular for the equations of isentropic or nonisentropic MHD 
with ideal pressure law.
See \cite{MaZ4,Z1,GMWZ1,GMWZ2} for further discussion and examples.

\begin{lemma}[\cite{MaZ3}]\label{hyp}
Under the standard assumptions 
\eqref{B}, \eqref{anon}, \eqref{nonchar}, \eqref{diss}, 
$U_\pm$ are hyperbolic
rest points of the reduced traveling-wave ODE, i.e., have
stable and unstable but no center manifolds.
In particular, for $F,\CalB\in C^1$, traveling-wave solutions 
exhibit exponential convergence
\be\label{expdecay}
|\bar U(x)-U_\pm | \le Ce^{-\theta |x|}, \quad \theta >0.
\quad \hbox{ \rm for }\; x\gtrless 0.
\ee
\end{lemma}

\begin{proof}
Block matrix reduction and standard invariant manifold
theory; see Appendix A, \cite{MaZ3}.
\end{proof}

Denoting by $d_+$ the dimension of the
stable manifold of the $r$-dimensional reduced ODE at
the rest point corresponding to $U_{+}$ and by $d_-$
the dimension of the unstable manifold at the rest point
corresponding to $U_{-}$, define the {\it connection number}
\be\label{odeindex}
d:=d_++ d_-- r
\ee
measuring the type of the potential connection between rest points $U_\pm$
as a connecting orbit of the reduced ODE.
Then, we have the following fundamental relation, generalizing the
corresponding observation of \cite{MP} in the strictly parabolic case.

\begin{lemma}[\cite{MaZ3}]
Under the standard assumptions 
\eqref{B}, \eqref{anon}, \eqref{nonchar}, \eqref{diss}, 
\be\label{led}
\ell=d.
\ee
More precisely,
\be\label{led2}
i_+=d_+ + \dim \CalS(A_*),
\quad
i_-=d_- + \dim \CalU(A_*).
\ee
\end{lemma}

\begin{proof}
Results \eqref{led}--\eqref{led2} are obtained in \cite{MaZ3} 
under the additional assumption that there exist a connecting profile.  
However, the proof uses existence
only to conclude via homotopy that the number of positive eigenvalues
$\dim \CalU(A_*)$
of $A_*$ is the same at $U_+$ as at $U_-$, under the weaker assumption
that $A_*$ have real 
nonvanishing eigenvalues {\it only along the profile}.
Under our global assumption on $A_*$, we have the same conclusions
also in the absence of a profile.
\end{proof}

That is, the type of the inviscid shock wave determines the type
of the potential connection.  In the simple, planar setting
\eqref{redode} of the $\sigma=\infty$ case, we have the simple
relation that {\it Lax shocks correspond to saddle--node connections,
overcompressive shocks to repellor--attractor connections,
and undercompressive shocks to saddle--saddle connections.}

An important consequence is that viscous profiles associated with
Lax or undercompressive shocks are generically unique up to
translation, while profiles associated with 
overcompressive shocks generically appear as part of an
$\ell$-parameter family (counting translations).

\subsubsection{Type and orientation} \label{orientation}
We point out in passing a similar reduction
 principle at the level of the Rankine--Hugoniot equations,
this time measuring the parity of $d_\pm$, or equivalently
the orientation $\sgn \det(df(U_\pm)-sI)$ of roots $U_\pm$
of the Rankine--Hugoniot relations.
It is sometimes the case that certain of the Rankine--Hugoniot equations
\eqref{RH} can be solved for certain variables in terms of others,
that is, without loss of generality, after relabeling $U=(U_a,U_b)$, 
 $F=(F_a,F_b)$, that $\partial_{U_b} F_b-I$ is invertible,
so that $F_b(U_a, \Psi(U_a))\equiv 0$.  In this case, \eqref{RH}
reduces to
\be\label{redRH}
0=\tilde F_a (U_a):= F_a(U_a,\Psi(U_a)).
\ee
In the present case of isentropic MHD, we will reduce to a scalar
equation in the specific volume $v$.

Evidently, we have in this case 
$
\det(dF(U)-I)= \det(\partial_b \partial F_b(U)- I)
\det(d \tilde F_b(U)- I),
$
whence, since
$\det(\partial_b \partial F_b(U)- I)$ is real and nonvanishing by assumption,
\be\label{signred}
\sgn (\det(dF(U)-I))= \omega \, \sgn (\det(d \tilde F_b(U)- I)),
\qquad \omega \equiv \pm 1.
\ee
That is, {\it the orientation of zeros of the full Rankine--Hugoniot
relations is determined by the orientation of zeros of the
reduced Rankine--Hugoniot relations \eqref{redRH}.}
We make use of this later to help determine the types of rest points by
consideration of a scalar reduced relation.  Similar reasoning is
used in \cite{FR1} for a planar reduced relation, looking at orientations
of intersections of nullcline curves (equivalent to orientation of zeros
of the planar reduced condition).

See Remark \ref{Hrmk} and Appendix \ref{signature} for related
observations.

\subsection{The Evans function and stability}\label{s:stabtheory}
We conclude these preliminaries by a brief discussion of stability
of general traveling-wave profiles, as determined by an {\it Evans function},
or ``generalized spectral stability'' condition.
Throughout this section, we make the general assumptions 
\eqref{B}, \eqref{anon}, \eqref{nonchar}, \eqref{diss} of \cite{MaZ3,MaZ4,Z1},
as hold in particular for the MHD equations studied here.
We add to these the further assumption of {\it symmetric-dissipative
hyperbolic--parabolic form} \cite{Z1,Z2}:
\medskip

(S) \quad There exist coordinates $W$ for which \eqref{visc} becomes
$G(W)_t + F(W)_x= (\tilde \CalB(W)W_x)_x$, with $dG$ symmetric positive
definite and block-diagonal, $dF_{11}$ symmetric and either negative
or positive definite, and $\tilde \CalB$ block-diagonal, with
$\Re \tilde \CalB_{22} $ positive definite.
(Here and elsewhere, $\Re (M):=(1/2)(M+M^*)$ denotes the symmetric
part of a matrix or linear operator $M$.)

\medskip
\noindent
This structure guarantees the minimal properties needed to carry
out an analysis,  in particular that the nonlinear equations
be local well-posed and 
that the linearized equations generate a $C^0$ semigroup;
see \cite{Z2,GMWZ1,GMWZ2} for further discussion.
It is implied by existence of a viscosity-compatible convex entropy
together with the condition that $ \sigma ( A_*)$, real and nonzero
by assumption,
be also strictly positive or strictly negative, a minimal further
requirement since
$A_*$ in most applications is a scalar multiple of the identity.
{\it In particular, (S) and all other hypotheses are
satisfied for the equations of MHD with ideal
gas equation of state \cite{MaZ4,Z1} under the single condition
\eqref{nonchar} of noncharacteristicity.}

Linearizing about a stationary wave $U\equiv \bar U(x)$ of \eqref{visc}
(stationarity may always be achieved by a change to coordinates moving
with the wave), we obtain linearized evolution equations
\be\label{lin}
U_t=\CalL U:=(\CalB U_x)_x-(AU)_x,
\ee
where $A$ and $\CalB$ depend on $x$, converging asymptotically
to values $A(\pm \infty)=dF(U_\pm)$, $\CalB(\pm \infty)=\CalB(U_\pm)$.
By asymptotic convergence \eqref{expdecay} and dissipativity,
\eqref{diss}, we find from a standard result of Henry \cite{He}
equating essential spectrum of asymptotically constant coefficient operators
to that of their limiting constant-coefficient operators, that
$$
\sigma_{ess}(\CalL)\subset \{\lambda: \, \Re \lambda \le -\theta |\Im \lambda|^2\}
$$
for some $\theta>0$, where $\sigma_{ess}$ denotes essential spectrum
(defined as the part of the spectrum not consisting of eigenvalues); 
see \cite{AGJ,GZ,Z1}.
Moreover, this bound is sharp; in particular, {\it $\lambda=0$ is in the limit
of the essential spectrum.}
At the same time, $\lambda=0$ is always an eigenvalue of $\CalL$, due to
translational invariance of the underlying equations \eqref{visc}, 
with associated eigenfunction $\bar U'$.

The fact that there is no gap between the spectrum of $\CalL$ and the
imaginary axis makes this a degenerate case for which
linearized and nonlinear stability analysis is trickier than usual.
In the standard case of a sectorial operator 
for which there exists a spectral gap, one may conclude bounded linear
stability from the {\it spectral stability conditions}
of (i) nonexistence of unstable eigenvalues $\Re \lambda>0$,
and (ii) semisimplicity of neutral eigenvalues $\Re \lambda=0$;
indeed, these are necessary and sufficient.
However, here we have a nonsectorial operator with no spectral gap.
Moreover, for the eigenvalue $\lambda=0$ embedded in the essential
spectrum of $\CalL$, it is not clear even what is the meaning of semisimplicity;
see discussions of \cite{ZH,MaZ3,Z1,Z2}.

Nonetheless, as shown in \cite{GZ,ZH,MaZ3,MaZ4,Z1}, 
one can extract a simple necessary and sufficient condition
for stability analogous to (i)--(ii) 
in terms of the {\it Evans function} $D(\lambda)$
associated with $\CalL$, a Wronskian
\be\label{Ddef1}
D(\lambda):=\det (W_1^-, \dots, W_k^-, W_{k+1}^+, \dots, W_N)|_{x=0}
\ee
defined in terms of analytically-chosen 
bases $ \{W_1^-, \dots, W_k^-\}(\lambda,x)$ 
and $\{ W_{k+1}^+, \dots, W_N)\}(\lambda,x)$ of
the manifolds of solutions decaying as $x\to \infty$ and $x\to +\infty$
of the eigenvalue equations $(\CalL-\lambda)w=0$ written as a first-order
system 
\be\label{firstorder}
W'=A(x,\lambda)W,
\ee
where $W$ is an augmented ``phase variable'' 
including $w$ and suitable derivatives.
By standard considerations, this may be defined on the complement
of $\sigma_{ess}(\CalL)$; a more detailed look shows that $D$ permits an
analytic extension to the boundary of this set-- in particular,
to the nonstable complex half-plane  $\{\Re \lambda \ge 0\}$.
For details of this construction, see, e.g., \cite{AGJ,GZ,Z1,HuZ};
we give some further discussion also in 
Section \ref{s:lin} and Appendix \ref{s:conj}.

Evidently, away from the essential spectrum $\sigma_{ess}(\CalL)$,
the Evans function vanishes at $\lambda$ if and only if $\lambda$
is an eigenvalue of $\CalL$, corresponding to existence of a solution
of the eigenvalue equations decaying at both $x\to \pm \infty$.
Indeed, the multiplicity of the root is equal to the multiplicity
of the eigenvalue \cite{GJ1,GJ2,MaZ3,Z1}.
The meaning of the multiplicity of the root of $D$ at embedded eigenvalue
$\lambda=0$ is less obvious, but is always greater than or equal to the
order of the embedded eigenvalue \cite{MaZ3,Z1}.

By the discussion in Section \ref{shocktype}, in particular relation
\eqref{led}, a traveling-wave profile $\bar U$ lies in an 
$\tilde \ell$-parameter family of nearby solutions, where $\tilde \ell$ is
(by dimensionality) at least $\min\{1,\ell\}$, where $\ell$
is the degree of compressivity defined in \eqref{shockindex},
with equality in the case that the connection is a maximally transversal
intersection of the unstable manifold at $U_-$ with the stable
manifold at $U_+$.
Assume for simplicity the typical case that equality holds,
\be\label{tildel}
\tilde \ell= \min\{1,\ell\},
\ee
and the manifold of nearby solutions is smooth.
Then, the stability condition is
\medskip

(D) $D$ has precisely $\tilde \ell$ roots on the 
nonstable half-plane $\{\lambda: \, \Re \lambda \ge 0\}$,
necessarily at $\lambda=0$.
\medskip

\noindent This is analogous to the stability condition
in the standard sectorial case, with nonvanishing away from $\lambda=0$ 
corresponding to the standard spectral condition (i), and
vanishing to order $\tilde \ell$ at $\lambda=0$ indicating that
the multiplicity of this zero is
accounted for entirely by genuine eigenfunctions corresponding to
variations of the traveling wave connection along the $\tilde \ell$-parameter
family of nearby solutions, a generalized version of
semi-simplicity \cite{ZH,MaZ3,Z1}.

\subsubsection{Linear and nonlinear stability}\label{s:stabsec}
We have the following basic results relating the Evans condition
(D) to stability.

\begin{proposition}[\cite{MaZ3}]
Under the standard assumptions 
\eqref{B}, \eqref{anon}, \eqref{nonchar}, \eqref{diss}, and
assuming an $\tilde \ell$-parameter family of traveling-wave
solutions near $\bar U$, (D) is necessary and sufficient for
linearized stability from $L^1\cap L^\infty \to L^p$ of $\bar U$, all
$1\le p\le \infty$, defined as
$ |e^{Lt}f|_{L^p}\le C|f|_{L^1\cap L^p}.  $ 
\end{proposition}


\begin{proposition}[\cite{MaZ4,RZ}]\label{orbital}
Under assumptions \eqref{B}--\eqref{diss}, (S),
and definining $\tilde \ell$ as in \eqref{tildel},
the Evans condition (D) implies, first, 
existence of a $C^1$ family of nearby solutions $\{\bar u^\alpha\}$,
$\alpha\in R^{\tilde \ell}$, and, second, nonlinear time-asymptotic orbital stability,
in the following sense:
For any solution 
$\tilde U$ of \eqref{visc} with initial difference
$E_0:=\|(1+|x|^2)^{3/4}( \tilde U(\cdot, 0)- \hat U)\|_{ H^5}$
sufficiently small and some uniform $C>0$, $\tilde U$ exists for all $t\ge 0$, with
\begin{equation}
\label{stabstatement}
\begin{aligned}
\|(1+|x|^2)^{3/4}( \tilde U(\cdot, t)- \hat U(\cdot-st))\|_{ H^5}
&\le CE_0
\quad
\hbox{\rm (stability)}.\\
\end{aligned}
\end{equation}
Moreover, there exist $\alpha(t)$, $\alpha_\infty$ such that
\begin{equation}\label{stabstatement2}
\begin{aligned}
\| \tilde U(\cdot, t)-  \hat U^{\alpha(t)}(\cdot -st))\|_{L^p}
&\le CE_0(1+t)^{-(1/2)(1-1/p)},\\
\end{aligned}
\end{equation}
and
\begin{equation}
\label{phasebd}
|\alpha(t)- \alpha_\infty|\le C E_0 (1+t)^{-1/2}, 
\quad |\dot \alpha(t)| \le C E_0 (1+t)^{-1}, 
\end{equation}
for all $1\le p\le \infty$.
\quad
{\rm (phase-asymptotic orbital stability)}.
\end{proposition}

A similar result holds in the 
mixed, under-overcompressive case that the
family of nearby traveling waves has dimension different from
(necessarily greater than) $\tilde \ell$; see \cite{RZ}.

\subsubsection{The integrated Evans condition}
Noting that $\CalL=\partial_x (B \partial x - A)$
is in divergence form, we may conclude for any $\lambda \ne 0$
that satisfaction of the eigenvalue ODE $(\CalL-\lambda)w=0$ for
an solution $w$ decaying exponentially in $x$ up to one derivative implies that
$\tilde w(x):=\int_{-\infty}^x w(y)dy$ is also bounded and exponentially
decaying, and satisfies the {\it integrated eigenvalue equation}
\be\label{inteval1}
(\tilde \CalL- \lambda)\tilde w=0,
\ee
where $\tilde \CalL:= B\partial_x^2 - A\partial_x$.
Associated with $\tilde \CalL$ is an {\it integrated Evans function}
$\tilde D(\lambda)$, which like $D$ may be defined analytically
on the nonstable half-plane $\{\lambda: \, \Re \lambda \ge 0\}$.
This permits the following simplified stability condition,
in practice easier to verify.

\begin{proposition}[\cite{ZH,MaZ3}]\label{intcond}
Under assumptions \eqref{B}--\eqref{diss}, 
in the Lax or overcompressive case,
the Evans condition (D) is equivalent to the {\rm integrated Evans condition}
\medskip

($\tilde D$) $\tilde D$ is nonvanishing on the
nonstable half-plane $\{\lambda: \, \Re \lambda \ge 0\}$,
\medskip

\noindent and in the undercompressive case to
\medskip

($\tilde D$') $\tilde D$ has on the
nonstable half-plane $\{\lambda: \, \Re \lambda \ge 0\}$
a single zero of multiplicity one at $\lambda=0$.
\medskip
\end{proposition}

In the Lax and overcompressive cases that are the main focus of our
investigation here, the change to integrated coordinates has the
effect of removing the zeros of $D$ at the origin, making the Evans
function easier to compute numerically and the Evans condition easier to verify.

\section{Rankine-Hugoniot Conditions}\label{s:RH}
The Rankine-Hugoniot conditions for isentropic MHD are,
in the notation $u=u_1$, $B=(B_2,B_3)$, $w=(u_2,u_3)$,
\begin{align}
-s[v] &=[u],\label{eq:rh_mass}\\
-s[u]&=-\left[p+ \frac{B^2}{2\mu_0} \right],\label{eq:rh_momentum}\\
-s[w]&=I\left[\frac{B}{\mu_0} \right],\label{eq:rh_momentum2}\\
-s[vB]&=I[w].
\label{eq:rh_magfield}
\end{align}

Under the scaling \eqref{scaling},
we have $s=-1$, $v_-=1$, and without loss of generality
(by translation invariance), we may take $u_-=0$, $w_-=0$.  
Last, we may take without loss of generality (by rotational
invariance) $w_{3-}=0$, whereupon we obtain from
\eqref{eq:rh_momentum2}--\eqref{eq:rh_magfield}
that $[B_3]=[vB_3]=0$, which, so long as
\begin{equation}\label{vdiff}
v_+\ne v_-=1,
\end{equation}
gives, finally,
\begin{equation}\label{Bzero}
B_{3-}= B_{3+}=0.
\end{equation}
Collecting, we have the normalizations
\begin{equation}\label{normalization}
s=-1,\quad v_-=1, \quad
u_-=0, \quad 
w_{2-}= 0,\quad 
w_{3-}= w_{3+}=0,\quad 
B_{3-}= B_{3+}=0.
\end{equation}

To generate all possible shock profiles, up to invariances
of the equations, we shall vary $I$, 
$B_{2+}$, without loss of generality nonnegative, and $v_+$, 
without loss of generality between $0$ and $1$ (since we can
always arrange that $v_-=1$ correspond to the rest point with larger $v$ value),
and solve for the remaining coordinates
$u_+$, $B_{2-}$,
and the parameter $a$ appearing in the pressure law.
Parameters that will be important in the whole study are
\begin{equation}\label{def: magnetic}
J:=\frac{(B_{2-})^2}{2\mu_0} \; \hbox{\rm and } 
K:=\frac{I^2}{\mu_0}.
\end{equation}
(Note that, under the rescaling that we used, 
$I=-\frac{I}{s}$,
$J= \frac{B_{2-}^2}{2\eps s^2 \mu_0}=
 \frac{B_{2-}^2}{2v_- s^2 \mu_0} $,
$K= \frac{(I)^2}{\eps s^2 \mu_0}=
 \frac{(I)^2}{v_- s^2 \mu_0} $ in the original coordinates.)

\begin{remark}
\textup{
In the excluded case $v_+=v_-=1$, profiles are prohibited by
entropy consideration, \eqref{strict}.
}
\end{remark}

\begin{remark}
\textup{
Note that it does {\it not} follow in general that $\hat B_3\equiv 0$
or $\hat w_3\equiv 0$, but does follow when profiles are unique,
i.e., in the Lax or undercompressive case,
and one such profile is known to exist. 
(Recall the discussion of types of shocks
and relation to uniqueness of profiles in Section \ref{shocktype}).
}
\end{remark}

\begin{proposition}\label{rhprop}
Under the normalizations \eqref{normalization},
for each $0<v_+\le 1$ and $I, B_{2+}\ge 0$, 
the Rankine--Hugoniot equations \eqref{eq:rh_mass}--\eqref{eq:rh_magfield} 
have a unique solution
\begin{equation}\label{rhsoln}
u_+=v_+-1, \quad 
B_{2-}=\Big(\frac{v_+-K}{1-K}\Big) B_{2_+},
\quad
w_+=\frac{K}{I}\Big(\frac{1-v_+}{1-K}\Big)B_{2_+},
\quad
\end{equation}
\begin{equation}\label{asoln}
a=
\Big(\frac{1-v_+}{v_+^{-\gamma} -1}\Big)
\Big(
1- \frac{B_{2_+}^2}{2\mu_0}
\frac{(1+v_+-2K)}{(1-K)^2}
\Big)
=
\Big(\frac{1-v_+}{v_+^{-\gamma} -1}\Big)
\Big( 1 - J \frac{(1+v_+-2K)}{(v_+-K)^2} \Big).
\end{equation}
This is physically meaningful if and only if $a>0$, or
\begin{equation}\label{physa}
-1<v_+-1<2(K-1)+(1-K)^2\Big(\frac{2\mu_0}{B_{2_+}^2}\Big).
\end{equation}
(For $K\ge 1/2$, this gives no restriction.
For $K<1/2$, $B_{2_+}^2< \frac{2\mu_0(1-K)^2}{1-2K}$
or $J< \frac{(v_+-K)^2}{1-2K}$.)
\end{proposition}

\begin{proof}
From $[u]=[v]$, we obtain immediately $u_+=v_+-1$.
Expanding $[Bv]=I[w]=K[B_2]$ and solving, we obtain
$$
B_{2-}=\Big(\frac{v_+-K}{1-K}\Big) B_{2_+}.
$$
From $[w]=(K/I)[B_2]$, we then obtain
$w_+=\frac{K}{I}\Big(\frac{1-v_+}{1-K}\Big)B_{2_+}$.
Finally, from the remaining condition 
$[u]=-\left[p+ \frac{B^2}{2\mu_0} \right]$, 
we obtain
\be\label{keyrh}
[p]=
(1-v_+) - (1/2\mu_0)(B_{2+}^2-B_{2-}^2),
\ee
yielding \eqref{asoln} and \eqref{physa}.
\end{proof}

\br\label{gen1}
\textup{
So far, we have made no restriction on dimension or
$\sigma$, so our analysis of the Rankine--Hugoniot conditions
holds for the general three-dimensional isentropic case.
}
\er

\subsection{Global rest point configuration}
Proposition \ref{rhprop} gives a convenient means for
stepping through the possible shock connections, and is the
main method we will use to generate shocks
in our numerical investigations of shock stability.
For the study of the existence problem it is more useful
to take a global point of view, fixing a left state and
speed in the unrescaled coordinates, and studying the 
configuration of rest points (possible right states) in the resulting
traveling-wave ODE.
In the rescaled coordinates, this amounts to fixing $I$, $B_{2-}$,
and $a$, or, equivalently, the more convenient parameters $(J,K,a)$,
and solving for all possible $v_+$.

\begin{proposition}\label{parprop}
In the parallel case $J=0$, for $0<a\ne \gamma^{-1}$ and $0\le K\ne 1$,
there exists a unique parallel solution $v_*\ne 1$ satisfying
$0=g(v):=p(v)-p(1)+ v-1$, with associated magnetic field $B_{2*}=0$.
If $K$ is not between $v_*$ and $1$, then these are the only rest points,
with $v_*$ corresponding to a saddle and $1$ to a repellor if $K<v_*<1$ and 
$1$ corresponding to a saddle and $v_*$ to an attractor if $v_*<1<K$.
If $K$ lies between $v_*$ and $1$, then $1$ corresponds to a repellor
and $v_*$ to an attractor and there are two additional
nonparallel saddle-type rest points 
$$
v=K, \quad u=K-1,\quad  B=\pm \sqrt{2\mu_o (p(K)-p(1)+ K-1)}, \quad w=KB/I.
$$
\end{proposition}

\begin{proof}
We have $g(1)=0$ and $g'(1)=p'(v)+1=-\gamma a+1\ne 0$ by assumption.
Since $g$ is evidently convex, and $g\to +\infty$ as
$v\to 0, +\infty$, we find that there is precisely one other
root $v_*\ne 1$.
There are a further two solutions $v=K$ $w=-KB_{2}/I$, 
$B^2/2\mu_0= -([p]+[v])=-g(K)$, which are physically relevant
only if $g(K)\le 0$, or (by convexity) $K$ lies between
$1$ and $v_*$.
The types of the rest points may be obtained by straightforward
computation \cite{BHZ}.
\end{proof}

\begin{proposition}\label{globalprop}
In the nonparallel case
$J>0$, for $a>0$ and $0\le K\ne 1$, rest points of traveling-wave
ODE \eqref{iprof}, or, equivalently, right states satisfying
the Rankine--Hugoniot equations \eqref{eq:rh_mass}--\eqref{eq:rh_magfield} 
with $s=-1$ and $v_-=1$, 
correspond to roots $v=v_+$ of
\ba\label{tildef}
\tilde f(v)&:= 
 p(v)-p(1) + J\Big( \frac{(1-K)^2}{(v-K)^2} -1\Big) + v -1,
\ea
of which there are at most two greater than $K$ and at most
two less than $K$.
For all except a measure-zero set of parameters, 
there are exactly two or four roots in total, consisting
of an attractor $v_1$ and a saddle $v_2$ ordered as $v_1<v_2<K$, 
a saddle $v_3$ and a repellor $v_4$ ordered as $K<v_3<v_4$, or both,
with $(u,w,B)$ values determined by \eqref{rhsoln}.
Moreover, the relative entropy $\phi(v_j)$ decreases with $j$.
\end{proposition}

\begin{proof}
Combining \eqref{rhsoln}(ii) and \eqref{keyrh}, we
obtain \eqref{tildef}.
Noting that $\tilde f$, since $p$ is convex, is convex on
$(0,K) $ and $(K,+\infty)$, with $\tilde f(v)\to +\infty$
as $v\to 0, K,  +\infty$, we find that $\tilde f$ can have
at most two roots on each of the intervals $(0,K)$ and
$(K,+\infty)$.
Noting that $\tilde f$ is monotone increasing in $a$ on $(0,K)$
we find for each
fixed $(J,K)$ that there are at most two values of $a$ for which
$\tilde f$ has a double root, hence, for all except this measure
zero set of parameters, there are exactly two or four.
Applying the reduced orientation principle 
\eqref{signred} together with the reduced type relation \eqref{led2}, 
we find using the fact that $f'$ changes sign between two roots
on one side of $K$ that one must be of saddle type and the other of node type.

Finally, tracking down the orientations of intermediate transformations,
which change sign as $v$ crosses $K$, by 
the relation $w=\frac{K}{I}\Big(\frac{1-v_+}{1-K}\Big)B_{2}$,
or, more simply, directly computing the sign of the determinant
of the $2\times 2$ matrix arising from the linearization of the
planar ODE \eqref{redode} about the rest points $(v,w)$,
we find that the largest root $>K$ and the smallest root $<K$ are nodes,
and the others saddles.
Computing the trace of the $2\times 2$ coefficient matrix of
the linearized system, we find that the largest root $>K$ is
a repellor and the smallest root $<K$ an attractor.

Alternatively, and much more simply, recalling the formula \eqref{ncheck}
for $\nabla_{v,w} \check \phi$, solving 
\be\label{nablaw}
0=\nabla_w \check \phi= -\frac{I(B_-+Iw-B_-v)}{\mu_0v} + w
=
\frac{w(v-K)}{v}-\frac{IB_-(1-v)}{\mu_0v} 
\ee
for $w= \frac{IB-(1-v)}{\mu_0(v-K)}$, and substituting
into $\nabla_v\check \phi$, we find after a brief
computation that, along this nullcline,
$d\check \phi/dv= \tilde f(v)$, hence the relative entropy $\check \phi$
is {\it decreasing} with respect to $v$
between rest points lying on the same side of $K$, again identifying
nodes $>K$ as repellors and nodes $<K$ as attractors for the flow
of the planar traveling-wave ODE.  (Recall that $\check \phi$ increases
along the flow, with rest points of the flow
corresponding to critical point of $\check \phi$.)
Finally, taking without loss of generality $B_-<0$,
note that, by \eqref{nablaw}, at $v=K$,
$\nabla_w\check \phi\equiv -\frac{IB_-(1-v)}{\mu_0v} >0$ for all $w$,
so that the limiting value of $\check \phi$ as $v\to K^+$ 
on the negative-$w$ nullcline branch for $v>K$ is less than
the limiting value of $\check \phi$ as $v\to K^-$ 
on the positive-$w$ nullcline branch for $v<K$, verifying
decrease with $j$ of $\phi(v_j)$ for all $j$ and completing the proof.
%
%
\end{proof}

\br\label{genp}
\textup{
Note that the above argument depends only on the general properties
of the pressure law $p$ of convexity, blowup at $0$ at rate at least $1/v$ 
and decay as $v\to +\infty$, and not on the specific form of a polytropic
gas law, hence our conclusions extend to general pressure laws of
this type.
}
\er

\br\label{comb}
\textup {
The parallel and nonparallel cases can be combined,
associating rest points to roots of the continuous function
$\hat f(v):=(v-K)^2\tilde f(v)$.
In all cases, there is a two-parameter bifurcation at
$K=1$, with three rest points collapsing at $v=K=1$.
}
\er

\br\label{gen2}
\textup{
Though we carried out our analyis for the planar system
arising through the choice $\sigma=\infty$ and
the restriction to two dimensions, 
our conclusions on the number and type of states
satisfying the Rankine--Hugoniot conditions
apply to the general three-dimensional isentropic case.
That is, through the relation \eqref{led2} we are able to make 
quite general conclusions on types of shocks by examination of
the simple {\it planar realization} of the two-dimensional
traveling-wave ODE.  Indeed, our final argument determining
the type of rest points for the planar system by looking
along nullclines of $\nabla_w\check \phi$ amounts to
a further reduction to the {\it scalar realization}
obtained by setting $\mu=0$ as well as $\sigma=\infty$.
}
\er

Factoring out the root $v=1$, we may examine instead roots of
\begin{equation}\label{f}
f(v):= \frac{\tilde f(v)}{v-1}= 
 \frac{p(v)-p(1)}{v-1} + 
\frac{J(1+v-2K)}{(v-K)^2}  + 1.
\end{equation}

\br\label{largermk}
\textup{
In describing the possible four rest point configurations in the
nonparallel case $J>0$,
we may (by rescaling if necessary) without loss of generality
consider only the case that $v_-=1$ is the largest rest point:
that is,
$\tilde f'(1)>0$ and there is a rest point $v_+<K<1$.
Fixing $v_+$ and $K$, and letting $J$ vary, we obtain by
\eqref{asoln} that
\be\label{rels}
0<a=c-dJ,\;
\hbox{\rm  where }\;
c= \Big(\frac{1-v_+}{v_+^{-\gamma} -1}\Big)>0
\; \hbox{\rm  and }\;
d= \Big(\frac{1-v_+}{v_+^{-\gamma} -1}\Big) \frac{(1+v_+-2K)}{(v_+-K)^2}. 
\ee
Thus, since $K<1$,
\be\label{largest}
0< \tilde f'(1)= -\gamma a+1-\frac{2J}{1-K}=
(1-\gamma c) +J\Big( \gamma d - \frac{2}{1-K}\Big)
\ee
implies either 
$ d>0$, in which case $J<c/d= \frac {(v_+-K)^2} {1+v_+-2K} $ 
by \eqref{rels}, or else
$d\le 0$, in which case
$ J <\frac{1-\gamma c}{  -\gamma d + 2/(1-K)}.  $
The same considerations hold whenever there exist Lax $1$-shocks,
or, without loss of generality (by rescaling the largest rest point
to value $v_4=1$) $K<1$.
That is, {\it it is sufficient to consider a bounded parameter range
$(a,J)$
in studying four rest point configurations or Lax $1$-shocks, 
for 
$K$ bounded away from $1$.}
This is important for numerical explorations, in which the parameter
range is necessarily finite.
}
\er

\subsection{Four rest-point configurations}\label{4rest}
To aid our later numerical investigations,
we give a simple description of the set of parameters $(J,K,a)$ for
which four-rest point configurations appear, and with them
the possibility of intermediate, overcompressive, and undercompressive shocks,
without loss of generality taking $K<1$ by rescaling if necessary
so that $v_-=1$ is the largest root of $\tilde f$.

\begin{proposition}\label{4desc}
For $0\le K<1$, the set of  $J\ge 0$ and $a>0$ for which
there exist four solutions of the 
Rankine--Hugoniot equations \eqref{eq:rh_mass}--\eqref{eq:rh_magfield} 
(equivalently, four rest points of traveling-wave ODE
\eqref{eq:profile1}--\eqref{eq:profile3} or \eqref{redode}),
except for the measure-zero set of values $a=a_*(J,K):=\frac{1-K-2J}{\gamma
(1-K)}$ for
which $\tilde f'(1)=0$,
consists of a connected set
\be\label{4graph}
\{(J,K)\in \CalR:=\CalR_1\cup \CalR_2, \,  0<a<A(J,K)\}
\ee
for some $A(J,K)> 0$, where
\be\label{Rs}
\CalR_1:=\Big\{ 0\le K\le \frac{1}{2}, \, 0\le J < \frac{K^2}{2K-1}\Big\}
\quad  \hbox{\rm and }\quad
\CalR_2:=\Big\{ \frac{1}{2}\le K\le 1, \, 0\le J\Big\}.
\ee
\end{proposition}

\begin{proof}
As there is always a rest point with $v_-=1$, existence of four
rest points is equivalent (except on the measure zero set of parameters
for which degenerate roots appear), to
existence of a second rest point with $v\in (0,K)$, i.e., a
root of $\tilde f$.
Since $\tilde f$ is monotone increasing in $a$ for $0<v<K$,
this consists of an open interval $a\in (0,A(J,K))$,
with $A>0$ only if $\tilde f$ has a root in $(0,K)$
for the limiting values $(J,K,0)$.

Multiplying $\tilde f$ by $(v-K)^2/(1-v)$ reduces this
question to existence of a root $v\in (0,K)$ of the quadratic
$q(v)=-(v-K)^2 + J(1+v-2K)$,
$q'(v)= -2(v-K) + J$.
We readily compute that $q(K)= J(1-K)>0$ for $J>0$ and $q'(K)=J >0$
and that
$ q(0)= -K^2 + J(1-2K)$
and $q'(0)= 2K+J>0$, 
so that the only way there can be a root of $q$ on $0<v<K$ is if $q(0)<0$,
or
$J(1-2K)<K^2$.
The set of $(J,K)$ satisfying this condition
is easily seen to correspond to the connected set $\CalR_1\cup \CalR_2$.
\end{proof}

\br\label{Kg1}
\textup{
In the case $v>K>1$, $\tilde f$ is monotone decreasing with $a$,
and so we find that the set of parameters generating four
rest point configurations (ignoring the measure-zero set
corresponding to $f'(1)=0$) is, rather, of form $a>\tilde A(J,K)$
for arbitrary $J\ge 0$, $K\ge 1$.
The set of $(J,K)$ for which four point configurations appear
{\it for all $a>0$}, i.e., $\tilde A=0$,
is readily seen to be
$\CalR_3:=\{  K\ge 1 \; \hbox{\rm and } J> 4(K-1)\}$.
For, in this case, $q(K)=J(1-K)<0$ but $q'(K)=J>0$.
Meanwhile, $q(+\infty)<0$ as well, so that the only chance for
a root $v>K$ is that the maximum value of q be positive.
Solving $q'(v)=0$ for the critical point $v_{max}= J/2+K$,
we find that $q(v_{max})= J(J/4 + 1-K)$,
which is positive precisely for $J>4(K-1)$.
}

\textup{
Recall, for four rest point configurations, we may take without
loss of generality $K\le 1$.
}
\er

\subsection{Two-dimensional shock types}
Restricting to two dimensions, 
we find that shocks connecting rest points in decreasing order
of $v$ are of Lax $2$-type for both values $<K$, of
Lax $1$-type for both values $>K$.  Shocks connecting the
largest $v$ value to the smallest are overcompressive $1$-$2$
type, while shocks connecting the largest $v$-value $>K$ to the
largest $v$-value $<K$ are Lax $1$-type and shocks connecting
the smallest $v$-value $>K$ to the smallest $<K$ are Lax
$2$-type.
Shocks connecting the two middle (nonextremal) values of $k$
are undercompressive $2$-$1$ type.
In the terminology of the literature \cite{G,CS1,CS2,FS}, all
shocks bridging across the value $K$ are called {\it intermediate
shocks}; as shown above, these may in principle be of
Lax, overcompressive, or undercompressive type.

In our main parameter range (monatomic
or diatomic gas with standard viscosity ratio for
nonmagnetic gas),
only Lax and overcompressive
type appear to have profiles for the two-dimensional $\sigma=\infty$ case
considered here.  


\subsection{The three-dimensional case}\label{s:3dexist}
The results of Propositions \ref{rhprop}, \ref{parprop},
and \ref{globalprop} 
extend by rotation to the full, three-dimensional case to
yield the same basic $2$-$4$ rest point configuration,
with all rest points confined to a rotation of the planar case.
The single exception is in the parallel case $J=0$, for which
the data corresponding to $v_-=1$, and to the attractor $v_*<K$
if it occurs, is rotation invariant; in this case, the intermediate
rest points extend by rotation to yield a circle of rest points.
Associated intermediate shocks are called degenerate type \cite{FS};
in case $J\ne 0$, they are called nondegenerate type.
The double-cone configuration arising from rotation of the four
rest point parallel configuration, and associated interesting
bifurcations, are discussed in \cite{FS}.

Considered as waves of the full three-dimensional system,
shocks connecting rest points in decreasing order
of $v$ are of Lax $3$-type for both values $<K$, of
Lax $1$-type for both values $>K$.  Shocks connecting the
largest $v$ value to the smallest are overcompressive $1$-$3$
type, while shocks connecting the largest $v$-value $>K$ to the
largest $v$-value $<K$ are overcompressive $1$-$2$ type and shocks connecting
the smallest $v$-value $>K$ to the smallest $<K$ are overcompressive
$2$-$3$ type.
Shocks connecting the two middle (nonextremal) values of $k$,
undercompressive when considered as two-dimensional waves,
are in three dimensions of Lax $2$-type, or {\it Alfven waves}.
Thus, in three dimensions, only Lax or overcompressive shocks appear.
That is, undercompressivity is an artifact of the restriction
to two dimensions.

In our main parameter range (monatomic
or diatomic gas with standard viscosity ratio for
nonmagnetic gas),
Lax $2$-shocks do not appear to
have profiles when resricted to two dimensions, for
the $\sigma=\infty$ case considered here.  
In the full, three-dimensional $\sigma=\infty$ case, 
therefore, Lax $2$-shock profiles if they exist 
must be {\it nonplanar} in the sense that they leave the plane of 
the rest point configuration.
Likewise, overcompressive shocks, besides the planar connections
studied here, admit also nonplanar connections when considered
in the full, three-dimensional setting.  We shall not study
such genuinely three-dimensional profiles here, restricting
attention to planar profiles that can be studied within the
two-dimensional framework.
For discussion of fully three-dimensional phenomena in
the related nonisentropic case, see \cite{G,CS1,CS2,FS}.

\begin{figure}[t]
\includegraphics[width=7.5cm]{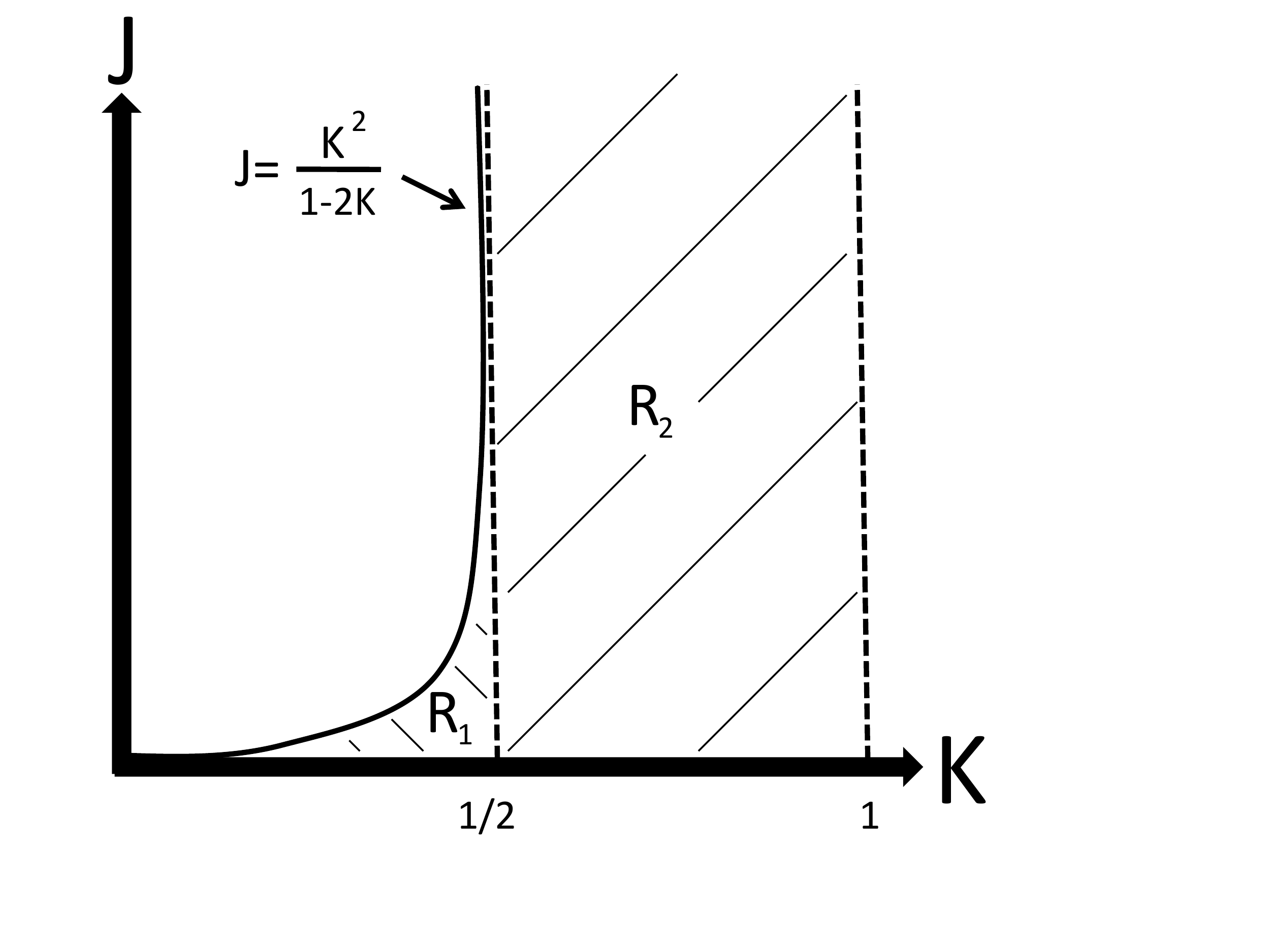}
\caption{Region of four rest point configurations, $a=0$,
in $(K,J)$ space.}
\label{regions}
\end{figure}

\section{Existence of profiles}\label{profexist}

In this section, we describe the possible 
viscous shock profile connections for the various
rest point configurations described in Section \ref{s:RH}.
Typical phase portraits (determined numerically)
for the two variable system \eqref{redode} with $\sigma=\infty$ 
are graphed in Figures \ref{phaseA}, \ref{phaseB}, and \ref{phaseC}. 

\subsection{The parallel case, $J=0$}\label{s:Lax}

\begin{proposition}\label{parexist}
In the parallel case $J=0$, for $0<a\ne 1$ and $0\le K\ne 1$,
and $\sigma=\infty$, assuming without loss of generality
that $1$ is the largest rest point of the traveling-wave equation,
there is always a profile connecting $v_-=1$ and the
unique parallel rest point $v_*\ne 1$, $w_*=0$, $B_{2*}=0$,
which is of Lax $1$-type if $K<v_*$, Lax $3$-type if $K>1$,
and overcompressive type if $v_*<K<1$.
In the latter case, there are Lax connections from repellor $v_-=1$ to
the additional saddle-type rest points
$$
v=K, \quad u=K-1,\quad  B=\pm \sqrt{2\mu_o (p(K)-p(1)+ K-1)}, \quad w=KB/I,
$$
and from these saddle-type rest points to the attractor $v_*$, whose
orbits bound a four-sided region foliated by overcompressive connections.
\end{proposition}

\begin{proof}
In the parallel case, \eqref{redode} reduces to
\ba\label{parredode}
(2\mu+\eta) v'&=vh(v)
+\frac{Jw^2}{2 v},\\
\mu w'&=(v-K) w,
\ea
where $h(v):=(v-1)+ (p-p_-)$ is convex and vanishing at $v_*$, $1$,
hence negative for $v\in (v_*,1)$.
Setting $w\equiv 0$, we find that there is a monotone decreasing
solution connecting $v_-=1$ to $v_*$, which has the type described
by the results of Proposition \ref{parprop}.

As $h<0$, the nullclines $w=\pm \sqrt{-v^2h(v)/J}$ for $v'$ are
well-defined for $v\in (v_*,1)$, bounding a lens-shaped set
$\CalR$
between $v_*$ and $1$, passing through the saddle-type
rest points at $v=K$ and pinching to a single point at $v=1,v_*$,
within which $v'<0$.
Noting that $\sgn w'=\pm \sgn w$ for $v\gtrless K$, we find
that this region is invariant in backward (resp. forward) $x$,
whence, starting at the saddles and integrating in backward (resp. forward)
$x$ along the stable (resp. unstable) manifold, we find that the
orbit remains for all $x$ in $\CalR$, hence must connect to
$1$ (resp. $v_*$), verifying existence of the bounding Lax-type connections.
Starting at any point $(K,w)$ lying on the open interval between the
two saddles and integrating in both forward and backward $x$, we likewise
find that the orbits are all trapped in $\CalR$ for all $x$, so generate
a one-parameter family of overcompressive connections filling up $\CalR$.
See Figure \ref{paranull}.
\end{proof}

\begin{figure}[t]
\includegraphics[width=7.5cm]{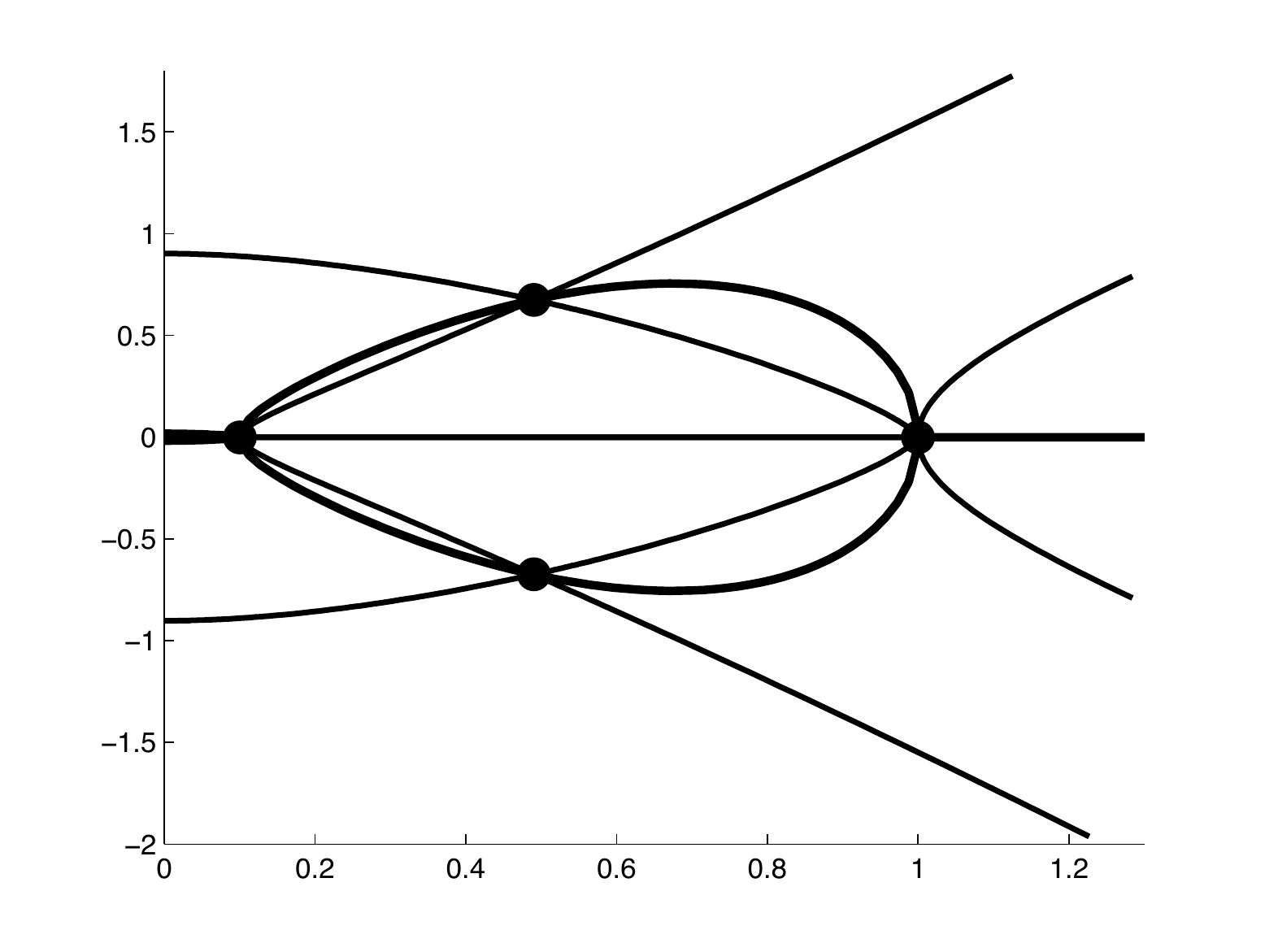}
\caption{Nullclines and phase portrait for typical parallel case,
parameters
 $v_+=0.1$, $I=0.7$, $B_+=0$, and $\mu=\tau=1$.}
\label{paranull}
\end{figure}

\subsection{Existence of Lax-type profiles, $J>0$}\label{s:Laxexist}

\begin{proposition}\label{Laxexist}
In the nonparallel case $J>0$, for $a>0$ and $0\le K\ne 1$, 
with $\sigma=\infty$,
rest points $v_i<v_j$ of traveling-wave ODE \eqref{redode}
lying on the same side of $K$ always admit a Lax-type profile,
which, moreover, is monotone in both $\hat v$ and $\hat w$.
\end{proposition}

\begin{proof}
Without loss of generality, let the rest points
be $v_+<1$ and $v_-=1$. Rewriting \eqref{redode} as
\ba\label{rredode}
(2\mu+\eta) v'&=v\tilde h(v) +\frac{ (B_-+Iw)^2 }{2\mu_0 v} ,\\
\mu w'&=(v-K) w-\frac{IB_-(1-v)}{\mu_0},
\ea
where $\tilde h(v):= p(v)-p_- + v-1 -J$ is convex and
(since $v\tilde h(v)+ \frac{ (B_-+Iw)^2 }{2\mu_0 v} =0$)
negative at $v=v_+,1$, hence negative on $(v_+,1)$,
we find that the nullclines 
$Iw= -B_- \pm \sqrt{-2\mu_0 v^2h(v)}$ for $v'$ 
for $v'$ are well-defined for $v\in (v_+,1)$.
Likewise, the nullclines $w=\frac{IB_-(1-v)}{\mu_0(v-K)}$
for $w'$ are well-defined for $v$ on either side of $K$,
forming two disconnected branches asymptotic to the
line $v=K$.


{\it Case $K<v_+$.} In this case we find that the rest points
$v_+$, $1$ must lie on the intersection of the lower branch of the nullcline
$v'=0$, and the righthand ($>K$) branch of the nullcline $w'=0$,
and these nullcline branches have no other intersection (else there
would be a third rest point for $v>K$, impossible by Proposition 
\ref{globalprop}).
Looking at asymptotics, we find that the nullcline $w'=0$ must lie above 
the nullcline $v'=0$ for $v\in t(v_+,1)$, with the two curves forming
a lens-shaped region $\tilde \CalR$ between $v_+$ and $1$, within which
$v'<0$ and $w'<0$.
Looking along the boundaries, we find that the vector field
$(v',w')$ points out of $\tilde \CalR$, so that $\tilde \CalR$ is
invariant in backwards $x$.
Thus, integrating backward in $x$ from $v=v_+$ along the stable manifold,
we find that there exists a connection to $v=v_-=1$, which
is monotone decreasing in $\hat v$ and $\hat w$.

{\it Case $K>1$.} A symmetric argument yields existence in case $K>1$,
again with $\hat v$ monotone decreasing and $\hat w$ monotone
increasing, this time via invariance in forward $x$.  
See Figure \ref{nullclines}.
\end{proof}

\begin{figure}[t]
\begin{center}
$\begin{array}{lr}
\includegraphics[width=7.5cm]{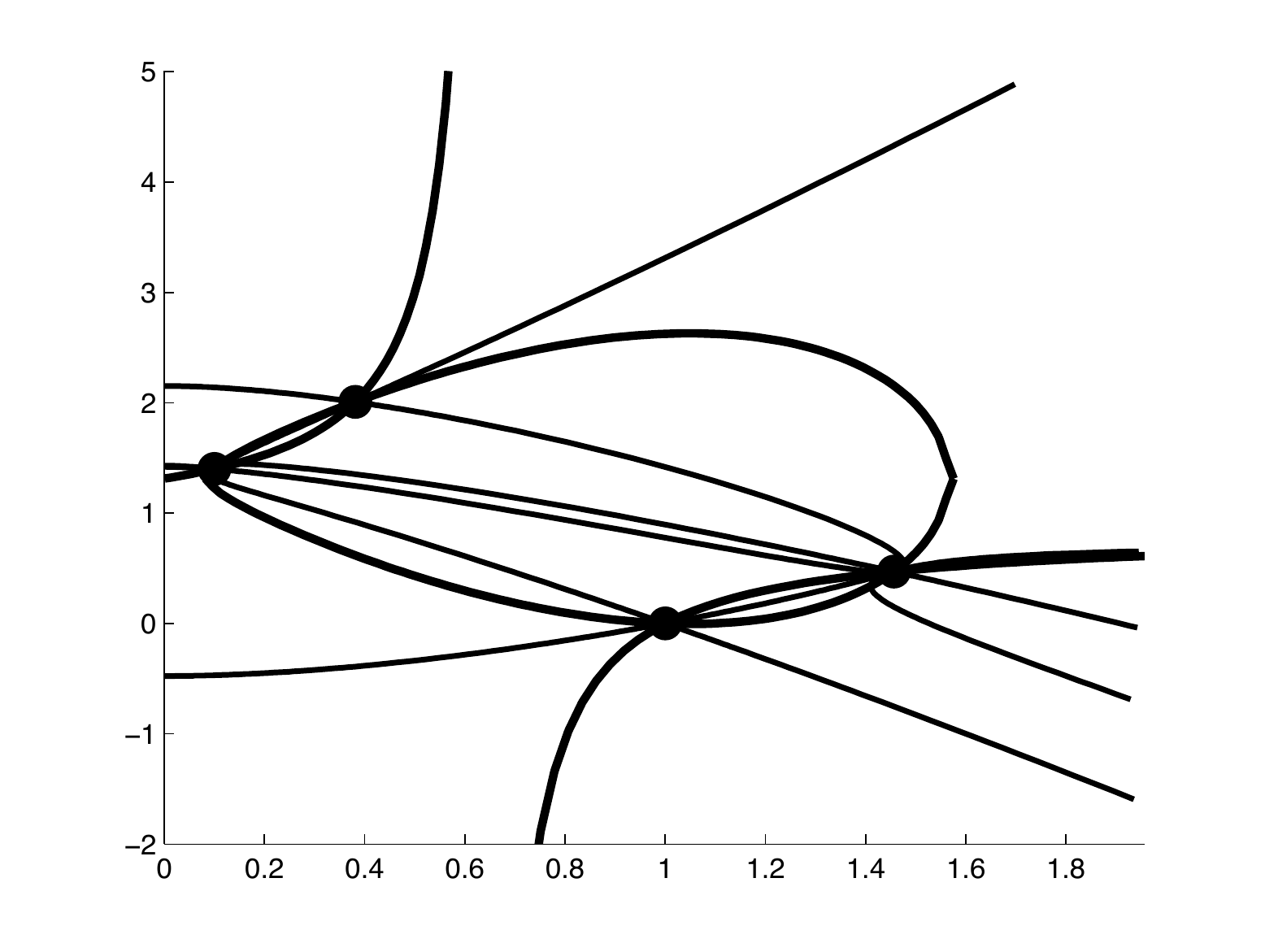}(a) & \includegraphics[width=7.5cm]{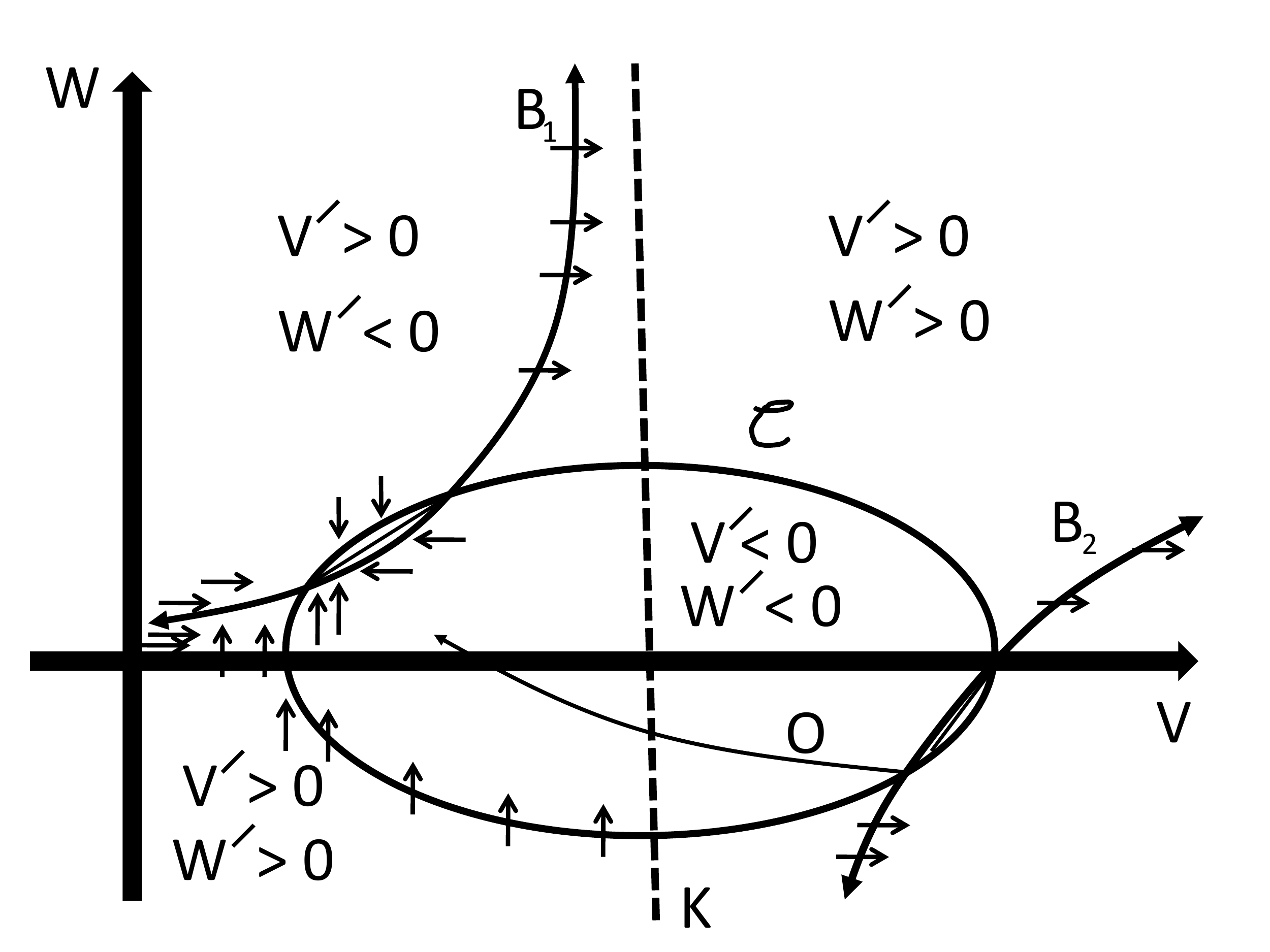}(b)
\end{array}$
\end{center}
\caption{Nullclines and phase portrait for typical
four rest point configuration, nonparallel case;
parameters $v_+=0.1$, $I=0.8$, $B_+=0.7$ and $\mu=\tau=1$.}
\label{nullclines}
\end{figure}

\br\label{Gilbargrmk}
\textup{
The argument above may be recognized as the same one
used to prove existence of nonisentropic gas-dynamical
profiles in \cite{Gi}.
It should be possible to obtain this result alternatively
by a relative entropy argument as in \cite{G} for the nonisentropic
case, showing in case $K<v_+$ that the level set
of $\phi$ through $v_+$ encloses $v_-=1$, yielding existence by
a Lyapunov-function argument in backward $x$; this would apply
also for $\sigma$ finite.
}
\er

\subsection{Existence of intermediate shock profiles, $J>0$} \label{s:sing}

\begin{proposition}\label{intexist}
Set $r:=\mu/\tau$.
In the nonparallel case $J>0$, with $\sigma=\infty$,
for each fixed $(a,J,K)$ with $a>0$, $J>0$, and $0\le K\ne 1$ for which
there exist four rest points $v_1<v_2<K<v_3<v_4$ of 
traveling-wave ODE \eqref{redode},
there exists a value $r_*=r_*(J,K,a)>0$ such that: (i) for $r<r_*$, there exist
no intermediate shock profiles (i.e., the only connections
are regular Lax profiles between $v_2$ and $v_1$ and $v_4$ and
$v_3$ as described in Proposition \ref{Laxexist});
(ii) for $r=r_*$, there exists an undercompressive profile connecting
$v_3$ to $v_2$, monotone decreasing in $v$ and increasing in $w$,
and no other intermediate shock profiles;
(iii) for $r>r_*$, there exist intermediate Lax connections from $v_3$ to $v_1$
and $v_4$ to $v_2$, in general not monotone in $v$ or
else not montone in $w$, and a one-parameter family of overcompressive
profiles from $v_4$ to $v_1$, in general not monotone in 
$v$ or $w$, 
with no other intermediate shock profiles.
\end{proposition}

\begin{proof}
Referring to Figure \ref{nullclines}(b), rewrite \eqref{redode}
again as
\ba\label{rrredode}
\tau v'&=v\tilde h(v) +\frac{ (B_-+Iw)^2 }{2\mu_0 v} ,\\
\mu w'&=(v-K) w-\frac{IB_-(1-v)}{\mu_0},
\ea
$\tau=2\mu+\eta$,
where (see proof of Proposition \ref{Laxexist}) 
$\tilde h(v):= p(v)-p_- + v-1 -J$ is convex,
negative on $(v_1,v_4)$, and goes to $+\infty$ as $v\to 0,+\infty$.
Denote by $\underline v<v_1$ and $\overline v>v_4$ the two
points at which $h$ vanishes.

The nullclines 
$$
Iw= -B_- \pm \sqrt{-2\mu_0 v^2h(v)}
$$
for $v'$ evidently are well-defined on $v\in (\underline v,\overline v)$,
together forming a simple closed curve $\CalC$
enclosing a region on which $v'<0$, as seen in Figure \ref{nullclines}(b).
Likewise, the nullclines $w=\frac{IB_-(1-v)}{\mu_0(v-K)}$
for $w'$ are well-defined for $v$ on either side of $K$,
forming two disconnected branches $\CalB_1$ and $\CalB_2$ asymptotic to the
line $v=K$, Figure \ref{nullclines}(b).

The arc formed by the portion of $\CalC$ from the rest point
at $v_3$ to $\underline v$
together with the portion of the $v$ axis from $\underline v$ to $0$,
the portion of the $w$ axis from $0$ to the intersection of the
$w$ axis with $\CalB_1$, 
and the portion of $\CalB_1$ between 
the intersection of the $w$ axis with $\CalB_1$ and the rest point
at $v_1$, the Lax connection between the rest points at $v_1$
and $v_2$, and the portion of $\CalB_1$ between the rest point at $v_2$
out to $w\to +\infty$ form a barrier to the flow in forward $x$,
through which an orbit initiating inside $\CalC$ cannot cross,
as, likewise, does the arc formed by the Lax shock between the
rest points at $v_1$ and $v_2$ together with the portion of $\CalB_2$
extending from the rest point at $v_2$ to $w\to -\infty$.

Thus, the orbit $\CalO$ initiating along the unstable manifold of the rest
point at $v_3$ pointing in decreasing $v$-$w$ directions, and thus
initially lying inside $\CalC$, must either (a) strike the arc between
$(\underline v,0)$, after which, being trapped between $\CalB_1$
and $\CalC$, it must asymptotically approach the rest point at $v_1$;
(b) strike the arc of $\CalB_1$ between the rest points at $v_1$ and
$v_2$, after which, being trapped between this arc and the portion
of $\CalC$ between the rest points at $v_1$ and $v_2$, it must
asymptotically approach the rest point at $v_1$; (c) remain within
the interior of $\CalC$ and to the right of $\CalB_1$ for all time,
asymptotically approaching the rest point at $v_2$; or, (d)
exit the interior of $\CalC$ along the arc between the rest point
at $v_2$ and the rest point at $v_1$, after which it remains trapped
outside of $\CalC$ with $v$ increasing monotonically to $ +\infty$.

Depending whether the orbit approaches the rest point at $v_1$,
approaches the rest point at $v_2$, or takes $v\to +\infty$,
we are in cases (iii), (ii), or (i) of the proposition.
But, these cases are distinguished by the
location along the arc $\CalC'$ formed by the portion of the
upper branch of $\CalC$ lying below $\CalB_1$ together with
the portion of $\CalB_1$ lying below $\CalC$ at which the orbit
$\CalO$ exits the part of the interior of $\CalC$ lying below
$\CalB_1$, with the locations corresponding to different cases 
ordered in clockwise fashion along $\CalC'$.
Noting that the signs of $v'$ and $w'$ are constant 
while $\CalO$; remains inside $\CalC'$ (recall that it is trapped to the right
of $\CalB_1$), and are given respectively by $\tau^{-1}$
and $\mu^{-1}$ times the righthand sides in \eqref{rrredode},
we find that the exit point moves strictly clockwise along $\CalC'$
monotonically as $r= \tau^{-1}/\mu^{-1}=\mu/\tau$ increases.

Thus, as asserted, there is a unique value $r=r_*$ for which
$\CalO$ exits at the rest point at $v_2$, corresponding to an
undercompressive connection. For $r<r_*$, $\CalO$ exits to the
right of the rest point at $v_2$, going off to infinity, and for
$r>r_*$, $\CalO$ exits to the left of the rest point at $v_2$,
asymptotically approaching the rest point at $v_1$, corresponding
to an intermediate Lax connection and case (iii).

Note, in case (iii), that the existence of this intermediate
Lax connection means that, applying a symmetric argument in
backward $x$ to the orbit originating from the saddle rest point at $v_2$,
we find that it remains trapped within $\CalC$ for all negative
$x$, approaching asymptotically as $x\to-\infty$ the rest point at $v_1$
The four Lax connections enclose an invariant region, within which all
orbits must be overcompressive profiles connecting the rest points
at $v_1$ and $v_4$.
It is clear that in general the intermediate
Lax profile may leave either the interior of $\CalC$
or the region below $\CalB_1$, hence
may be nonmonotone in $v$ or $w$ but not both.
Similarly, we find that the members of the
family of overcompressive profiles
are in general nonmonotone in $v$ or $w$ (and sometimes both).

In case (ii), or case (c) above, the profile remains for all $x$ in
a region for which $v'<0$ and $w'>0$, hence the profile is monotone
decreasing in $v$ and increasing in $w$. 
This completes the description of the phase portrait in cases (iii)
and (ii), finishing the proof.
\end{proof}

\br\label{nonun}
\textup{
Evidently, there is nonuniformity in the behavior of $r_*$,
in view of fact that the parallel case $J=0$ 
is always in case (iii) of Proposition \ref{intexist},
by the result of Proposition \ref{parprop}.  
That is, $r_*\to \infty$ as $J\to 0$ for fixed $K$ and $a$.
}
\er

\subsubsection{Singular perturbation analysis}
The results of Proposition \ref{intexist} may be illuminated
somewhat by formal singular perturbation analyses
as $r\to 0$ and $r\to +\infty$: equivalently, taking $\mu\to 0$
with $\tau=1$ fixed, or $\tau\to 0$ with $\mu=1$ fixed.
%
In the limit as $\mu\to 0$, the phase portrait for $J> 0$ reduces
to slow flow along the $w'=0$ nullcline $\CalB$ (notation
of the proof above),
with fast flow involving jumps in the vertical $w$ direction.  
We find that regular Lax connections are accomplished by slow flow
along $\CalC$, but there are no further intermediate shock connections
since the branches $\CalB_1$ and $\CalB_2$ of $\CalB$ are separated
by the vertical line $v=K$.  See Figure \ref{phaseS}(b).
In the special case $J=0$, the hyperbolae $\CalB_j$ degenerate to
the connected union of $v=K$ and $w=0$, allowing intermediate
connections both from the rest point at $v_4$ to the rest points
at $v_2=v_3$ and from the rest point at $v_4$ to the rest point at $v_1$.

In the limit as $\tau\to 0$, the phase portrait reduces
to slow flow along the $v'=0$ nullcline $\CalC$ (notation
of the proof above),
with fast flow involving horizontal jumps in $v$.  
We find that Lax and intermediate Lax connections 
may all be accomplished by slow flow along $\CalC$,
with fast flow filling in the overcompressive family.
See Figure \ref{phaseS}(a).

Finally, note that as $r$ goes from $0^+$ ($\mu\to 0$ limit)
to $+\infty$ ($\tau\to 0$ limit) the relative orientation as
measured by a Melnikov separation function along an appropriate
transversal of the
unstable manifold pointing to the left at $r=0^+$
of the rest point associated with $v_3$ and the stable manifold
entering from the right at $r=0^+$ of the rest point associated with $v_2$
changes sign. 
In plain language, the former passes below and to the left of 
the latter for $r=0^+$ and above and to the right for $r\to +\infty$.
By the Intermediate Value Theorem and continuous dependence, therefore,
there exists at least one value $r_0$ for which they meet, i.e., there
exists an undercompressive profile from the rest point associated with
$v_3$ to the rest point associated with $v_2$.

\br\label{singrmk}
\textup{
The above, formal arguments, may be made rigorous as done in \cite{FS}
for the general nonisentropic case.
They give slightly less information in the planar case
(note that we lose the monotonicity/uniqueness of $r_*$ information
obtained by phase plane analysis) but have the advantage of applying also 
to more general, nonplanar situations.
}
\er

\subsubsection{The undercompressive bifurcation}
There is an interesting bifurcation as $r=\mu/\tau$ decreases,
between the situation of case (iii) in which there is a family
of overcompressive connections between the rest points at $v_1$
and $v_4$, bounded by Lax connections, and the situation of
case (i), in which there exist no intermediate shock connections.
This occurs at the point $r=r_*$ where an undercompressive connection
appears.
As illustrated in Figures \ref{phaseUC} and \ref{ucfig},
this occurs through squeezing of the infinite overcompressive
family to a single undercompressive--Lax profile pair, after
which, as $r$ is decreased past $r_*$, the undercompressive
connection breaks, leaving only the regular Lax connection
and no intermediate profiles remain.

Note that this occurs for the example in Figure
\ref{ucfig} for value $r_*=0.17$, substantially less
than the ``physical''
value predicted by \eqref{etaform} of $r=\mu/(2\mu+\eta)= 0.75$.
For the value $r=.75$ and $\gamma=5/3, 7/5$ (monatomic or
diatomic gas), we find numerically that undercompressive shocks
do not occur.

\subsubsection{Composite-wave limits}\label{composite}
In the limit as $r\to r_*^+$, the intermediate Lax shock connecting the
rest points at $v_3$ and $v_1$ approaches a ``doubly composite wave''
formed by an approximate superposition of the limiting undercompressive
profile between the rest points at $v_3$ and $v_2$ and the Lax
profile between rest points at $v_2$ and $v_1$ at value $r=r_*$,
separated by an interval of length going to infinity as $r\to r_*^+$
on which the solution is approximately equal to the value of the
saddle-type rest point at $v_2$ to which it passes nearby.
Likewise, as $r\to r_*^+$,
the family of intermediate overcompressive profiles 
connecting the rest points at $v_4$ and $v_1$ approaches a 
triply composite wave consisting of the approximate superposition of the
limiting Lax profile between rest points at $v_4$ and $v_3$,
undercompressive profile between rest points at $v_3$ and $v_2$,
and Lax profile between rest points at $v_2$ and $v_1$ at value $r=r_*$,
separated by intervals of length going to infinity on which the
solution stays near the saddle-type rest points at $v_3$ and $v_2$.

In either case,
because the resulting profiles require larger and larger intervals in
$x$ to converge to limits $U_\pm$, both the profiles and their associated
Evans functions are numerically impractical to compute, requiring larger
and larger computational domains, and must be handled separately
taking into account the underlying limiting structure.
We discuss this issue in Section \ref{s:compstab}

 \begin{figure}[htbp]
\begin{center}
$\begin{array}{lr}
\includegraphics[width=7.5cm]{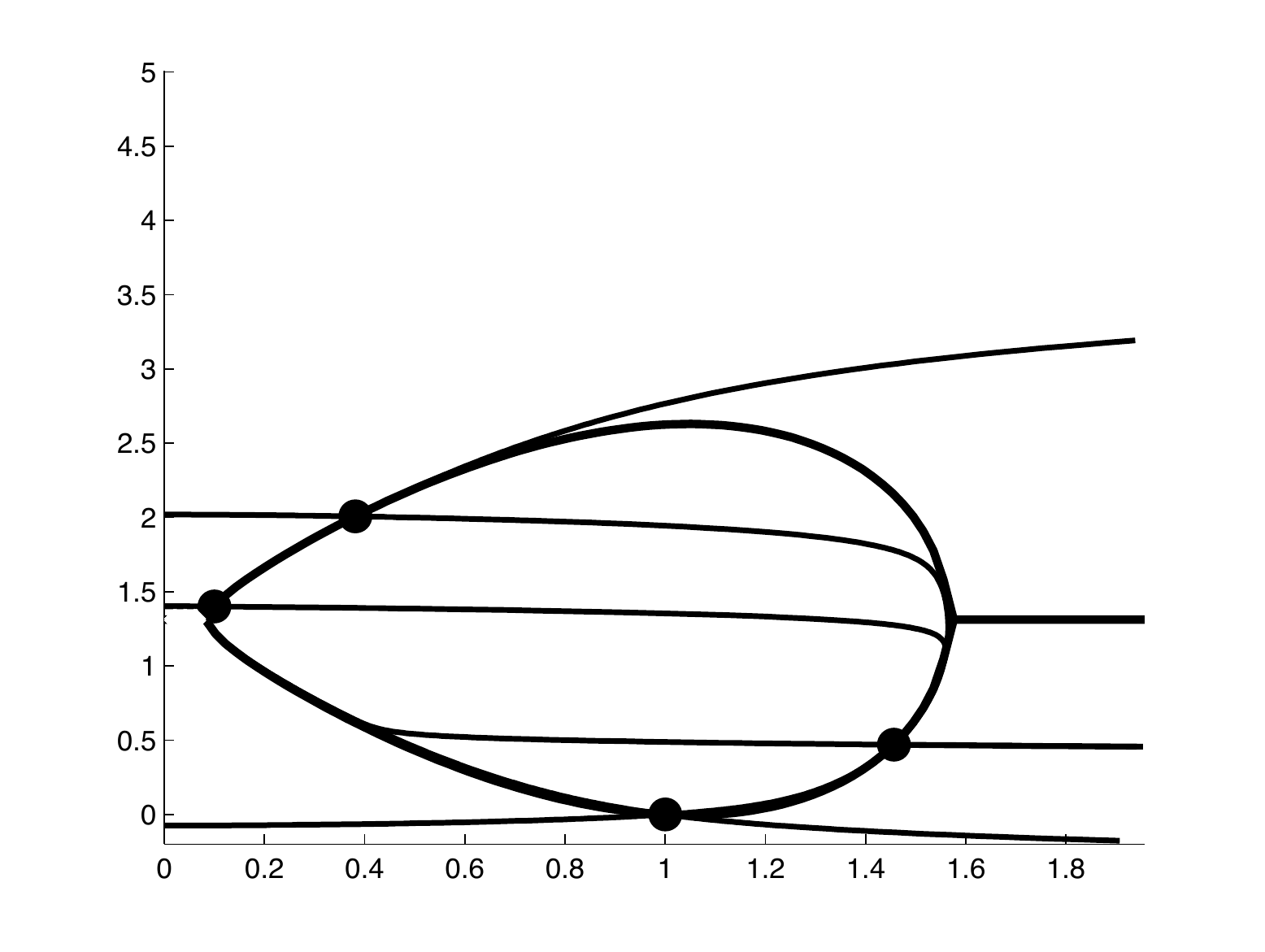}(a) & \includegraphics[width=7.5cm]{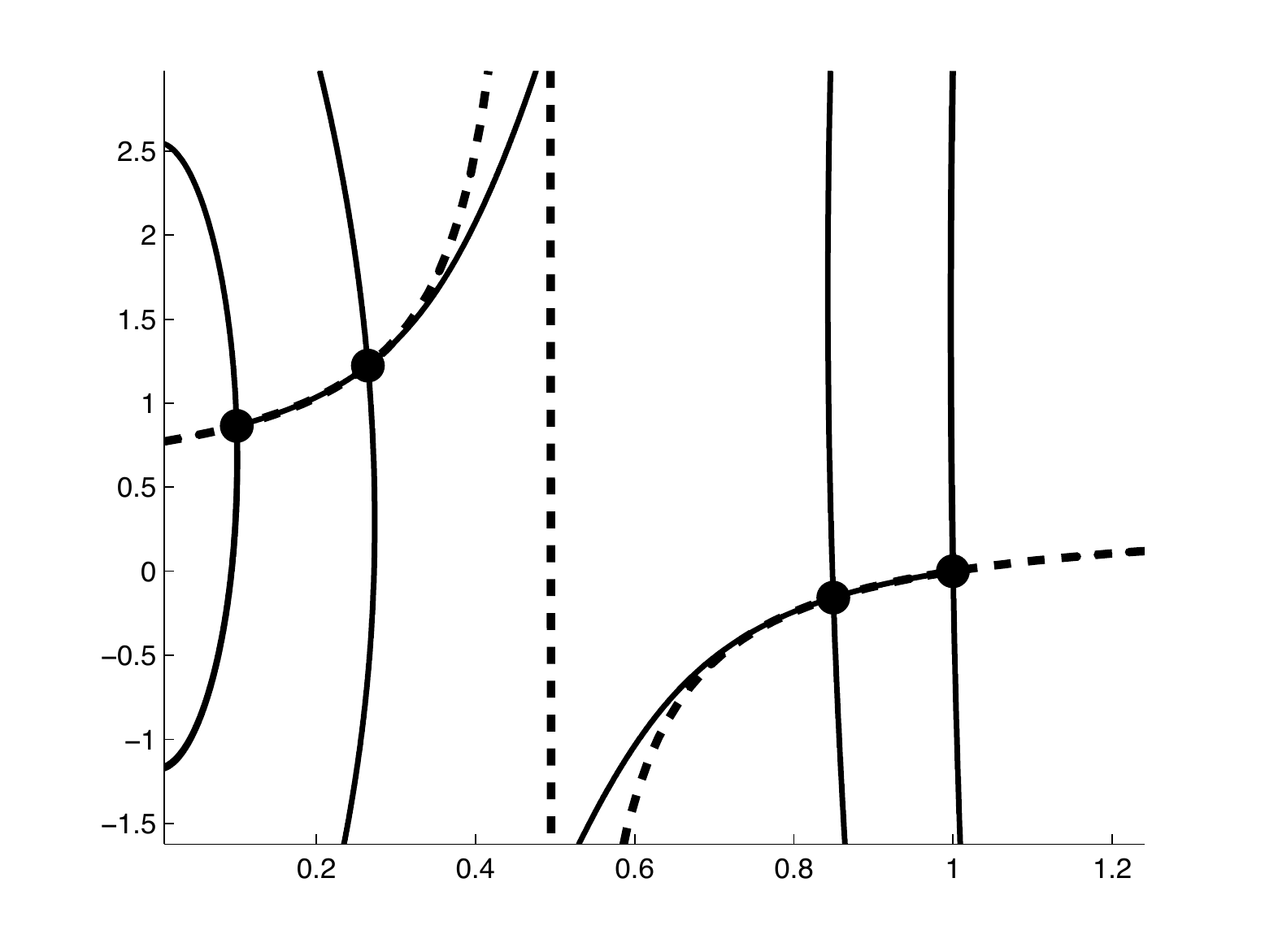}(b)
\end{array}$
\end{center}
	\caption{Phase portraits in singular limits.
Figure (a) $\mu=1$, $\tau=.1$,
Figure (b) $\mu=.005$, $\tau=1$;
 parameters $\gamma=5/3$, $v_+=0.1$, $I=0.7$, $B_{2+}=0.7$.}
\label{phaseS}
\end{figure}

 \begin{figure}[htbp]
\begin{center}
$\begin{array}{lr}
\includegraphics[width=7.5cm]{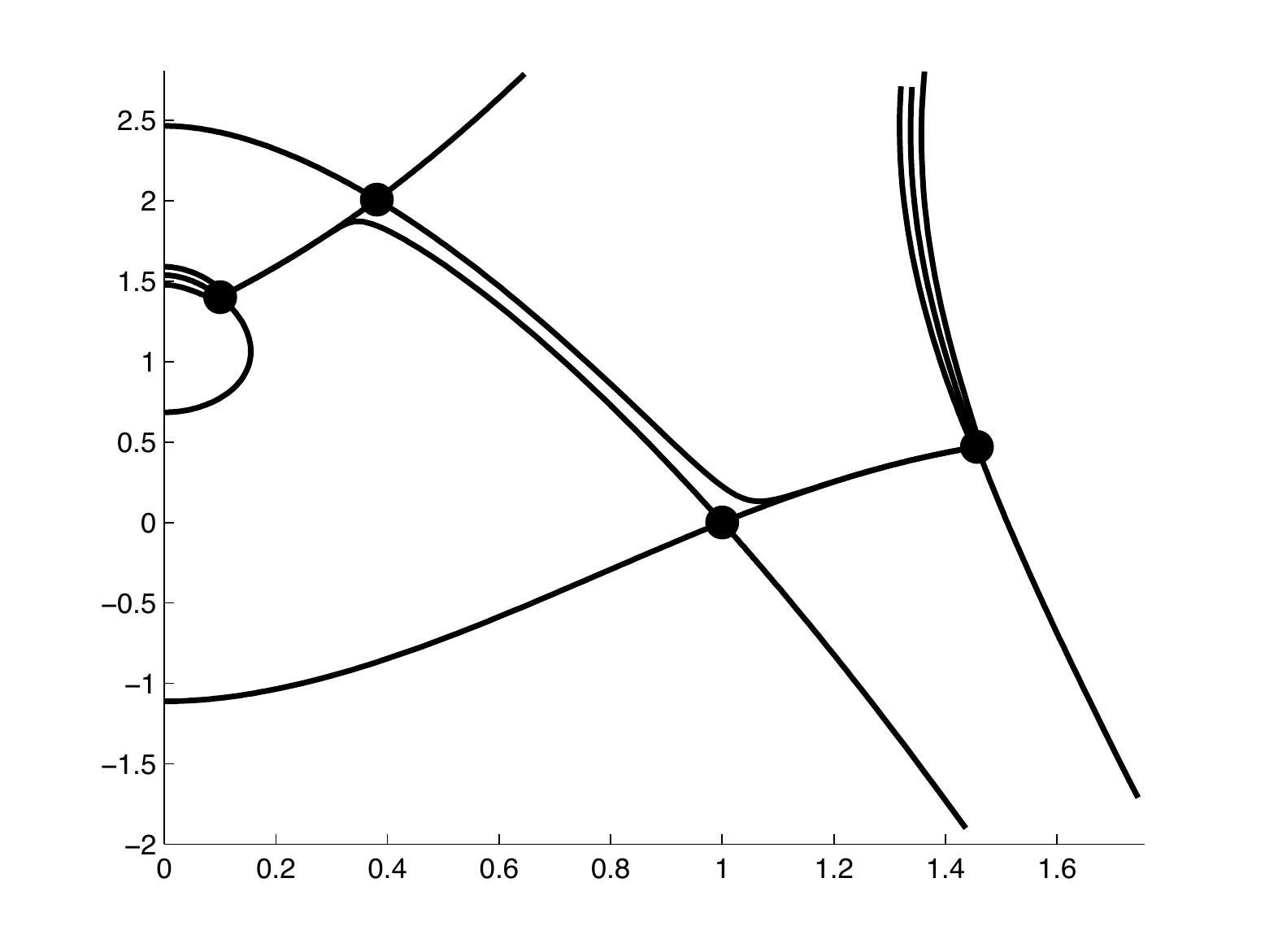}(a) & 
\includegraphics[width=7.5cm]{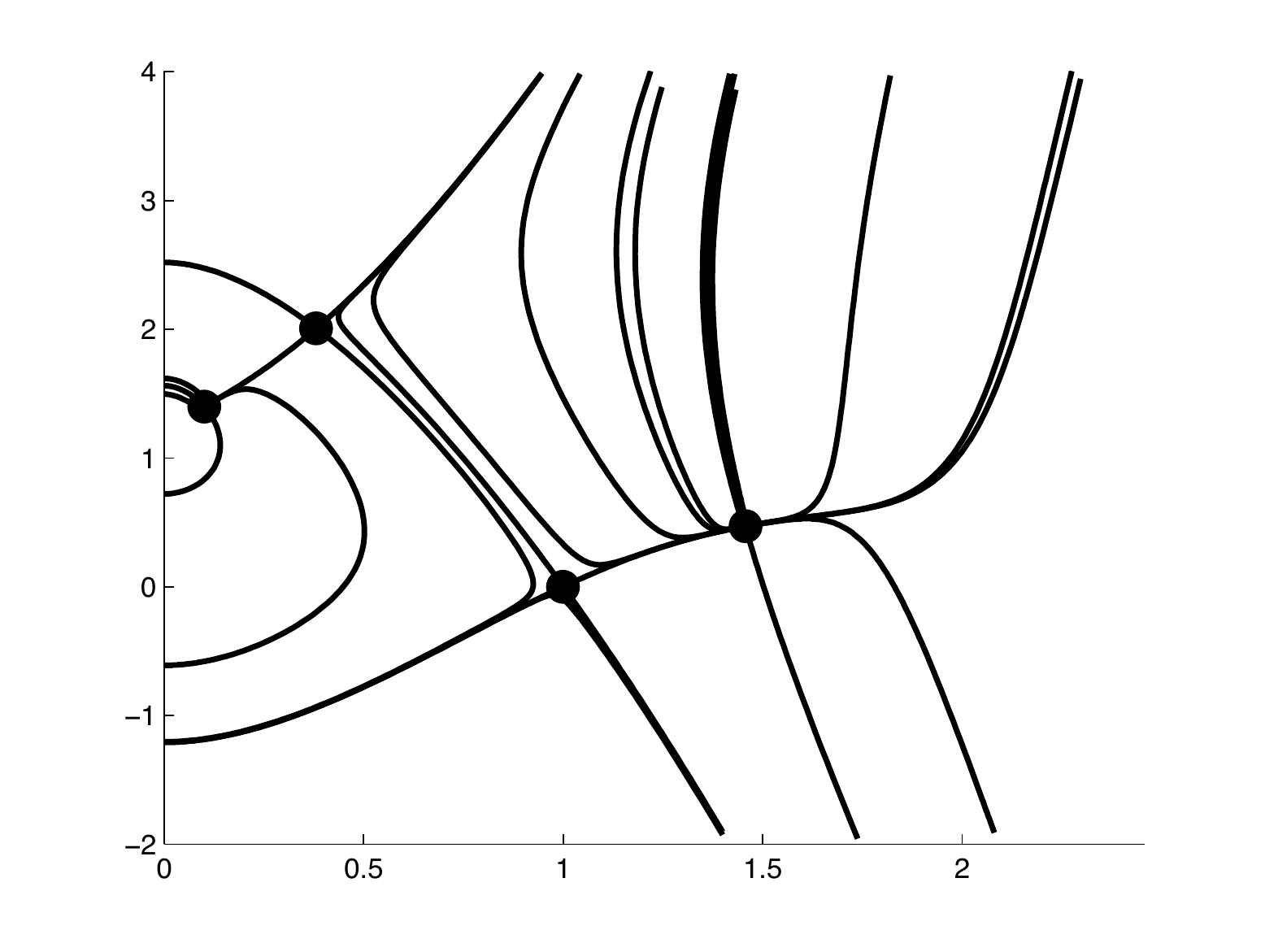}(b) 
\end{array}$
\end{center}
	\caption{Transition to nonexistence:
keeping $\tau=2\mu+\eta=1$ and letting 
$\mu \to 0$, we find that the overcompressive family is
squeezed more and more
but still connects until somewhere between 
$\mu=0.185$ (Figure (a)) and $\mu=0.16$ (Figure (b)).
At that point the flows switch sides so that 
neither overcompressive nor undercompressive connections exist;
 parameters $\gamma=5/3$, $v_+=0.1$, $I=0.7$, $B_{2+}=0.7$.
}
\label{phaseUC}
\end{figure}

\begin{figure}[t]
\includegraphics[width=7.5cm]{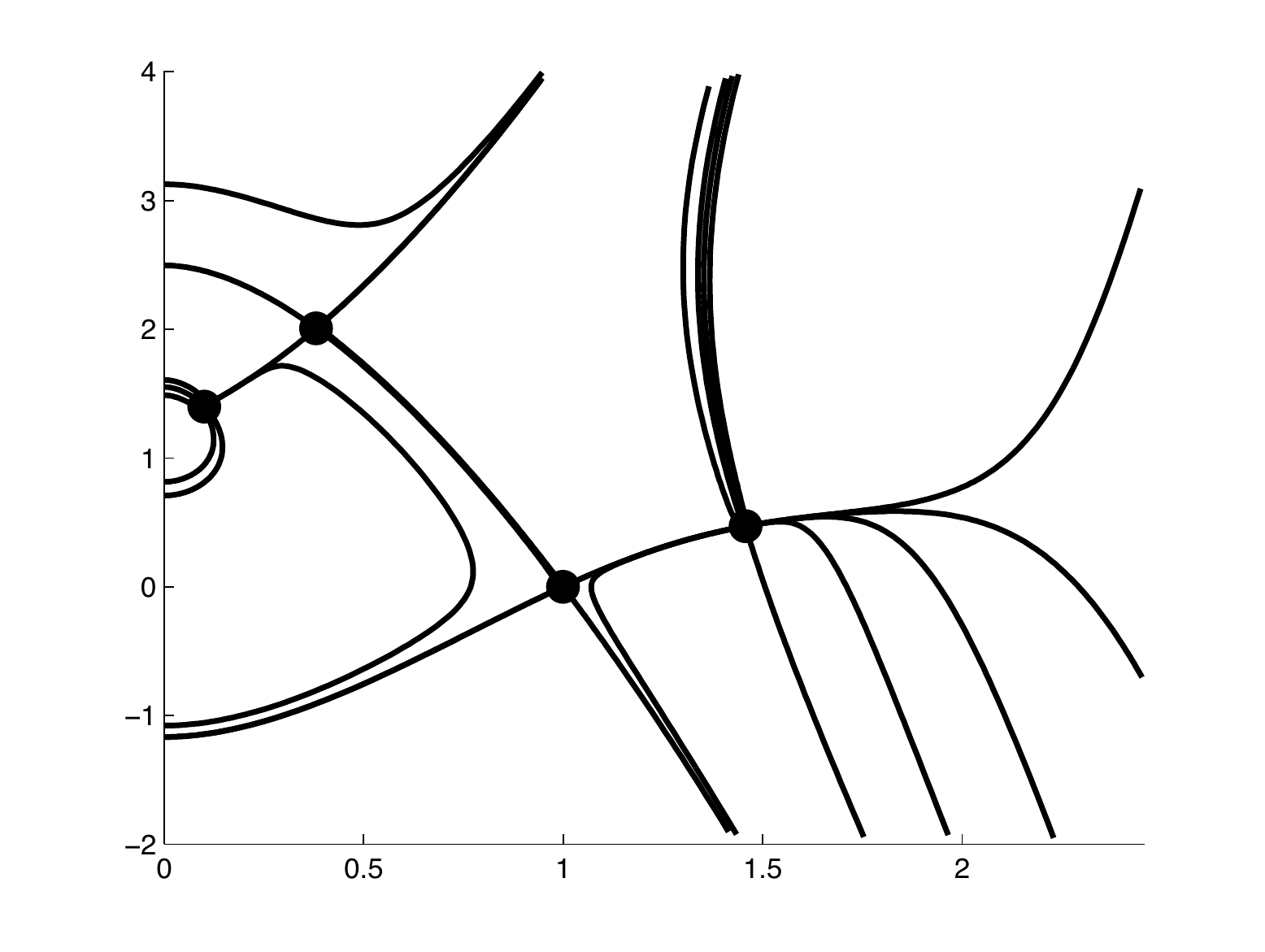}
\caption{Undercompressive connection: 
at the transition point $\approx \mu=0.17$,
an undercompressive profile appears;
parameters $\gamma=5/3$, $v_+=0.1$, $I=0.8$, $B_+=.7$.}
\label{ucfig}
\end{figure}

\begin{figure}[htbp]
\begin{center}
$\begin{array}{lr}
\includegraphics[width=7.5cm]{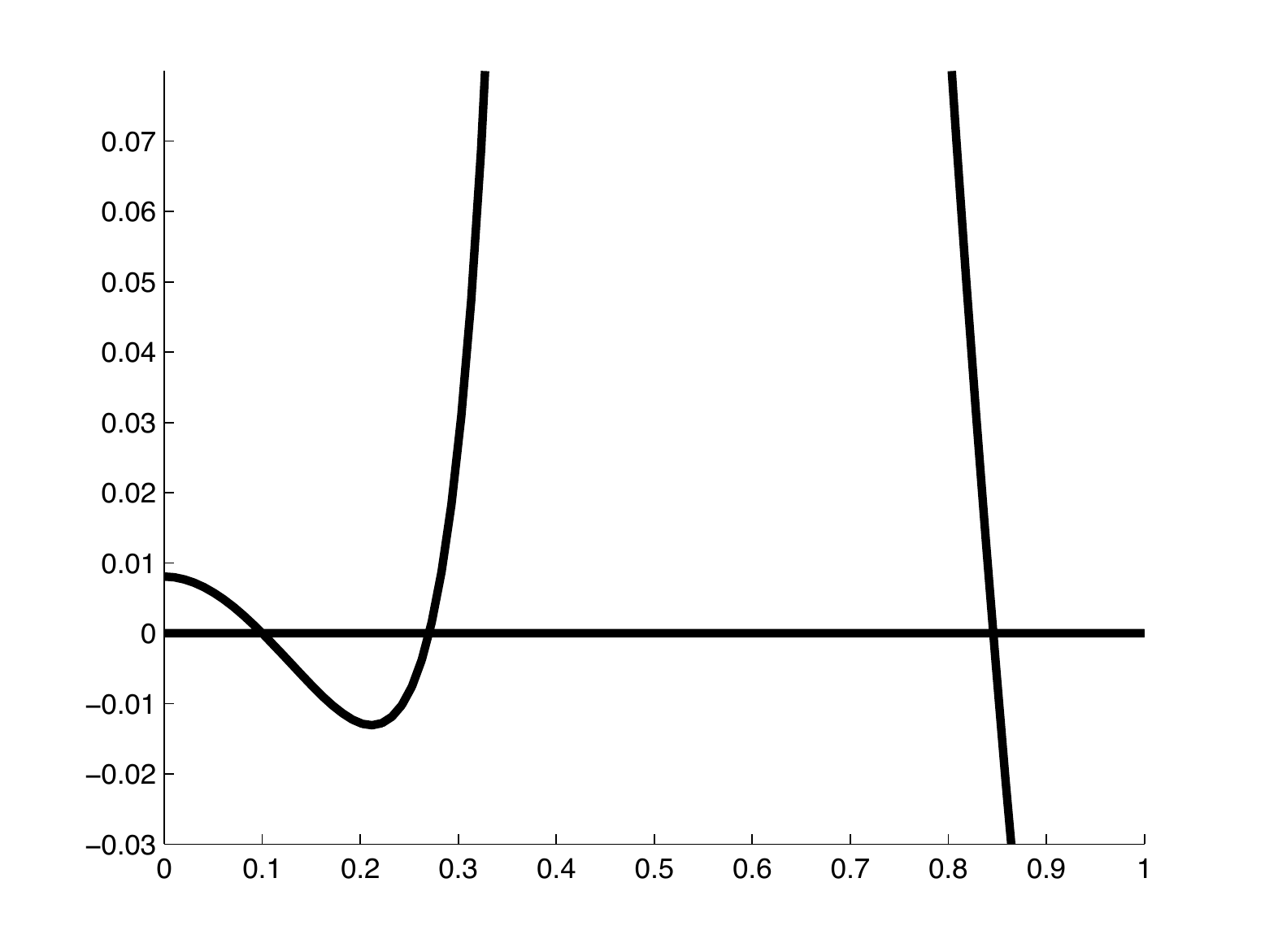}(a) & \includegraphics[width=7.5cm]{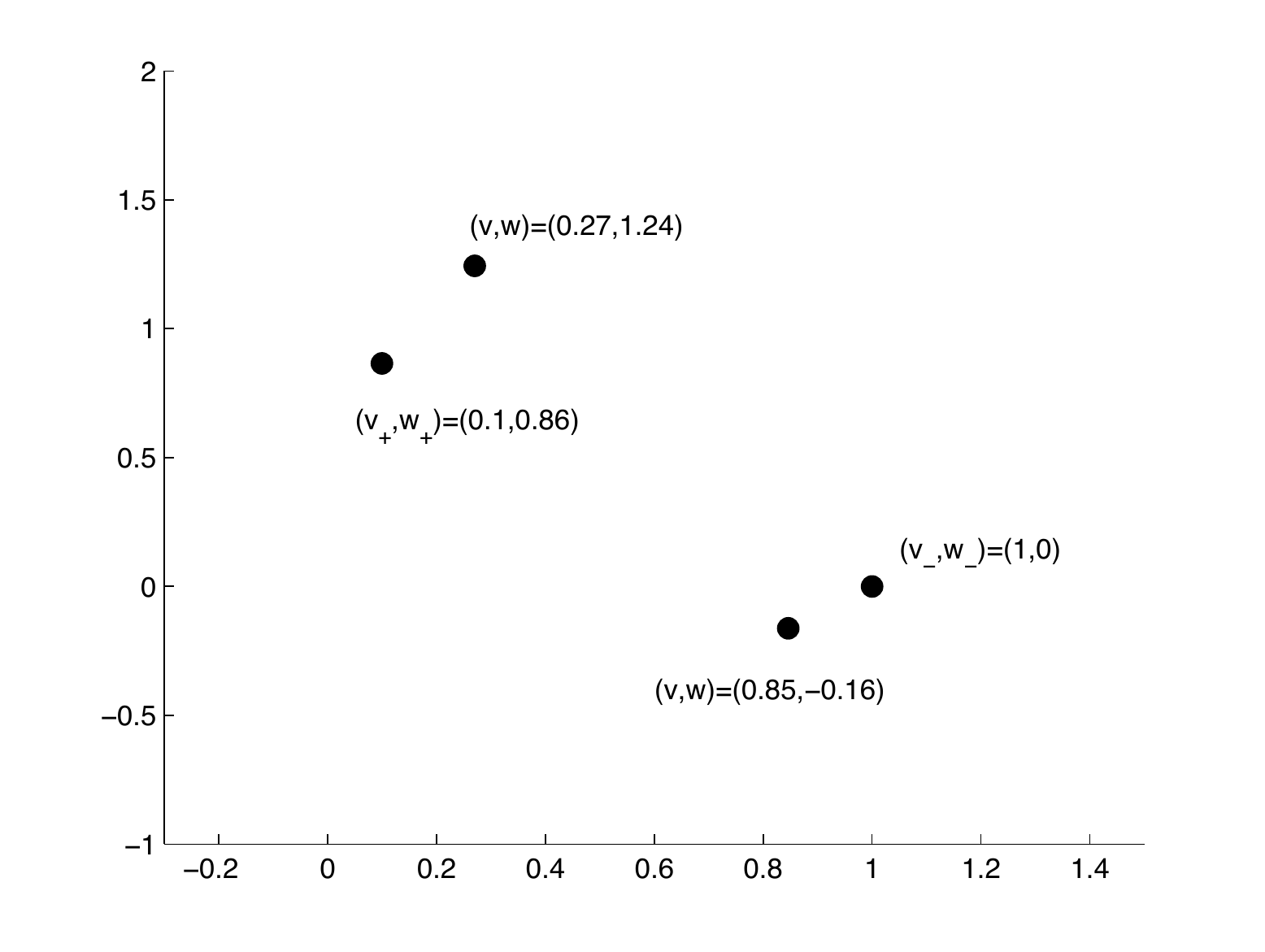}(b)\\
\includegraphics[width=7.5cm]{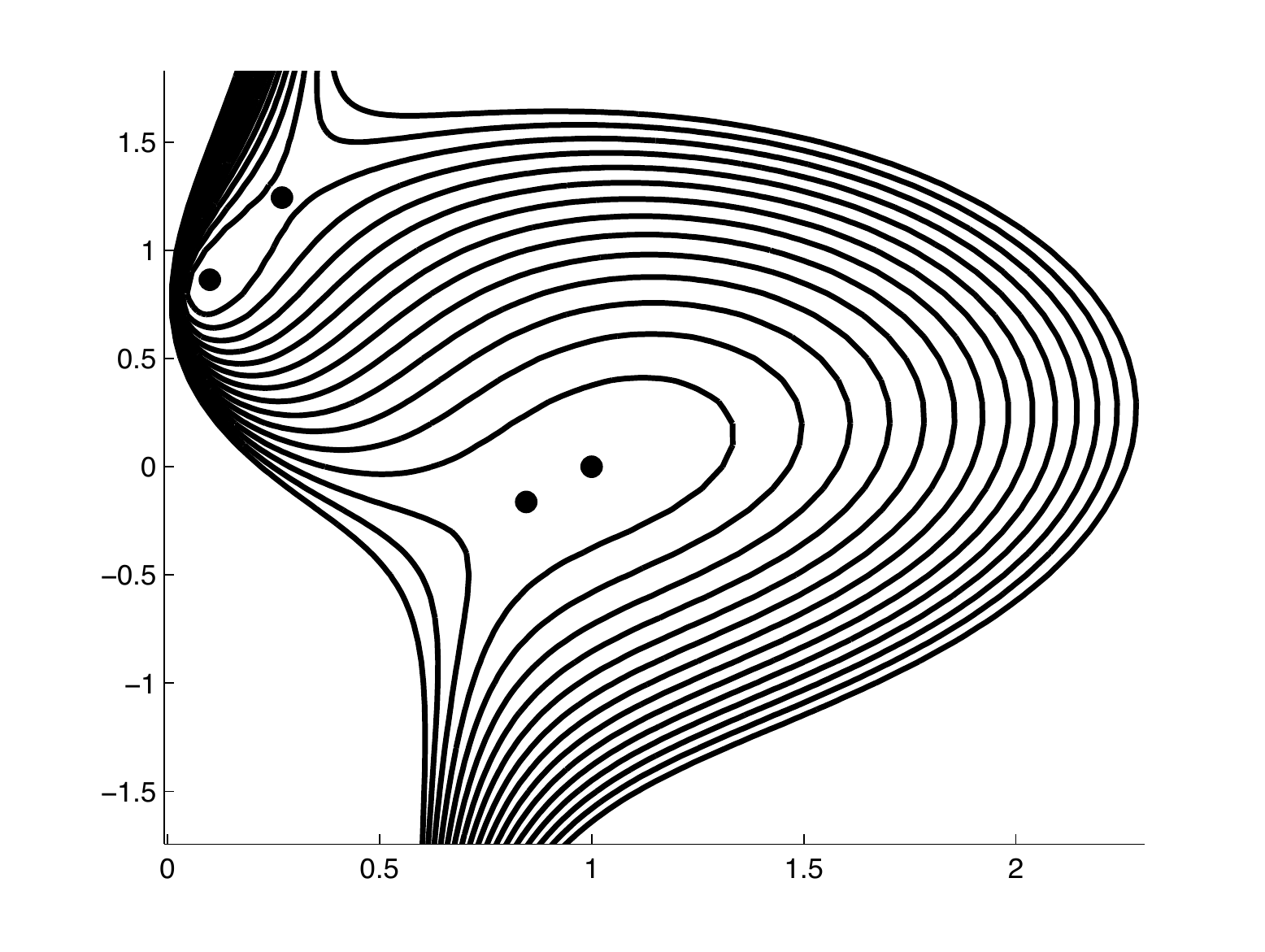}(c) & \includegraphics[width=7.5cm]{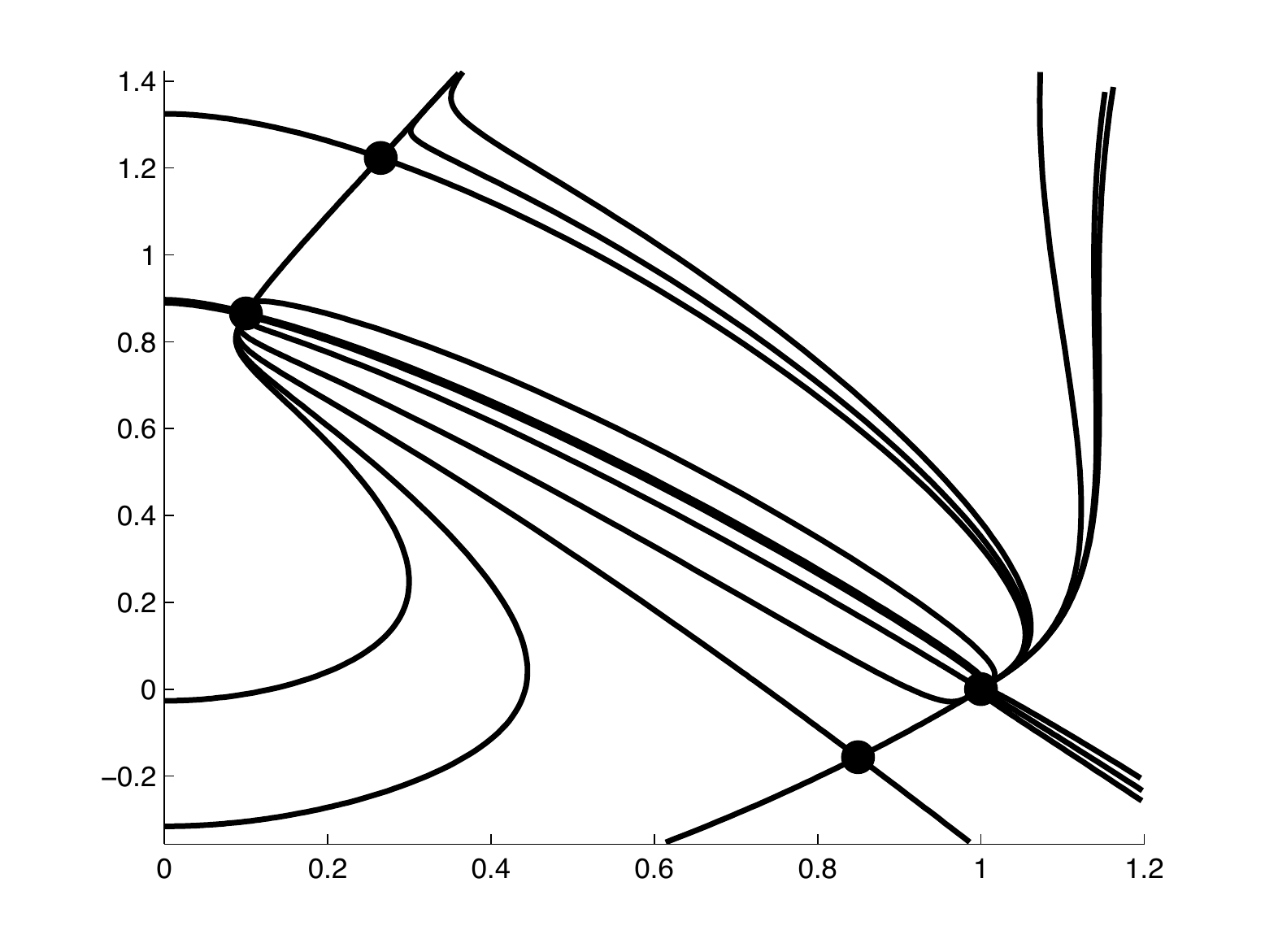}(d)
\end{array}$
\end{center}
\caption{Typical phase portrait for MHD with two variables and infinite electric resistivity ($\sigma=\infty$). Parameter values are $\gamma=2$, $v_+=0.1$, $I=0.7$, $B_{2+}=0.7$, and $\mu_0=1$. In Figure (a) we graph $f(v)$ given by \eqref{f}.  The Rankine--Hugoniot solutions corresponding to the roots in Figure (a) are given in Figure (b). In Figure (c) we plot level sets of $\check \phi(v,w)$ given 
by \eqref{sigphi} and in Figure (d) we draw the phase portrait.}
\label{phaseA}
\end{figure}

\begin{figure}[htbp]
\begin{center}
$\begin{array}{lr}
\includegraphics[width=7.5cm]{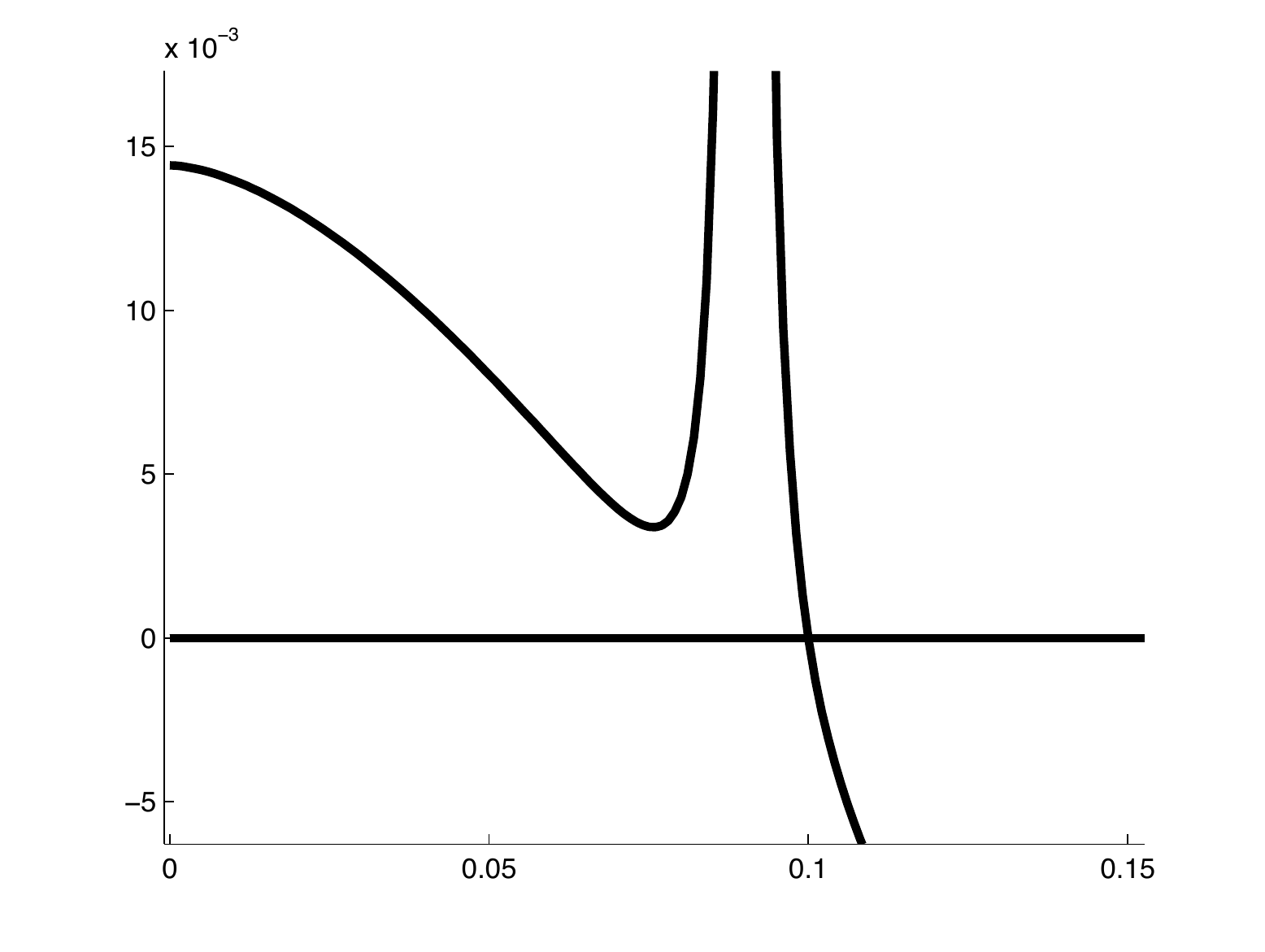}(a) & \includegraphics[width=7.5cm]{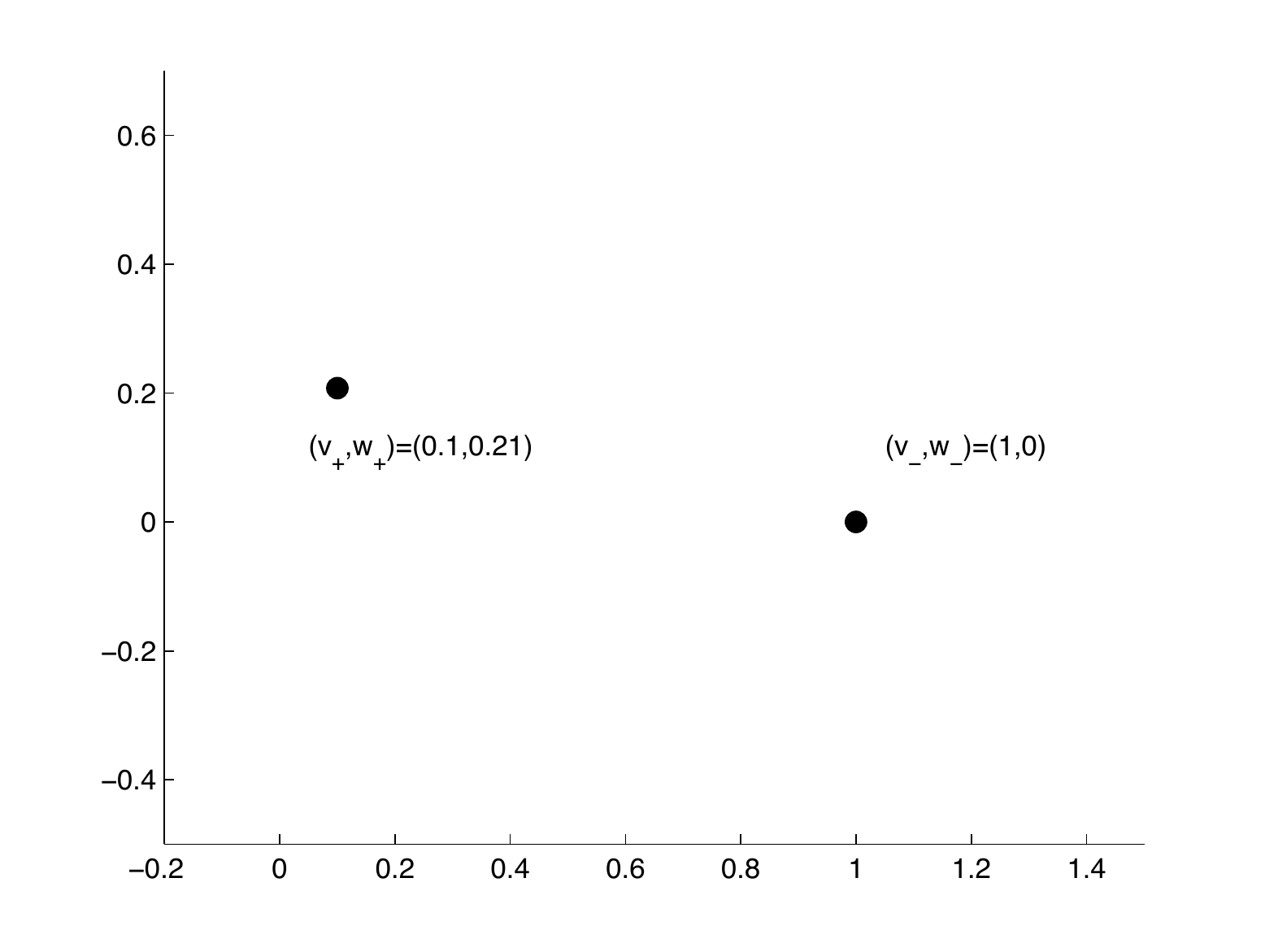}(b)\\
\includegraphics[width=7.5cm]{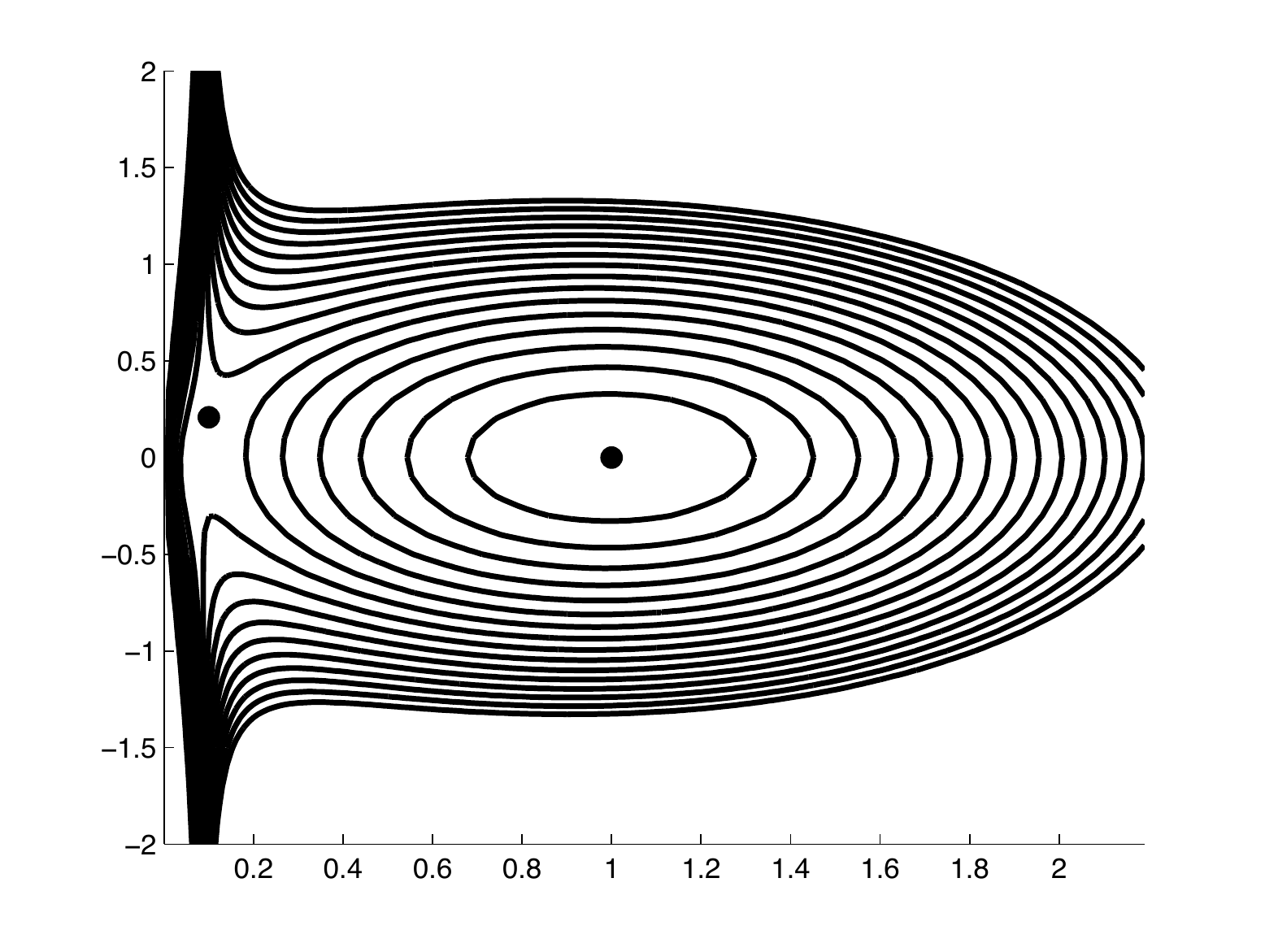}(c) & \includegraphics[width=7.5cm]{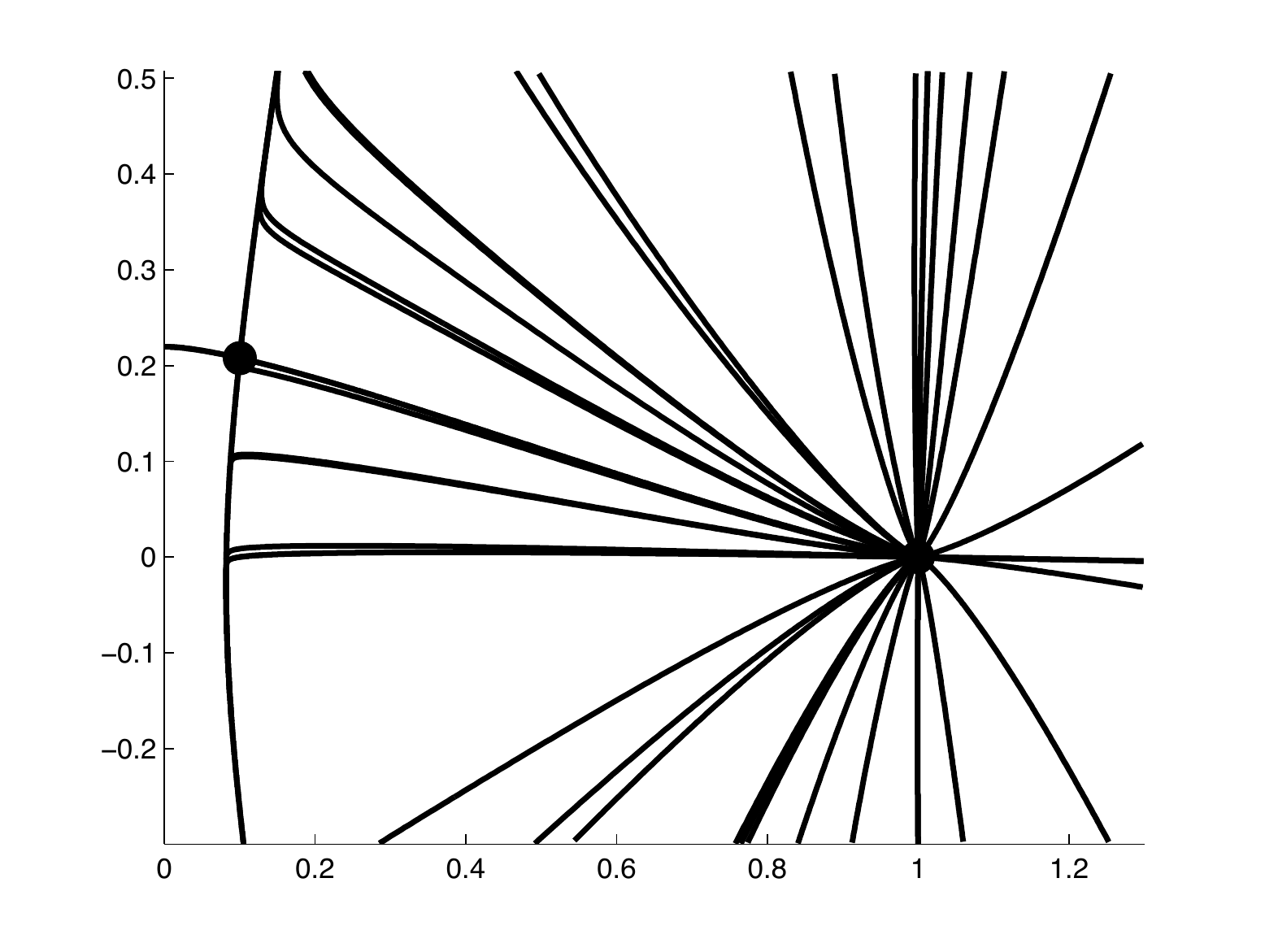}(d)
\end{array}$
\end{center}
\caption{Typical phase portrait for MHD with two variables and infinite electric resistivity ($\sigma=\infty$). Parameter values are $\gamma=5/3$, $v_+=0.1$, $I=0.3$, $B_{2+}=0.7$, and $\mu_0=1$. In Figure (a) we graph $f(v)$ given by \eqref{f}.  The Rankine--Hugoniot solutions corresponding to the roots in Figure (a) are given in Figure (b). In Figure (c) we plot level sets of
$\check \phi(v,w)$ given by \eqref{sigphi} and in Figure (d) 
we draw the phase portrait.}
\label{phaseB}
\end{figure}


\begin{figure}[htbp]
\begin{center}
$\begin{array}{lr}
\includegraphics[width=7.5cm]{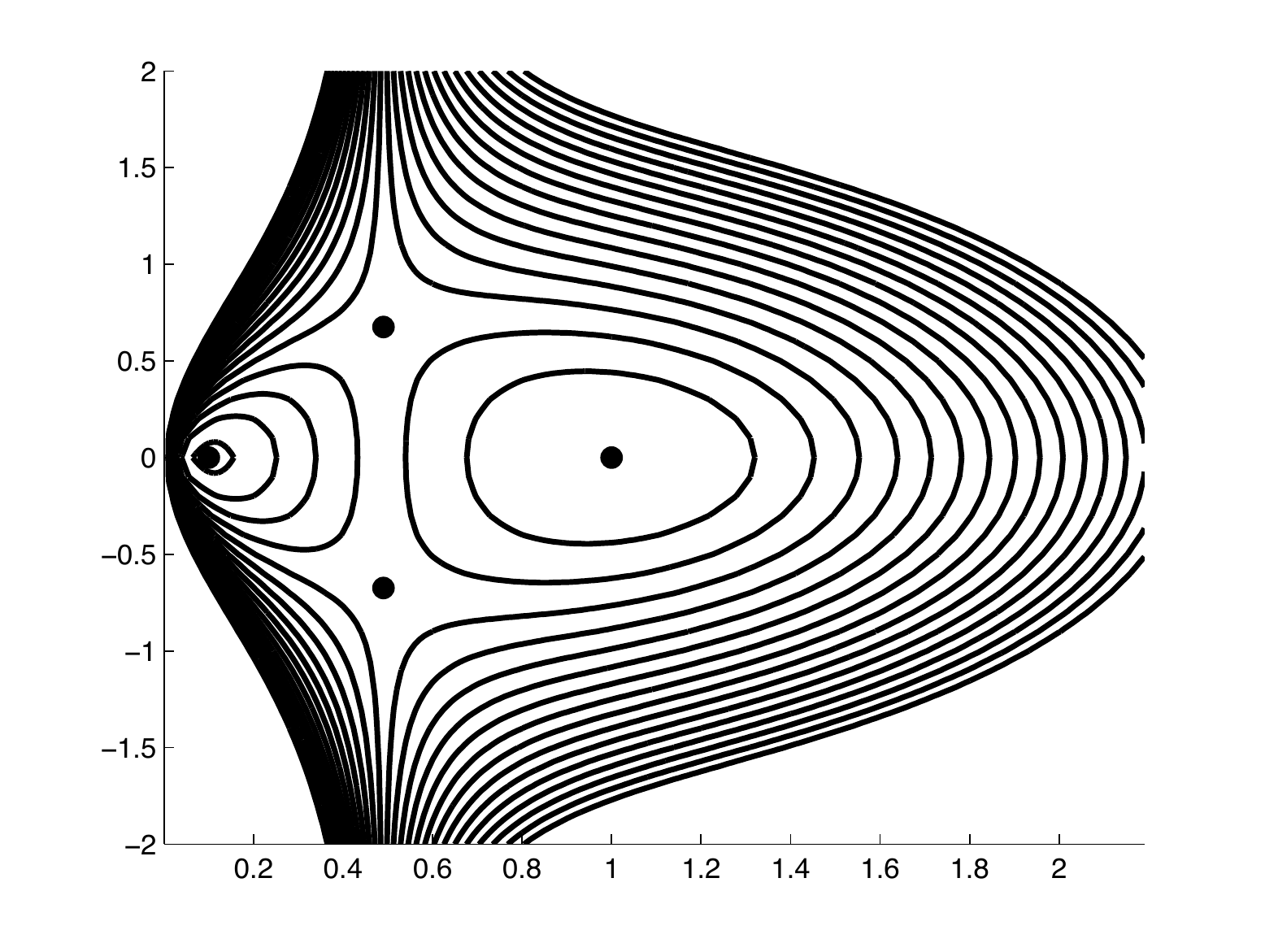}(a) & \includegraphics[width=7.5cm]{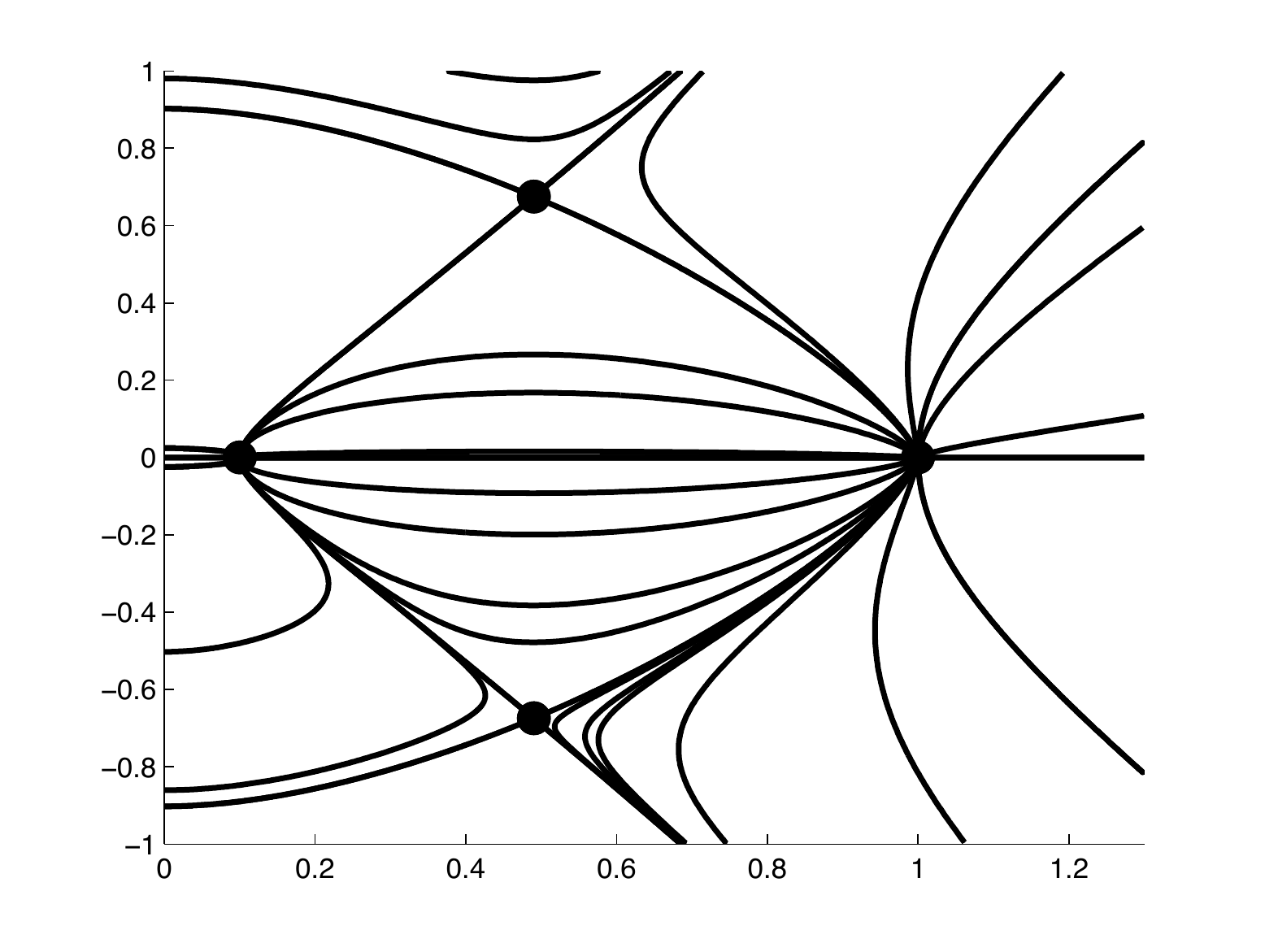}(b)
\end{array}$
\end{center}
\caption{Phase portrait in the planar case for MHD with two variables and infinite electric resistivity ($\sigma=\infty$). Parameter values are $\gamma=5/3$, $v_+=0.1$, $I=0.7$, $B_{2+}=0$, and $\mu_0=1$. In Figure (a) we plot level
sets of $\check \phi(v,w)$ given by \eqref{sigphi} and in Figure (b) we draw the phase portrait.}
\label{phaseC}
\end{figure}

%
%
%
%

\section{Evans function formulation}\label{s:lin}

\subsection{Two-dimensional MHD}\label{s:2d}
In the two-dimensional case $u_3\equiv B_3\equiv 0$,
\eqref{redeqs} becomes
\begin{equation}
\begin{split}
v_t+v_x-u_{1x}&=0\\
u_{1t}+u_{1x}+(av^{-\gamma}+B_2^2/(2\mu_0))_x&=\tau(u_{1x}/v)_x\\
u_{2t}+u_{2x}-(I/\mu_0)B_{2x}&=\mu(u_{2x}/v)_x\\
vB_2)_t+(vB_2)_x-Iu_{2x}&=(\sigma \mu_0)^{-1}(B_{2x}/v)_x,
\end{split}
\end{equation}
where $\tau=2\mu+\eta$. Linearizing about the profile solution $(\hat v,\hat u_1, \hat u_2, \hat B_2)$ we have
\begin{equation}
\begin{split}
 v_t+v_x-u_{1x}&=0\\
 u_{1t}+u_{1x}+(-a\gamma \hat v^{-\gamma -1}v+(1/\mu_0)(\hat B_2 B_2))_x&=\tau(u_{1x}/\hat v-\hat u_{1x}v/\hat v^2)_x\\
 u_{2t}+u_{2x}-(I/\mu_0)B_{2x}&=\mu(u_{2x}/\hat v-\hat u_{2x}v/\hat v^2)_x\\
 \tilde \alpha_t+\tilde \alpha _{x}-Iu_{2x}&=(\sigma \mu_0)^{-1}(B_{2x}/\hat v-\hat B_{2x}v/\hat v^2)_x,
\end{split}
\end{equation}
where $\tilde \alpha =\hat vB_2+v\hat B_2$, 
so that $B_2=(\tilde \alpha-\hat B_2v)/\hat v$. 
Substituting for $B_2$ we obtain the eigenvalue problem
\begin{equation}
\begin{split}
\lambda v+v'-u_1'&=0\\
\lambda u_1+u_1'-(h(\hat v)v/\hat v^{\gamma+1})'&=-(\hat B_2 (\tilde \alpha -\hat B_2 v)/(\mu_0\hat v))'+\tau(u_1'/\hat v)'\\
\lambda u_2+u_2'-(I/\mu_0)(\tilde \alpha/\hat v -\hat B_2v/\hat v)'&=\mu(u_2'/\hat v -\hat u_2'v/\hat v^2)'\\
\lambda \tilde\alpha +\tilde \alpha'-I u_2'&=(\sigma \mu_0)^{-1}(\hat v^{-1}(\tilde \alpha/\hat v -\hat B_2 v/\hat v)'-\hat B_2' v/\hat v^2)',
\end{split}
\end{equation}
where
\begin{small}
\begin{equation}
\begin{split}
	h(\hat v)&=-\hat v^{\gamma+1}(\tau\hat u_1'/\hat v^2-a\gamma \hat v^{-\gamma-1})\\
	&=-\hat v^{\gamma+1}(\tau\hat v'/\hat v^2-a\gamma \hat v^{-\gamma-1})\\
	&=-\hat v^{\gamma+1}(\hat v^{-2}(\hat v(\hat v-1)+a \hat v^{1-\gamma}-a\hat v+(2\mu_0\hat v)^{-1}((B_{2-}+I\hat u_2)^2-\hat v^2B_{2-}^2))-a\gamma \hat v^{-\gamma-1}).
\end{split}
\end{equation}
\end{small}

We let $u(x)=\int_{-\infty}^x u_1(z)dz$, $w=\int_{-\infty}^x u_2(z)dz$, $V=\int_{-\infty}^xv(z)dz$, and $\alpha=\int_{-\infty}^x  \tilde \alpha (z)dz$ to transform to integrated coordinates. Substituting we have
\begin{equation}
\begin{split}
\lambda V'+V''-u''&=0\\
\lambda u'+u''-(h(\hat v)V'/\hat v^{\gamma+1})'&=-(\hat B_2 ( \alpha' -\hat B_2 V')/(\mu_0\hat v))'+\tau(u''/\hat v)'\\
\lambda w'+w''-(I/\mu_0)(\alpha'/\hat v -\hat B_2V'/\hat v)'&=\mu(w''/\hat v -\hat u_2'V'/\hat v^2)'\\
\lambda \alpha' + \alpha''-I w''&=(\sigma \mu_0)^{-1}(\hat v^{-1}(( \alpha' -\hat B_2 V')/\hat v)'-\hat B_2' V'/\hat v^2)'.
\end{split}
\end{equation}

Integrating from $-\infty$ to $x$ we obtain
\begin{equation}\label{inteval}
\begin{split}
\lambda V+V'-u'&=0\\
\lambda u+u'-h(\hat v)V'/\hat v^{\gamma+1}&=-\hat B_2 ( \alpha' -\hat B_2 V')/(\mu_0\hat v)+\tau(u''/\hat v)\\
\lambda w+w'-(I/\mu_0)(\alpha'/\hat v-\hat B_2 V'/\hat v)&=\mu(w''/\hat v -\hat u_2'V'/\hat v^2)\\
\lambda \alpha + \alpha'-I w'&=(\sigma \mu_0)^{-1}(\hat v^{-1}( \alpha'/\hat v -\hat B_2 V'/\hat v)'-\hat B_2' V'/\hat v^2).
\end{split}
\end{equation}
We use the coordinates $(u,V,V',w,\mu w', \alpha, \alpha'/(\sigma \mu_0\hat v))^T$ for the Evans function formulation. Solving for the desired variables, 
and using $u''=\lambda V'+V''$, we have, finally,
\begin{small}
\begin{equation}
\begin{split}
u'&=\lambda V+V'\\
V''&=\frac{\lambda \hat vV}{\tau} +\left(-\frac{h(\hat v)}{\tau\hat v^{\gamma}} -\lambda +\frac{\hat v}{\tau} -\frac{\hat B_2^2}{\mu_0\tau}\right)V' + \frac{\lambda \hat vu}{\tau}+\frac{\hat B_2\alpha'}{\mu_0\tau}\\
w''&=\left(\frac{\hat u_2'}{\hat v}+\frac{I\hat B_2}{\mu_0\mu}\right)V' +\frac{\lambda \hat v w}{\mu}+\frac{\hat v w'}{\mu}-\frac{I\alpha'}{\mu_0 \mu}\\
\left(\frac{\alpha'}{\sigma \mu_0\hat v}\right)'&=\frac{\lambda \hat B_2 u}{\sigma \mu_0 \tau}+\frac{\lambda \hat B_2 V}{\sigma \mu_0\tau}-I\hat v w'+
\lambda \hat v \alpha  +\left(\hat v+ \frac{\hat B_2^2}{\sigma \mu_0^2 \tau \hat v}\right)\alpha'+
\\ &\left( \frac{2\hat B_2'}{\sigma \mu_0 \hat v}  -\frac{\lambda\hat B_2}{\sigma \mu_0 \hat v} +\frac{\hat B_2}{\sigma \mu_0\tau}-\frac{\hat B_2^3}{\sigma \mu_0^2\tau\hat v} -\frac{\hat B_2\hat v'}{\sigma \mu_0\hat v^2}-\frac{\hat B_2h(\hat v)}{\sigma \mu_0\tau\hat v^{\gamma+1}} \right)V'.
\end{split}
\end{equation}
\end{small}

This may be written as a first-order system $W'=A(x,\lambda)W$
from which the Evans function may be computed as described
in Section \ref{s:stabtheory}, where
\begin{equation}
A(x,\lambda)=
\begin{pmatrix}
0&\lambda &1&0&0&0&0\\
0&0&1&0&0&0&0\\
\frac{\lambda \hat v}{\tau}&\frac{\lambda \hat v}{\tau}&f(\hat v)-\lambda-\frac{\hat B_2^2}{\mu_0\tau}&0&0&0&\frac{\sigma \hat B_2 \hat v}{\tau}\\
0&0&0&0&\frac{1}{\mu}&0&0\\
0&0&\frac{\mu\hat u_2'}{\hat v}+\frac{I \hat B_2}{\mu_0}&\lambda \hat v&\frac{\hat v}{\mu}&0&-I\sigma \hat v\\
0&0&0&0&0&0&\sigma \mu_0\hat v\\
\frac{\lambda \hat B_2}{\sigma \mu_0\tau}&\frac{\lambda \hat B_2}{\sigma \mu_0\tau}&a_{73}&0&\frac{-I\hat v}{\mu}&\lambda \hat v&\sigma \mu_0\hat v^2 +\frac{\hat B_2^2}{\mu_0\tau} 
\end{pmatrix},
\end{equation}
\begin{equation}
a_{73}=\left( \frac{2\hat B_2'}{\sigma \mu_0 \hat v}  -\frac{\lambda\hat B_2}{\sigma \mu_0 \hat v} +\frac{\hat B_2}{\sigma \mu_0\tau}-\frac{\hat B_2^3}{\sigma \mu_0^2\tau\hat v} -\frac{\hat B_2\hat v'}{\sigma \mu_0\hat v^2}-\frac{\hat B_2h(\hat v)}{\sigma \mu_0\tau\hat v^{\gamma+1}} \right),
\end{equation}
  $W=(u,v,v',w,\mu w', \alpha, \alpha'/(\sigma \mu_0\hat v))^T,$
and $f(\hat v)=\tau^{-1}(\hat v-\hat v^{-\gamma}h(\hat v)). $

\subsubsection{The case $\sigma=\infty$}
In the case $u_3\equiv B_3\equiv 0$, with $\sigma=\infty$,
the integrated eigenvalue equation \eqref{inteval} becomes
\begin{equation}
\begin{split}
\lambda V+V'-u'&=0\\
\lambda u+u'-h(\hat v)V'/\hat v^{\gamma+1}&=-\hat B_2 ( \alpha' -\hat B_2 V')/(\mu_0\hat v)+\tau(u''/\hat v)\\
\lambda w+w'-(I/\mu_0)(\alpha'/\hat v-\hat B_2 V'/\hat v)&=\mu(w''/\hat v -\hat u_2'V'/\hat v^2)\\
\lambda \alpha + \alpha'-I w'&=0.
\end{split}
\end{equation}
We use the coordinates $(u,v,v',w,\mu w', \alpha)^T$. Solving for the desired variables, and using $u''=\lambda V'+V''$, $K=I^2/\mu_0$, we have

\begin{small}
\begin{equation}
\begin{split}
u'&=\lambda V+V'\\
V''&=\frac{\lambda \hat vV}{\tau} +\left(-\frac{h(\hat v)}{\tau\hat v^{\gamma}} -\lambda +\frac{\hat v}{\tau} -\frac{\hat B_2^2}{\mu_0\tau}\right)V' + \frac{\lambda \hat vu}{\tau}+\frac{I \hat B_2 w'}{\mu_0 \tau}-\frac{\lambda \hat B_2 \alpha}{\mu_0 \tau}\\
w''&=\left(\frac{\hat u_2'}{\hat v}+\frac{I\hat B_2}{\mu_0\mu}\right)V' +\frac{\lambda \hat v w}{\mu}+\frac{\hat v w'}{\mu}-\frac{Kw'}{\mu}+\frac{\lambda I\alpha}{\mu_0\mu}\\
\alpha'&=Iw'-\lambda \alpha.
\end{split}
\end{equation}
\end{small}

This may be written as a first-order system $W'=A(x,\lambda)W$, where
\begin{equation}
A(x,\lambda)=
\begin{pmatrix}
0&\lambda &1&0&0&0\\
0&0&1&0&0&0\\
\frac{\lambda \hat v}{\tau}&\frac{\lambda \hat v}{\tau}&f(\hat v)-\lambda-\frac{\hat B_2^2}{\mu_0\tau}&0&\frac{I\hat B_2}{\mu_0 \mu \tau}&-\frac{\lambda \hat B_2}{\mu_0 \tau}\\
0&0&0&0&\frac{1}{\mu}&0\\
0&0&\frac{\mu\hat u_2'}{\hat v}+\frac{I \hat B_2}{\mu_0}&\lambda \hat v&\frac{\hat v-K}{\mu}&\frac{\lambda I}{\mu_0}\\
0&0&0&0&\frac{I}{\mu}&-\lambda \end{pmatrix},\label{Amatrix2}
\end{equation}
  $W=(u,V,V',w,\mu w', \alpha)^T,$
and $f(\hat v)=\tau^{-1}(\hat v-\hat v^{-\gamma}h(\hat v)). $

\subsection{Three-dimensional stability}\label{s:3d}
Finally, we consider the question of {\it transverse stability},
or stability with respect to three-dimensional perturbations,
of a two-dimensional profile $\hat u_3\equiv \hat B_3\equiv 0$.

Carrying third components through the computations of Section
\ref{s:2d}, we obtain the two-dimensional integrated eigenvalue
equations \eqref{inteval} augmented with the additional equations
\be
\begin{split}
\lambda w_3+w_3'-(I/\mu_0)(\alpha_3'/\hat v-\hat B_3 V'/\hat v)&=\mu(w_3''/\hat v -\hat u_3'V'/\hat v^2),\\
\lambda \alpha_3 + \alpha_3'-I w_3'&=(\sigma \mu_0)^{-1}(\hat v^{-1}( \alpha_3'/\hat v -\hat B_3 V'/\hat v)'-\hat B_3' V'/\hat v^2).
\end{split}
\ee
In the case $\hat u_3\equiv \hat B_3\equiv 0$, these decouple from
the rest of the equations, reducing to a {\it transverse system}
\begin{equation}\label{3inteval}
\begin{split}
\lambda w_3+w_3'-(I/\mu_0 \hat v)\alpha_3' &=\mu w_3''/\hat v  ,\\
\lambda \alpha_3 + \alpha_3'-I w_3'&=
(\sigma \mu_0 \hat v)^{-1}( \alpha_3'/\hat v )',
\end{split}
\end{equation}
that may be studied separately.
This is similar to the situation of the parallel case
$\hat u_2=\hat u_3=\hat B_2=\hat B_3$ studied in \cite{FT,BHZ}.

\subsubsection{The case $\sigma=\infty$}\label{s:3dinfty}
For $\sigma=\infty$, the transverse equations \eqref{3inteval} become
\begin{equation}\label{3intevalinfty}
\begin{split}
\alpha_3'&=Iw_3'-\lambda \alpha_3,\\
\mu w_3''&=\lambda \hat v w_3 +\hat v w_3' -Kw_3' +\frac{I\lambda}{\mu_0}\alpha_3.
\end{split}
\end{equation}
This may be written as a first-order system $W'=A(x,\lambda)W$, where 
\begin{equation}\label{transverseA}
A(x,\lambda)=\begin{pmatrix}
0 &\frac{1}{\mu}&0\\
\lambda \hat v&\frac{\hat v-K}{\mu}&\frac{I \lambda}{\mu_0}\\
0 & \frac{I}{\mu}& -\lambda
\end{pmatrix}
\end{equation}
and $W=(w_3\ \mu w_3'\ \alpha_3)^T$, and used to compute a
{\it transverse Evans function} determining stability with respect
to perturbations in components $u_3$, $B_3$.

\subsection{Construction of the Evans function}\label{evansconst}
As described in \cite{MaZ3,Z1}, the above procedure may be
carried out for general hyperbolic--parabolic systems under
the standard assumptions 
\eqref{B}, \eqref{anon}, \eqref{nonchar}, and \eqref{diss}, 
to express the eigenvalue problem as a first-order
system $W'=A(x,\lambda)W$ with exponentially converging
coefficient
\be\label{expcoeff}
|A(x,\lambda)-A_\pm(\lambda)|\le Ce^{-\theta |x|},
\quad x\gtrless 0,
\ee
where the constant $C>0$ is uniformly bounded on bounded domains
in $\lambda$.

Moreover, defining
$ \Omega:= \{\Re \lambda \ge 0\}, $
we have under the same standard hypotheses 
the general fact \cite{MaZ3} that, on  $\Omega \setminus \{0\}$,
the system satisfies the {\it consistent splitting hypothesis}
of \cite{AGJ}:
the limiting coefficient matrices $A_\pm(\lambda)$
have no center subspaces, and the dimensions of their stable
and unstable subspaces agree and (by homotopy, using absence
of center subspace) are constant throughout $\Omega \setminus \{0\}$.
Further, the associated eigenprojections, analytic on
$\Omega\setminus \{0\}$ by spectral separation (again, absence
of center subspace), {\it extend analytically to $\lambda=0$},
so are analytic on all of the simply connected set $\Omega$.
By a standard construction of Kato \cite{Kato}, there exist
analytically chosen bases $(R_1^-,\dots, R_k^-)(\lambda)$ and
$(R_{k+1}^+,\dots R^-+N)(\lambda)$ of the unstable subspace of $A_-$ and
the stable subspace of $A_+$, respectively.

Appealing to the general construction of Appendix \ref{s:conj},
we may thus define the Evans function as
\ba\label{Ddef2}
D(\lambda)&:=\det (W_1^-, \dots, W_k^-, W_{k+1}^+, \dots, W_N^+)|_{x=0},
\ea
where, for $\lambda \in \Omega\setminus \{0\}$,
 $\{W_j^+\}$ and $\{W_j^-\}$ are analytically chosen 
bases of the manifolds of solutions of $W'=A(x,\lambda)W$
decaying as $x\to +\infty$
and $x\to -\infty$, respectively, with
\be\label{Wasympt}
W_j^\pm(\lambda, x) \sim e^{A_\pm(\lambda) x}R_j^\pm(\lambda)
\quad \hbox{\rm as } \; x\to \pm \infty.
\ee
Evidently, (i) $D$ is analytic for $\lambda\in \Omega$,
and (ii) $\lambda \in \Omega\setminus \{0\}$ is an eigenvalue
if and only if $D(\lambda)=0$; see \cite{MaZ3,Z1} for further discussion.
Moreover (see Appendix \ref{s:conj}), $D$ is continuous with respect
to model parameters $(a,J,K)$ or $(v_+,B_+, I)$.
The asymptotics \eqref{Wasympt} may be used as the basis for
numerical approximation of $D$; see \cite{Br1,Br2,BrZ,BDG,HuZ,Z3,Z4}.

\section{Analytical stability results}\label{s:analytical}

We begin by recording some analytical stability analyses
in special cases,
in particular certain asymptotic limits (small-amplitude, composite-wave,
and high-frequency limits) that are numerically difficult.

\subsection{Small-amplitude stability}\label{s:smallamp}
Consider first the small-amplitude limit, without loss of
generality (after rescaling) $v_-=1$ and $v_+\to 1$.

\begin{proposition}
For $J$, $K$ such that $0$ is a simple, genuinely nonlinear 
characteristic speed of the inviscid system for $v_+=1$, there is a unique
Lax-type profile connecting the rest points associated with $v=1$ and $v=v_+$ 
for $v_+$ sufficiently close to $1$, and this profile is Evans,
hence linearly and nonlinearly,
stable (both with respect to coplanar and transverse perturbations).
\end{proposition}

\begin{proof}
The existence result follows by a more general result of Pego \cite{Pe}
obtained by center manifold reduction.
The stability result follows by a more general stability result
of \cite{HuZ1}.
\end{proof}



\subsection{Transverse stability of monotone profiles}\label{s:transmon}
Similarly as observed in the parallel case in \cite{BHZ},
in the case $\sigma=\infty$, transverse stability holds automatically
for profiles that are monotone decreasing in $\hat v$.
Thus, in our numerical stability study, it is necessary to test
transverse stability {\it only for nonmonotone profiles}.

\begin{proposition}[\cite{BHZ}]\label{enprop}
For $\sigma=\infty$, monotone-density profiles, $\hat v_x<0$,
are Evans stable with respect to transverse perturbations:
that is, they are three-dimensionally Evans stable if and only
if they are two-dimensionally Evans stable.
\end{proposition}

\begin{proof}
Dropping subscripts,
we may rewrite \eqref{3intevalinfty} in symmetric form as
\begin{equation}\label{infsigmares}
\begin{aligned}
 \mu_0 \hat v \lambda w + \mu_0 \hat v w'  - I \alpha' &=\mu \mu_0 w'',\\
 \lambda \alpha  + \alpha' - I w' &=0.
\end{aligned}
\end{equation}
Taking the real part of the complex $L^2$-inner product of 
$w$ against the first equation and
$\alpha$ against the second equation and summing gives
$$
\Re \lambda( \int (\hat v \mu_0 |w|^2+|\alpha|^2) =
-\int \mu \mu_0 |w'|^2 +\int \hat v_x |w|^2 <0,
$$
a contradiction for $\Re \lambda \ge 0$ and $w$
not identically zero.  If $w\equiv 0$ on the other
hand, we have a constant-coefficient equation for $\alpha$,
which is therefore Evans stable. 
\end{proof}

\br
\textup{
In the above proof, we are using implicitly the fact that vanishing
of the Evans function on $\Re \lambda \ge 0$, $\lambda\ne 0$, 
away from essential spectrum of $L$, implies existence
of an eigenfunction decaying as $x\to \pm \infty$ \cite{GJ1,GJ2},
while vanishing at the point $\lambda=0$ embedded in the essential
spectrum implies existence of an $L^\infty $ eigenfunction \cite{ZH,MaZ3}.
}
\er

\br
\textup{
In the one-dimensional, parallel, case $B_2=B_3\equiv 0$,
$u_2=u_3\equiv 0$, the same argument yields stability if and
only if the corresponding gas-dynamical shock is stable \cite{BHZ}.
(Recall that in this case MHD profiles reduce to gas dynamical profiles
in $(v,u_1)$
\cite{FT,BHZ}.)
}
\er

\subsection{Stability of composite waves}\label{s:compstab}
We consider next the numerically difficult situation as
described in Section \ref{composite} of a family
of profiles passing closer and closer to one or more intermediate
rest points, i.e., composite
wave consisting of the approximate superposition of two or
more component profiles separated by a distance going to infinity.
This requires a computational domain $[-L_-,L_+]$ of size going
to infinity, hence arbitrarily large computational effort to
resolve directly.  However, it may be treated in straightforward
fashion by a singular perturbation analysis taking account of 
the limiting structure.

\subsubsection{Double Lax configuration}
Consider first the simplest case noted already in \cite{Br1,Br2}
of a family of overcompressive profiles in a four rest point configuration,
bounded by two Lax/intermediate Lax profile pairs.
Considering a family of overcompressive profiles $\bar U^\eps$ 
connecting $U_-$ and $U_+$ and passing closer and closer 
to an intermediate saddle $U_*$, parametrized by the distance
$\eps$ of the profile from $U_*$,
we find that the profiles approach composite waves consisting 
of the approximate superposition of the bounding Lax profiles 
$\bar U^1$ and $\bar U^2$ connecting
$U_-$ to $U_*$ and $U_*$ to $U_+$, separated by a distance $d(\eps)$
going to infinity as $\eps \to 0$.

\begin{proposition}\label{doublelax}
For the double-Lax configuration described,
stability of $\bar U^1$ and $\bar U^2$ implies 
stability of $\bar U^\eps$ for $\eps>0$ sufficiently small.
\end{proposition}

\begin{proof}
By a standard multi-wave argument, as described in the conservation
law setting in \cite{Z7} and (in slightly different periodic context) 
\cite{OZ}, the spectrum of any such composite wave $\bar U^\eps$ approaches 
the direct sum of the spectra of its component waves 
$\bar U^1$ and $\bar U^2$ as $\eps \to 0$.
More precisely, the Evans function $\tilde D^\eps$ associated
with $\bar U^\eps$ approaches a nonvanishing analytic multiple
of the product of the Evans functions $\tilde D^1$ and $\tilde D^2$
associated with $\bar U^1$ and $\bar U^2$, and so the
zeros of $\tilde D^\eps$ approach the union of
the zeros of $\tilde D^1$ and $\tilde D^2$; see \cite{Z7} for details.
In the case that $\bar U^1$ and $\bar U^2$ are stable Lax waves,
$\tilde D^1$ and $\tilde D^2$ are nonvanishing on $\Re \lambda \ge 0$,
and so the the union of their zeros is empty.
It follows that $D^\eps$ for $\eps$ sufficiently small is nonvanishing
on $\Re \lambda \ge 0$, giving the result.
\end{proof}

\subsubsection{Undercompressive configurations}
Next, consider the more complicated examples of 
Section \ref{composite}, of doubly composite Lax profiles $\bar U^\eps$ 
composed of the approximate superposition of
a Lax profile $\bar U^1$ and an undercompressive profile $U^2$,
and of triply composite overcompressive profiles $\bar U^{\eps_1,\eps_2}$
composed of the approximate superposition of
a Lax profile $\bar U^1$, an undercompressive profile $U^2$,
and a Lax profile $\bar U^3$, where the parameter $\eps$ (resp. 
$\eps=(\eps_1,\eps_2)$) indexes
distance of the profile from the intermediate rest point (resp. points).

\begin{proposition}\label{doubleuc}
For either the Lax--undercompressive or Lax--undercompressive--Lax
configurations described,
stability of the component waves $\bar U^j$ implies
that $\tilde D^\eps$ has at most one unstable root.
\end{proposition}

\begin{proof}
This follows by the observation, as in the proof of Proposition
\ref{doublelax}, that the zeros of $\tilde D^\eps$ approach
the union of the zeros of $\tilde D^j$ as $\eps\to 0$, together
with the fact that for 
stable Lax waves the associated Evans function has no zeros
on $\Re \lambda \ge 0$, while for stable undercompressive waves,
the associated Evans function has a single zero at $\lambda=0$
(see Proposition \ref{intcond}).
From this we may conclude that $\tilde D^\eps$ has at most one
zero on $\Re \lambda \ge 0$, giving the result.
\end{proof}

With the reduction to at most a single root, stability 
could in principle be decided as in \cite{CHNZ,Z7} by 
examination of the
mod two {\it stability index} of \cite{GZ,MaZ3}, a product
$\Gamma=\gamma \Delta$ of a transversality coefficient
$\gamma$ for the traveling wave connection and a low-frequency
stability determinant $\Delta$, both real-valued,
whose sign determines the parity of the number of unstable roots.
The boundary case $\Gamma=0$ corresponds to instability through
an extra root at $\lambda=0$ \cite{ZH,MaZ3}; hence, Evans stability
implies nonvanishing of $\Gamma$, $\gamma$, and $\Delta$.

The transversality coefficient $\gamma$ is a Wronskian of
the linearized traveling-wave ODE measuring
transversality of the intersection of the unstable manifold at
$U_-$ with the stable manifold at $U_+$ of the traveling-wave ODE,
with $\gamma\ne 0$ corresponding to transversality.
From the composite wave structure, we may deduce that transversality
of the component waves (a consequence of Evans stability, as noted
above) implies transversality of $\bar U^\eps$ for $\eps>0$
sufficiently small, or nonvanishing of $\gamma$; with further
effort, the sign of $\gamma$ may be deduced as well.

The low-frequency stability determinant $\Delta$ is for Lax
shocks equal to the Lopatinski determinant determining inviscid
stability; see the discussions of \cite{ZS,Z1}.
In particular, it is independent of the nature of the viscous
regularization, and readily computable.
For overcompressive shocks, it involves also certain variations
associated with the linearized traveling-wave ODE, as described
in \cite{ZS,Z1}, which though more complicated
can also be computed in the limit $\eps\to 0$, deciding stability.

In the Lax--undercompressive case, 
we can avoid such computations by the following observation.

\begin{corollary}\label{doubleucindex}
For the Lax--undercompressive configurations, 
suppose that the limiting endstates 
$U^0_\pm=\lim_{\eps \to 0}U_\pm^\eps$ have a stable connecting
viscous profile for some choice of viscosity ratios $r=\mu/\tau$.
Then, stability of the component waves $\bar U^j$, together
with stability (resp. instability) 
of {\bf some} composite wave $\bar U^{\eps_0}$ for
$\eps_0>0$ sufficiently small implies
stability (resp. instability) 
of {\bf all} composite waves $\bar U^\eps$ for $\eps>0$ sufficiently
small.
\end{corollary}

\begin{proof}
Since the composite wave is of Lax type, $\Delta$
is independent of the viscous regularization.  It follows that 
$\Delta\ne0$ for $\eps>0$ sufficiently small
 if there is an Evans stable profile for some choice
of viscous regularization connecting the limiting
endstates $U^0_\pm$.  (Recall from the discsussion above
that Evans stability implies $\Delta \ne 0$ \cite{MaZ3,Z1}.)
Since $\gamma\ne 0$ for $\eps>0$ sufficiently small, as observed
previously, we thus have that $\Gamma \ne 0$ for $\eps>0$ sufficiently
small, and thus $\Gamma $ is of fixed sign.
It follows that either {\it all} profiles $\bar U^\eps$ are stable
for $\eps>0$ sufficiently small, or
{\it no} profiles $\bar U^\eps$ are stable
for $\eps>0$ sufficiently small, yielding the result.
\end{proof}

Using Corollary \ref{doubleucindex}, we may conclude by numerical
evaluations of component waves and
 a sample of composite waves with $\eps$ small but nonzero
the stability of composite Lax--undercompressive waves in the 
numerically inaccessible $\eps\to 0$ limit. 

\br\label{triplermk}
\textup{
Supposing that both standard Lax and composite Lax waves composed of 
Lax--undercompressive waves have
been determined to be stable, and viewing the
Lax--undercompressive--Lax composites as the composition
of Lax$=$Lax--undercompressive and Lax waves,
we obtain the partial result that triply composite
waves $\bar U^{\eps_1, \eps_2}$ are stable for $\eps_1>0$
sufficiently small and $0<\eps_2< E(\eps_1)$, where $E>0$
depends on $\eps_1$.
However, to obtain a full result, it appears that one must carry
out the more complicated computations described in the 
introductory discussion above, and so we do not complete this case.
}
\er

\subsection{The large-amplitude limit}\label{s:largeamp}
We now consider behavior as shocks of different types 
approach their maximal amplitudes.
As computed in Appendix \ref{s:liim}, for four rest point
configurations, taking without loss of generality
$v_1 < v_2<K<v_3<v_4=1$, the minimal value of $v_1$ is $0$ and the maximum
value of $v_2$ is 
$$
\underline v(J, K)=K+\frac{J}{2}-\sqrt{\frac{J^2}{4}+J(1-K)}\le K.
$$
Likewise, the minimum value of $v_3$ is
$$
\overline v(J, K)=K+\frac{J}{2}+\sqrt{\frac{J^2}{4}+J(1-K)}\ge K.
$$

Thus, for fixed $J>0$, $K\ge 0$, the maximum-amplitude Lax $1$-shock
connects the rest points associated with $v_4=1$ and
$v_3=\overline v >K$, and the maximum-amplitude intermediate
Lax $1$-shock the rest points associated with
$v_4=1$ and $v_2=\underline v <K$.
The maximum-amplitude Lax $2$-shock connects the rest
points associated with $v_2=\underline v$ and $v_1\to 0$,
and the maximum-amplitude intermediate Lax $2$-shock the rest
points associated with $v_3=\overline v$ and $v_1\to 0$.
The maximum-amplitude (intermediate) overcompressive shock connects
the rest points associated with $v_4=1$ and $v_1\to 0$.
For each of these limits, also $a\to 0$.

(Here and below, we refer to two-dimensional shock types.)

\begin{proposition}\label{p:conv}
For fixed $J$, $K$, the Evans function 
associated with Lax $1$-shocks or intermediate Lax $1$-shocks
converges in the large-amplitude limit, uniformly on compact
subsets of $\Re \lambda \ge 0$, to the Evans function
associated with the zero-pressure limit $a=0$.
\end{proposition}

\begin{proof}
An immediate consequence of the general property of
continuous dependence on parameters of the Evans function,
so long as the profile remains noncharacteristic and $v$
remains bounded from the value $v=0$ at which the pressure
function becomes singular.
Noting that $a=0$ is bounded from the values $a_*(J,K)>0$
and $A(J,K)>0$ at which profiles become characteristic
(see Appendix \ref{s:char}), and that $\underline v$, $\overline v\ne 0$,
we obtain the result.
\end{proof}

The important implication of Proposition \ref{p:conv} is
that stability of $1$-shocks may be assessed numerically by computations
on a finite mesh, even in the large-amplitude limit.

\medskip
{\bf Conjecture.}  {\it We conjecture that, similarly, the Evans
function associated with Lax $2$-shocks or overcompressive shocks
converge in the large-amplitude limit $v_+,a\to 0$ to an Evans
function associated with the zero-pressure limit $a=0$.}

{\it Motivation}.  In the parallel case $J=0$, this was shown by
a delicate asymptotic ODE analysis in \cite{HLZ,BHZ}.
Our numerics (Section \ref{s:num}) indicate similar
behavior in the general case; moreover, the limiting 
structure of the equations is quite similar, suggesting that
the proof of \cite{HLZ,BHZ} might extend with further care
to nonzero values of $J$.

\subsubsection{Large-amplitude limit for transverse equations}
As observed in \cite{BHZ},
the coefficient matrix $A(x,\lambda)$ for
the transverse eigenvalue system \eqref{transverseA}
is smooth (indeed, linear!) in the profile variable $\hat v$,
hence we obtain convergence in the large-amplitude limit
of the transverse Evans function 
by the standard property of continuous dependence of the Evans
function on parameters, Appendix \ref{s:conj}, so long as
the profile $\hat v$ converges uniformly exponentially
to its endstates, independent of $a\ge 0$, as it 
does in the regular limit arising for Lax $1$-shocks, and
appears numerically to do for Lax $2$-shocks and overcompressive
shocks as well.
Our numerics (Section \ref{s:num}) indeed suggest convergence.

\subsection{The high-frequency limit}\label{HF}
Finally, we recall the following high-frequency asymptotics
established in \cite{HLyZ1}, which we will use in our
numerical studies to truncate
the computational domain in $\lambda$.  

\begin{proposition}[\cite{HLyZ1}] \label{hflem}
Let $\tilde D$ be the (integrated) Evans function associated with a
noncharacteristic shock profile of
\eqref{MHD} (with either $\sigma=\infty$ or $\sigma $ finite). 
Then, for some constants $C$, $\alpha$, 
\begin{equation}\label{hflim}
\lim_{|\lambda|\to \infty}
\tilde D(\lambda)/ e^{ \alpha \lambda^{1/2}}=C,
\; \hbox{\rm uniformly on $\Re \lambda\ge 0$}.
\end{equation}
In particular, $\tilde D$ does not vanish for $\Re \lambda \ge 0$
and $|\lambda|$ sufficiently large.
\end{proposition}

\begin{proof}
This was proved in \cite{HLyZ1} for isentropic gas dynamics
in Lagrangian coordinates by an argument using the tracking
lemma of \cite{MaZ3,PZ}.
However, the same argument applies to general hyperbolic--parabolic
systems satisfying the standard hypotheses
\eqref{B}, \eqref{anon}, \eqref{nonchar}, \eqref{diss}, 
with the additional property that convection in hyperbolic modes
is at constant speed.
In this case, hyperbolic modes are
specific volume $v$ and, when $\sigma=\infty$, magnetic field
$B$, each of which in Lagrangian coordinates
are convected with constant speed $-s=1$.
Thus, the hypotheses are satisfied, and the result follows.
(In the general case, 
$\tilde D(\lambda) \sim  Ce^{ \alpha \lambda^{1/2} +\beta \lambda}$
for some $\alpha$, $\beta$, $C$.)
\end{proof}

%
%

\section{Numerical stability investigation}\label{s:num}
\label{numerics}
In this section, we discuss our approach to Evans function computation, 
which is used to determine whether any unstable eigenvalues 
exist in our system.
Our approach follows the polar-coordinate method developed in \cite{HuZ}; 
see also \cite{BHRZ,HLZ,HLyZ1,BHZ}.  
Since the Evans function is analytic in the region of interest, 
we can numerically compute its winding number in the right-half plane 
around a large semicircle $B(0,R)\cap \{\Re \lambda \ge 0\}$ 
appropriately chosen, thus enclosing all possible unstable roots.  
This allows us to systematically locate roots (and hence unstable eigenvalues) 
within.  
As a result, spectral stability can be determined, and in the case of instability, 
one can produce bifurcation diagrams to illustrate and observe its onset.  
This approach was first used by Evans and Feroe \cite{EF} and has been applied to various systems since; 
see for example \cite{PSW,AS,Br2,BDG}.

\subsection{Approximation of the profile}
\label{profnum}

Following \cite{BHRZ,HLZ}, we approximate the traveling wave profile 
using one of \textsc{MATLAB}'s boundary-value solvers {\tt bvp4c} \cite{SGT}, {\tt bvp5c} \cite{KL}, or {\tt bvp6c} \cite{HM}, which are adaptive Lobatto quadrature schemes and can be interchanged for our purposes.  These calculations are performed on a finite computational domain $[-L_-,L_+]$ with projective boundary conditions $M_\pm (U-U_\pm)=0$.  
The values of approximate plus and minus spatial infinity $L_\pm$ are determined experimentally by the requirement that the absolute error $|U(\pm L_\pm)-U_\pm|$ be within a prescribed tolerance, say $TOL=10^{-3}$.
For rigorous error/convergence bounds for these algorithms, 
see, e.g., \cite{Be1,Be2}.

\subsection{Approximation of the Evans function}\label{evansnum} 

Throughout our numerical study, we use 
the polar-coordinate method described in \cite{HuZ}, 
which encodes $\cW=r\,\Omega$, 
where 
$$
\mathcal{W}=W_1\wedge \cdots \wedge W_k 
$$
is the exterior product encoding the minors of $W_1,\dots, W_k$,
``angle'' $\Omega=\omega_1\wedge \cdots \wedge \omega_k$ is the exterior product of an orthonormal basis $\{\omega_j\}$ of $\Span \{W_1, \dots, W_k\}$ evolving independently of $r$ by some implementation (e.g., Drury's method) of continuous orthogonalization and ``radius'' $r$ is a complex scalar evolving by a scalar ODE slaved to $\Omega$, related to Abel's formula for evolution of a full Wronskian; see \cite{HuZ,Z3,Z4} for further details.
The Evans function is then recovered through
$$
D(\lambda)=\mathcal{W}^- \wedge \mathcal{W}^+|_{x=0}=
\det(W_1^-, \dots, W_k^-, W_{k+1}^+, \dots, W_{N})|_{x=0}.
$$

Here, $\mathcal{W}^\pm$ are approximated 
at $x=-L_-,L_+$ using asymptotics \eqref{Wasympt} for $W_j^\pm$ by
$$
\mathcal{W}^-(-L_-) \sim
e^{-\mu L_-}(R_1^-\wedge \cdots \wedge R_k^-)
$$
where $\{R_j^-\}$ is an analytically chosen basis for the
unstable subspace $U(A_-)$ of $A_-$ and 
$\mu= \Trace A_-|_{U(A_-)}$, and then evolved using the polar coordinate
ODE toward the value $x=0$ where the Evans function is evaluated.
The requirements on approximate plus and minus spatial infinity $L_\pm$ 
needed for accuracy are in practice the same as the
requirement already imposed in the approximation of the profile
that the absolute error $|U(\pm L_\pm)-U_\pm|$ be within prescribed tolerance
$TOL=10^{-3}$; see \cite[Section 5.3.4]{HLyZ1} for a complete discussion.  
$L_\pm=10$ sufficed for most parameter values.

\subsubsection{Shooting and initialization}

The ODE calculations for individual $\lambda$ are carried out using \textsc{MATLAB}'s {\tt ode45} routine, which is the adaptive 4th-order Runge-Kutta-Fehlberg method (RKF45).  This method is known to have excellent accuracy with automatic error control.  Typical runs involved roughly $60$ mesh points per side, with error tolerance set to {\tt AbsTol = 1e-8} and {\tt RelTol = 1e-6}.  
To produce analytically varying Evans function output, the initializing
bases $\{R^\pm_j\}$ are chosen analytically using Kato's ODE; see
\cite{GZ,HuZ,BrZ,BHZ} for further discussion.  
Numerical integration of Kato's ODE is carried out using a simple
second-order algorithm introduced in \cite{Z3,Z4}, a generalization
of the first-order algorithm of \cite{BrZ}.

\subsubsection{Winding number computation}
\label{windingalg}
We compute the winding number of the integrated Evans function
$\tilde D$ around the 
around the semicircle 
$$
S:= \partial \big( B(0,R)\cap \{\Re \lambda \ge 0\} \big)
$$
by varying values of $\lambda$ 
along $20$ points of the contour $S$, with mesh size taken quadratic in modulus to concentrate sample points near the origin where angles change more quickly,
and summing the resulting changes in ${\rm arg}(\tilde D(\lambda))$, using $\Im \log \tilde D(\lambda) = {\rm arg} \tilde D(\lambda) ({\rm mod} 2\pi)$, available in \textsc{MATLAB} by direct function calls.
As a check on winding number accuracy, we test a posteriori that the change in 
$\tilde D$ for each step is less than $0.2$, and add mesh points, as necessary, to achieve this.  Recall, by Rouch\'e's Theorem, that accuracy is preserved 
so long as relative variation of $\tilde D$ along each mesh interval remains
less than $1.0$. 
In Table \ref{tb_error} we give as a triple the radius of the domain contour, the number of mesh points, and the relative error for change in argument of $\tilde D(\lambda)$
 between steps.

Care must be taken to choose $R$ sufficiently large to ensure any unstable eigenvalues lie inside the domain contour $S$.
Recall, Proposition \ref{hflem}, that
\begin{equation}\label{Rlimit}
\lim_{|\lambda|\to \infty} \frac{\tilde D(\lambda)}{e^{\alpha \lambda^{1/2}}}=C
\; \hbox{\rm uniformly on} \; \Re \lambda\ge 0,
\end{equation}
where $\alpha$ and $C$ are constants. 
The knowledge that limit \eqref{Rlimit} exists allows us to determine $\alpha,\ C$ by curve fitting of $\log \tilde D(\lambda)=\log C+\alpha \lambda^{1/2}$ with respect to $z:=\lambda^{1/2}$, for $|\lambda|>>1$. When $\tilde D$ is initialized in the standard way on the real axis, so that $\tilde D(\lambda)=\tilde D(\bar \lambda)$, $\alpha$
and $C$ are necessarily real. 
We then determine the necessary size $R$ of the radius by a convergence
study, taking $R$ to be a value for which the relative error
between $\tilde D(\lambda)$ and $Ce^{\alpha \sqrt{\lambda}}$ becomes
less than $.1$ on the entire semicircle with $\Re \lambda \ge 0$, 
indicating sufficient convergence to ensure nonvanishing.  
(Relative error $<1$ implies nonvanishing.)
For many parameter combinations, $R=2$ was sufficiently large, though some required a much larger radius.

\br
\textup{
Alternatively, we could use energy estimates or direct tracking bounds
as in \cite{HLZ} and \cite{HLyZ1}, respectively, to eliminate
the possibility of eigenvalues of sufficiently high frequency.
However, we have found the convergence study to be much more efficient
in practice; see \cite{HLyZ1}.
}
\er

\begin{table}[!b]
\begin{tabular}{|c||c|c|c|c|c|}
\hline
$v_+$&$B_+=0.2$&$B_+=0.8$&$B_+=1.2$&$B_+=1.6$&$B_+=2$\\
\hline
\hline
0.1&(2,20,1.5(-1)&(4,20,1.8(-1)&(8,64,9.1(-2))&(16,64,1.3(-1)&(2,20,7.6(-2))\\
\hline
0.4&(2,20,1.1(-1))&(2,20,3.3(-2))&(2,20,3.1(-2))&(2,20,3.1(-2))&(2,20,3.2(-2))\\
\hline
0.6&(2,20,7.7(-2))&(2,20,1.6(-2))&(2,20,1.8(-2))&(2,20,1.7(-2))&(2,20,1.8(-2))\\
\hline
0.8&(2,20,6.6(-2))&(2,20,2.0(-2))&(2,20,1.9(-2))&(2,20,1.9(-2))&(2,20,2.0(-2))\\
\hline
\end{tabular}
\caption{Table demonstrating contour radius, number of mesh points, and relative error. Here $I=1.2$ and $\gamma=5/3$.}
\label{tb_error}
\end{table}

\subsection{Description of experiments: broad range}

In our numerical study, we covered a broad intermediate parameter range to demonstrate stability of Lax and overcompressive profiles. 
To avoid redundancy, we discarded 
four rest-point configurations for which $v=1$ was not the largest 
($v$-value of a) rest point,
since these can always be rescaled to an equivalent configuration
for which $v=1$ is largest, hence otherwise would be counted twice.
The following parameter combinations were examined, when physically meaningful, for Evans stability:

\begin{align*}
(\gamma,v_+,I,B_{2+}, \mu_0) &\in \{7/5, 5/3\}\\
&\quad \times \{0.8, 0.7, 0.6,0.5, 0.4,0.3, 0.2, 10^{-1},10^{-2}\}\\
&\quad \times \{0.2, 0.4, 0.6, 0.8, 1.2, 1.4, 1.6, 1.8, 2.0\}\\
&\quad \times \{0.2, 0.4, 0.6, 0.8, 1.0, 1.2, 1.4, 1.6, 1.8, 2.0\}\\
&\quad \times \{1.0\}.
\end{align*}

For $v_+=10^{-2}$ above,
the Mach number, as computed in appendix \ref{Machnum}, typically varies between $20$ and $40$. For a little over 30 of the parameter combinations
above for which $I>1$, we took $v_+=10^{-3},\ 10^{-4},$ and $10^{-5}$ attaining a Mach number of over $10,000$ in some cases. All Evans function computations were consistent with stability. 

We also covered a broad intermediate 
range in terms of the parameters $(K,J,a)$. When physically relevant we examined the parameter combinations:
\begin{align*}
&(\gamma,K,J,v_+, \mu_0) \in \{7/5, 5/3\}\\
&\quad \times \{0.1, 0.2, 0.3,0 .4, 0.5,0 .6,0 .7,0 .8,0 .9,0 .95, 1.05, 1.1, 1.2, 1.3, 1.4, 1.5, 1.6, 1.7, 1.8, 1.9, 2.0\}\\
&\quad \times \{0.1,0 .2,0 .3,0 .4,0 .5,0 .6,0 .7,0 .8 ,0.9 ,1.0, 1.1, 1.2, 1.3, 1.4, 1.5, 1.6, 1.7, 1.8, 1.9, 2.0\}\\
&\quad \times \{0.1\}\\
&\quad \times \{1.0\}.
\end{align*}

Finally, we examined the stability of the whole family of over compressive profiles for the relevant parameters belonging to
\begin{align*}
&(\gamma,K,J,a, \mu_0) \in \{7/5, 5/3\}\\
&\quad \times \{0.1, 0.2, 0.3,0 .4, 0.5,0 .6,0 .7,0 .8,0 .9,0 .95\}\\
&\quad \times \{0.1,0 .2,0 .3,0 .4,0 .5,0 .6,0 .7,0 .8 ,0.9 ,1.0, 1.1, 1.2, 1.3, 1.4, 1.5, 1.6, 1.7, 1.8, 1.9, 2.0\}\\
&\quad \times \{a_1,a_2,a_3,a_4,a_5\}\\
&\quad \times \{1.0\}.
\end{align*}
where $a_1=10^{-3}$ and  $a_5$ is the largest value of $a$ such that the system has $4$ fixed points of the form $(v,w)$ with $v\leq 1$. For each value $a_j$ we examined the stability of $5$ profiles chosen by requiring they pass through evenly spaced points along the line in the phase plane connecting the two rest points with intermediate $v_+$ coordinates, thus insuring our profiles be representative of the family of over compressive traveling waves. In Figure \ref{phaseD} we plot in bold some profiles examined in our over compressive study.

 \begin{figure}[htbp]
\begin{center}
$\begin{array}{lr}
\includegraphics[width=7.5cm]{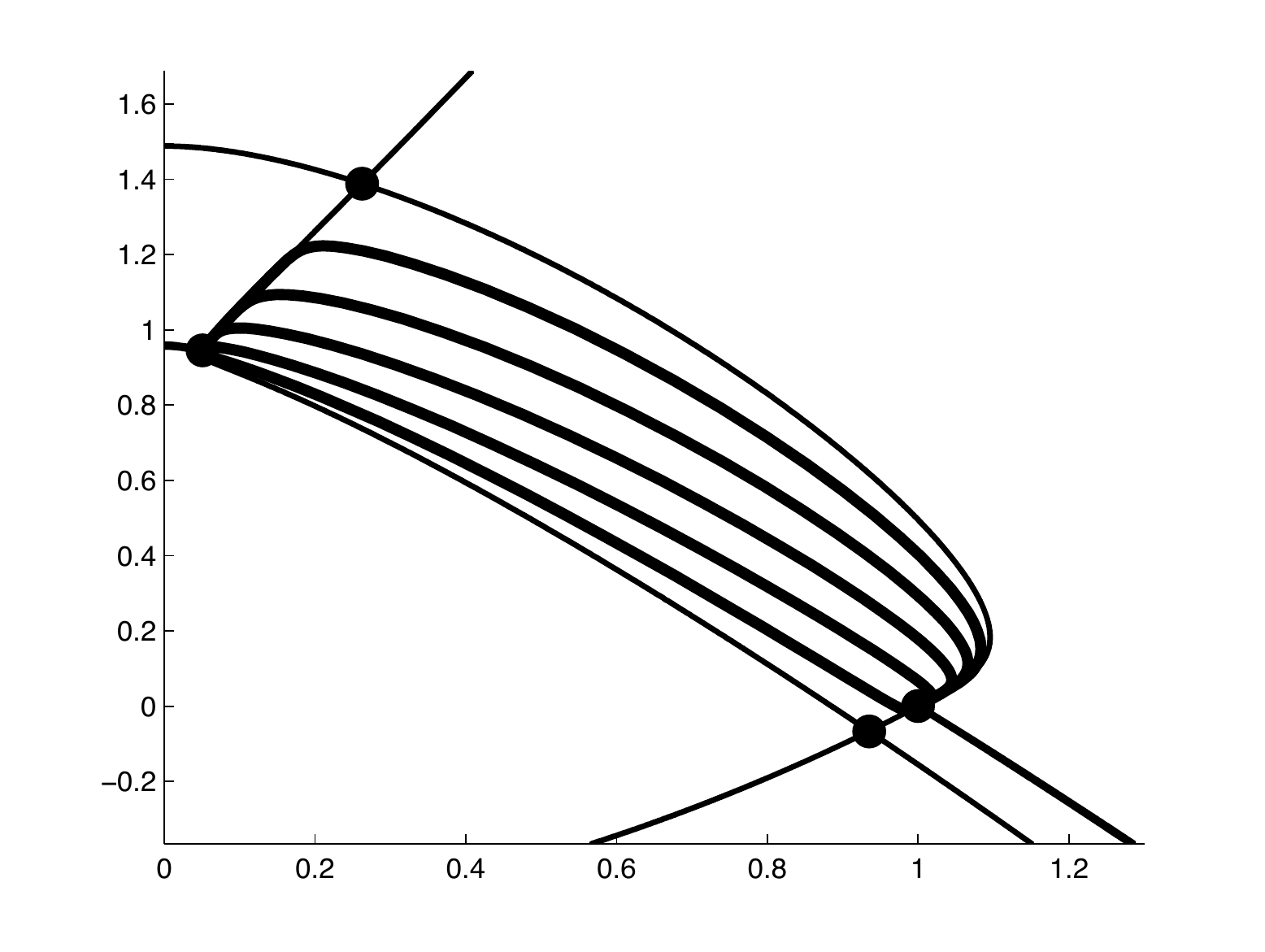}(a) & \includegraphics[width=7.5cm]{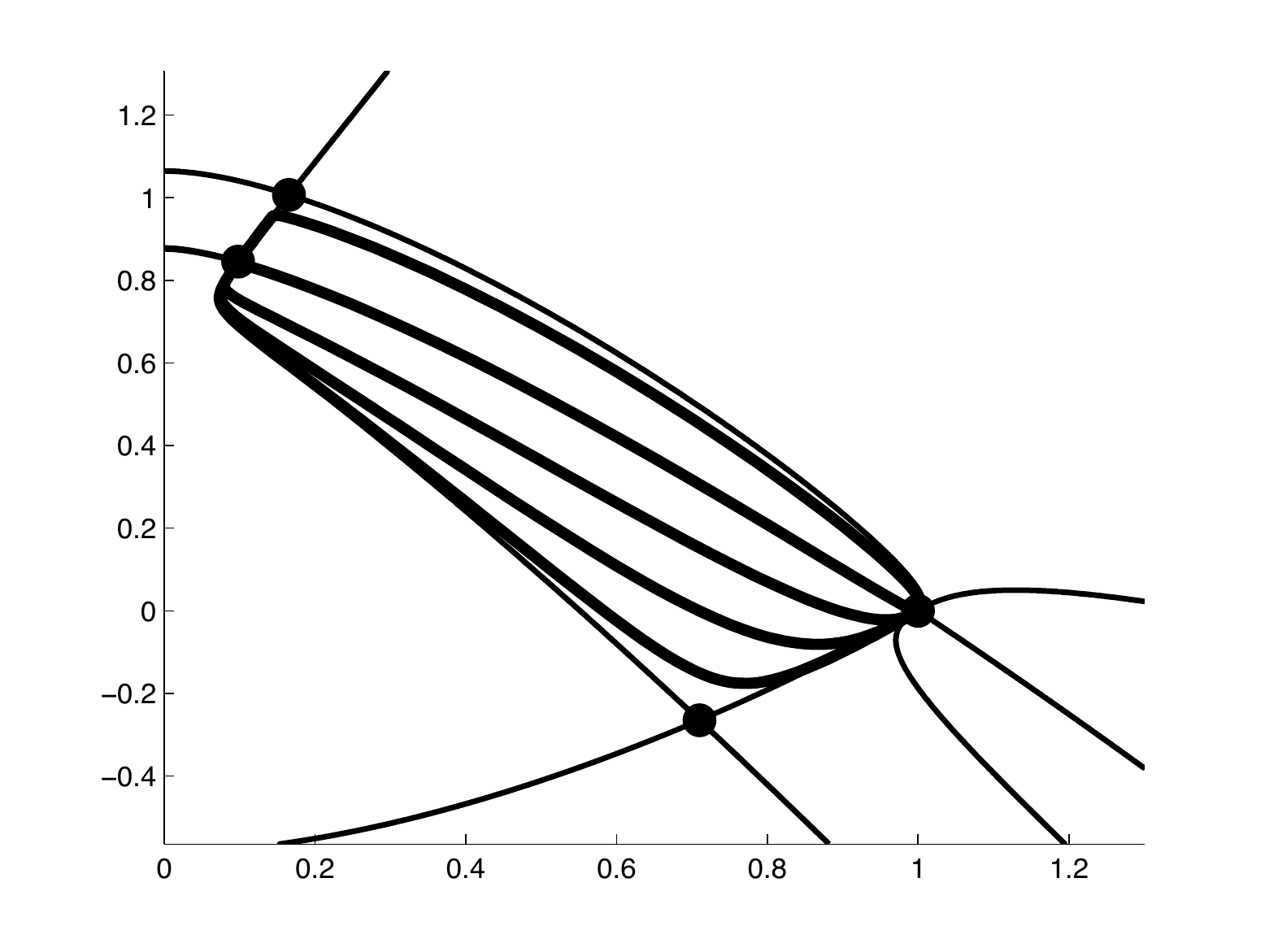}(b)
\end{array}$
\end{center}
	\caption{The bold curves in the phase portrait are the over compressive profiles for which the integrated
Evans function $\tilde D(\lambda)$ was computed. The parameter values for Figure (a) are $\gamma=5/3$, $v_+=0.1$, $I=0.7$, $B_{2+}=0.7$, and $\mu_0=1$. In Figure (b) we have $\gamma=5/3$, $v_+=0.1$, $I=0.6$, $B_{2+}=0.9$, and $\mu_0=1$.}
\label{phaseD}
\end{figure}

In the whole investigation, each contour computed consisted of at least 40 points in $\lambda$.  
{\it In all cases, we found the system to be Evans stable}.  
Typical output is given in Figure \ref{EvansA}.
We remark that the Evans function is symmetric under reflections along the real axis (conjugation).  Hence, we only needed to compute along half of the contour (usually 20 points in the first quadrant) to produce our results.

 \begin{figure}[htbp]
\begin{center}
$\begin{array}{lr}
\includegraphics[width=7.5cm]{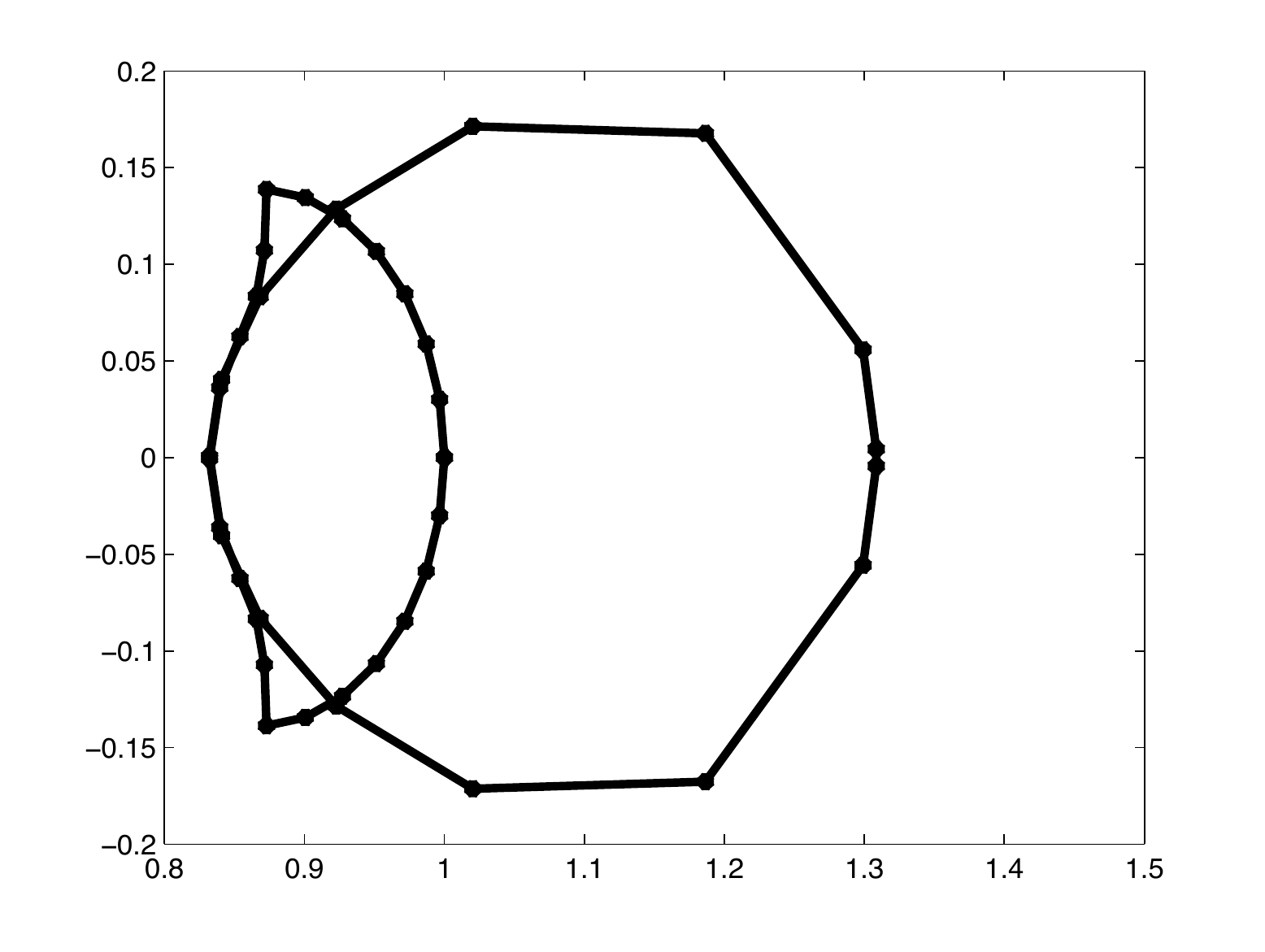}(a) & \includegraphics[width=7.5cm]{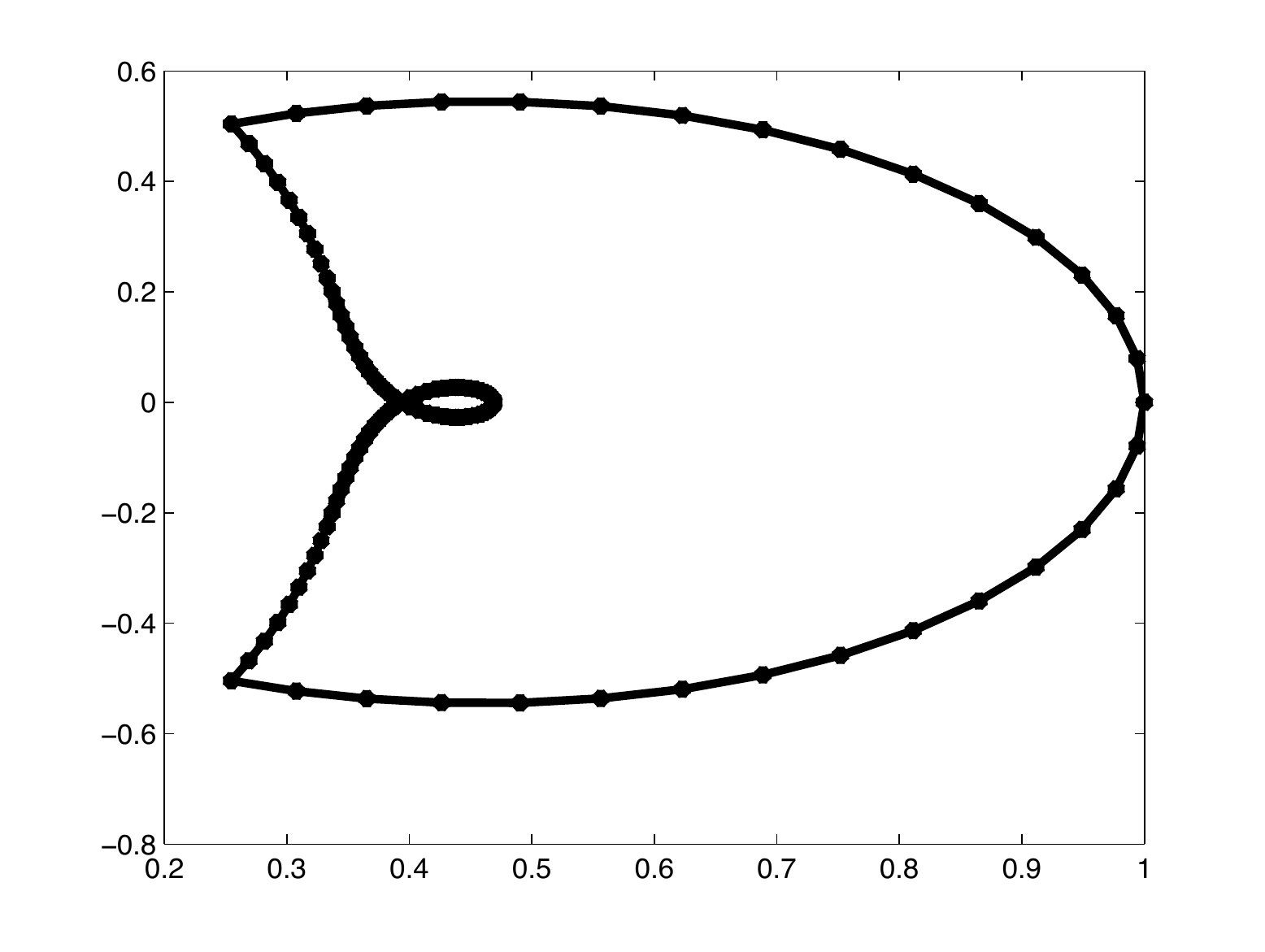}(b)
\end{array}$
\end{center}
	\caption{Typical Evans function output. The parameter values for Figure (a) are $\gamma=5/3$, $v_+=0.1$, $I=0.7$, $B_{2+}=0.7$, and $\mu_0=1$. In Figure (b) we have $\gamma=5/3$, $v_+=0.1$, $I=1.4$, $B_{2+}=1.4$, and $\mu_0=1$.}
\label{EvansA}
\end{figure}

\subsection{Composite limit}
As described in Section \ref{s:compstab}, 
as overcompressive shocks approach
the limit of a composite wave formed by the approximate superposition
of the bounding Lax $1$ and $3$-shocks, separated by larger and larger
distance, the Evans function computation becomes prohibitively costly.
However, the analytical result of Proposition \ref{doublelax} shows that
we need not carry that out, since stability in the composite limit
follows by stability of the component Lax waves, already tested.

\subsection{Large-amplitude limit}
As shown analytically in Section \ref{s:largeamp}, the Evans functions for
Lax $1$-shocks and intermediate Lax $1$-shocks converge in
the large-amplitude limit $a\to 0$, both for coplanar and transverse
perturbations, in this case the nonphysical boundary $a\to 0$
need not be treated in any special way.

We carried out numerical case studies suggesting 
that the Evans function converges
in the large-amplitude limit also for the more singular
cases of Lax $2$-shocks, intermediate Lax $2$-shocks,
and intermediate overcompressive shocks, left unresolved
in the analytical treatment of Section \ref{s:largeamp}.
We conjecture that convergence holds also in these cases,
as shown in the parallel case $J=0$ in \cite{HLZ,BHZ}.

A case study of the Lax $2$-shock and intermediate Lax
$2$-shock cases is
displayed in Figure \ref{large2}, corresponding to a two-rest point
configuration, with parameters $K=2$, $J=1$, $\gamma=5/3$,
$a=10^{-3}, 10^{-4}, 10^{-5}, 10^{-6}, 10^{-7}, 10^{-8}$.
We found stability for all amplitudes in each of these cases. 
We only had to take contour radius $R=16$ for $a$ as small as $a=10^{-8}$,
so these runs were not computationally expensive.
The Mach number for $a=10^{-8}$ is $ \approx 10,954$.

 \begin{figure}[htbp]
\begin{center}
$\begin{array}{lr}
\includegraphics[width=7.5cm]{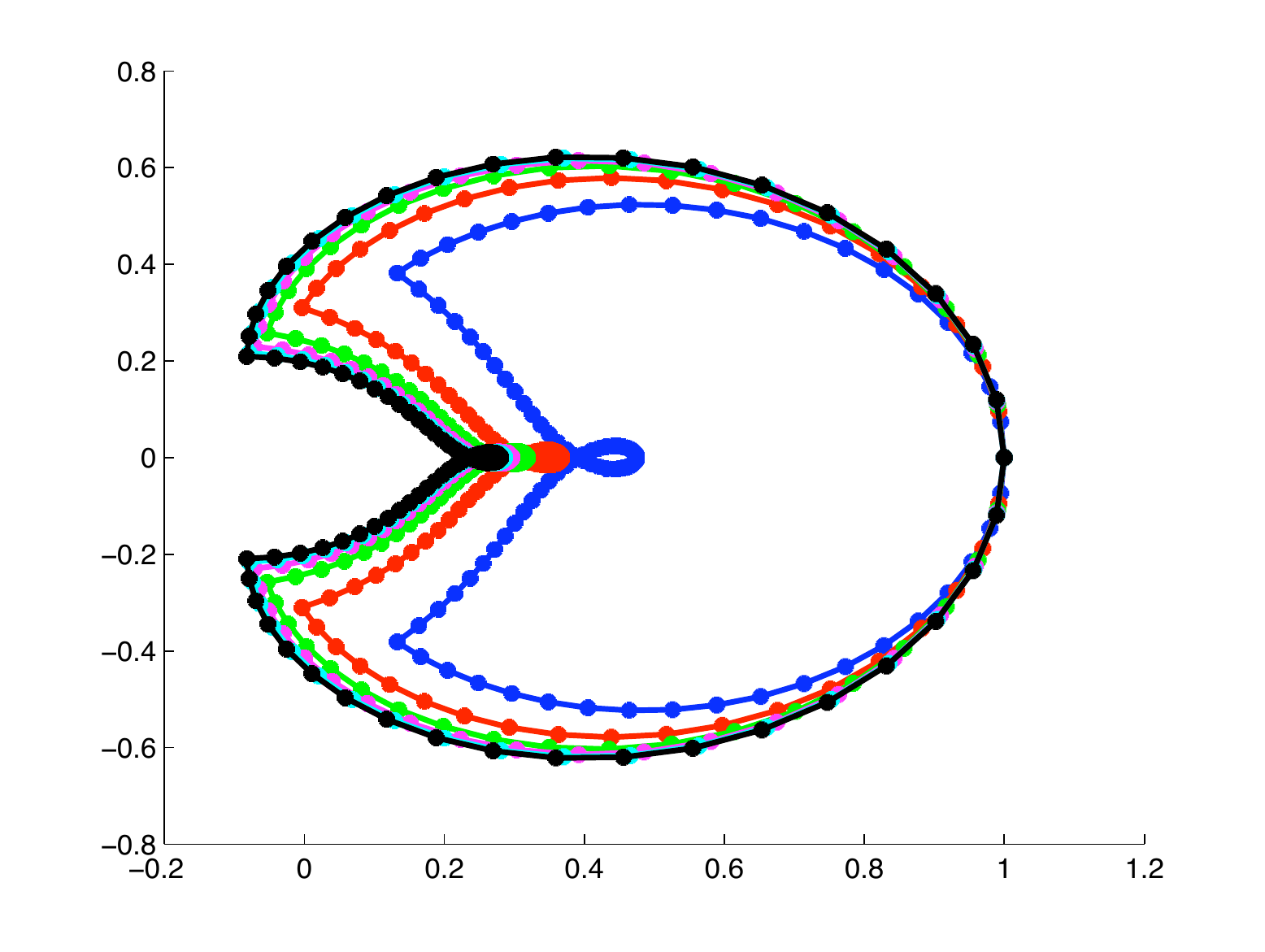}(a) & \includegraphics[width=7.5cm]{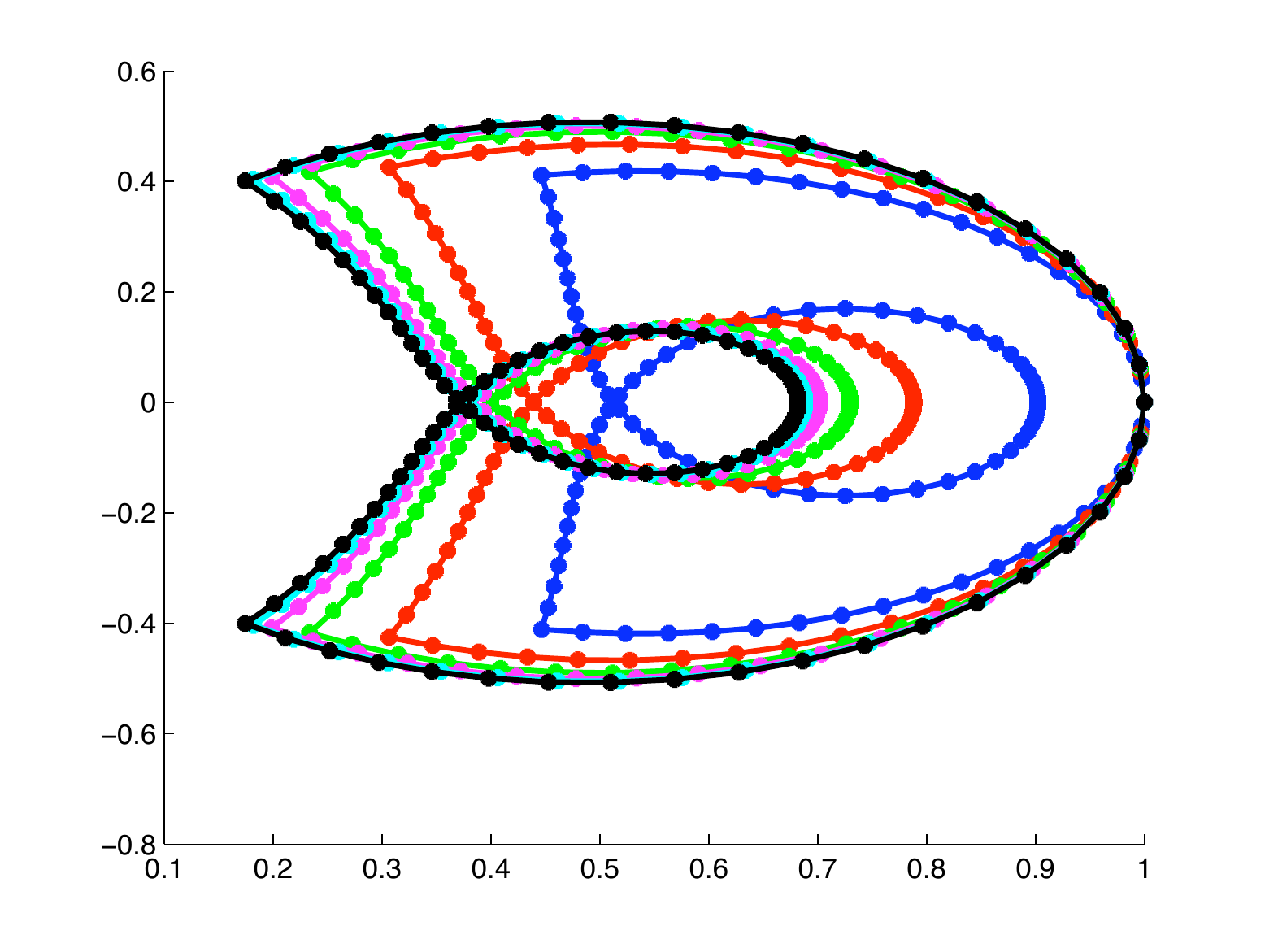}(b)
\end{array}$
\end{center}
	\caption{
Large-amplitude limits, parameters $K=2$, $J=1$, $\gamma=5/3$.
In Figure (a), we display the image of the semicircle under $\tilde D$
for a Lax $2$-shock in two-rest point configuration in the
$a\to 0$ limit,
$a=10^{-3}, 10^{-4}, 10^{-5}, 10^{-6}, 10^{-7}, 10^{-8}$.
where $a=10^{-8}$ corresponds to Mach number $\approx 10,954$.
Convergence of contours appears to occur at $a\sim 10^{-6}$.
or Mach number $\approx 1,095$.
In Figure (b), for the same sequence of $a$-values, we display the
images under the transverse Evans function, again suggestive of convergence. 
}
\label{large2}
\end{figure}

A case study of the Lax $2$-shock, intermediate Lax
$2$-shock, and intermediate overcompressive cases is
displayed in Figure \ref{large4}, corresponding to a four-rest point
configuration, with parameters 
$K=0.7$, $J=0.5$, and $a=10^{-1}, 10^{-2},\dots , 10^{-k}$, 
taking $a$ as small as necessary to achieve convergence:
for example, in the overcompressive case, $a=10^{-7}$, or 
Mach number $\approx 3,817$.
In each case, convergence was achieved; likewise,
we again found stability for all amplitudes. 
See Figure \ref{phaselimit} for the corresponding phase portrait
with $K=0.7$, $J=0.5$, and $a=10^{-8}\sim 0$, approximating the
$a\to 0$ limit.
Note that each of the Lax $2$-shock, intermediate
Lax $2$-shock, and intermediate overcompressive shock profiles
appear to lie on a straight line orbit.  It would be interesting
to check whether the $a=0$ traveling-wave ODE, 
a polynomial (cubic) vector field, indeed supports exact straight
line connections.

\begin{figure}[htbp]
\begin{center}
$\begin{array}{lr}
\includegraphics[width=7.5cm]{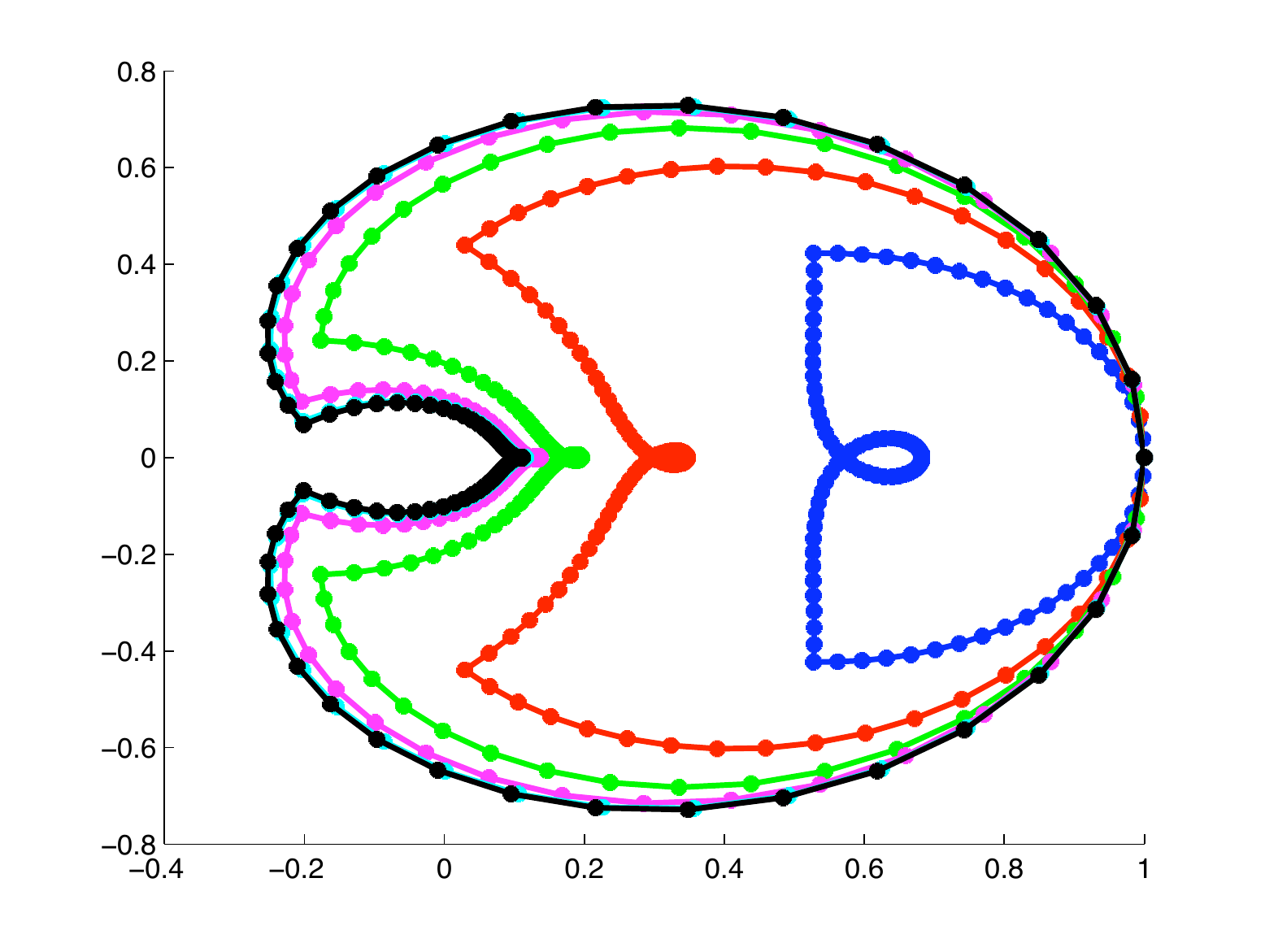}(a) & 
\includegraphics[width=7.5cm]{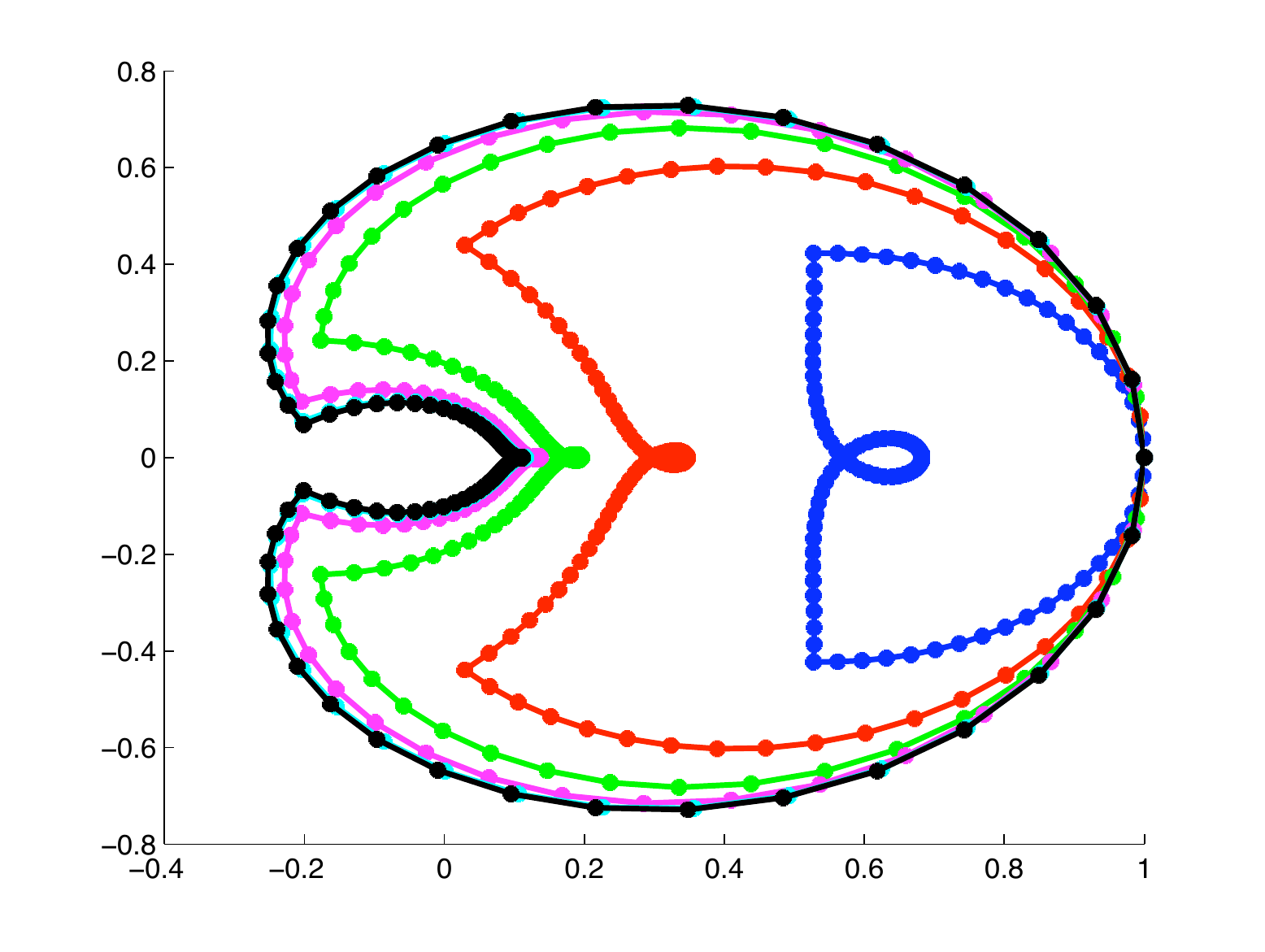}(b)\\
\includegraphics[width=7.5cm]{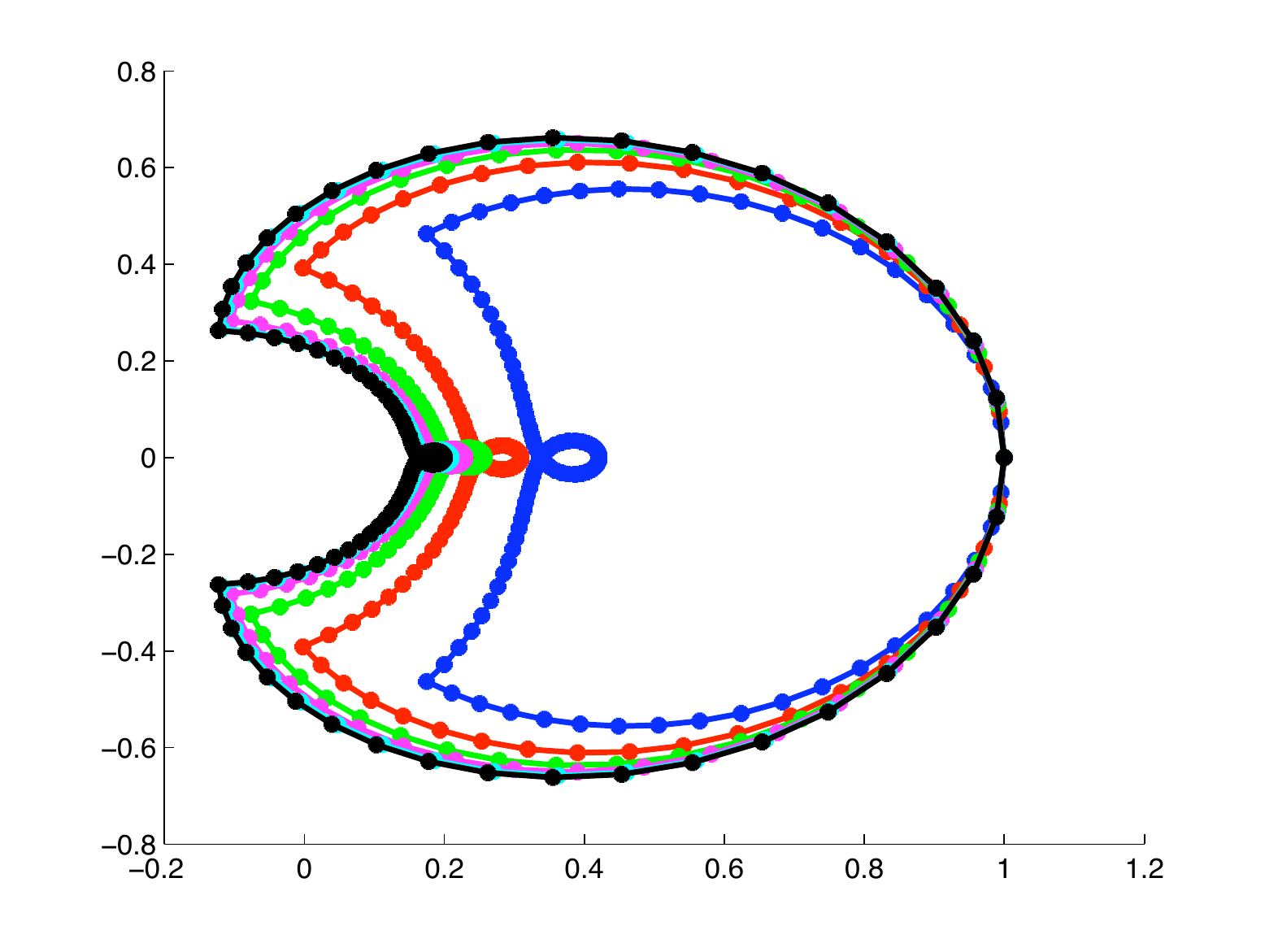}(c) & 
\includegraphics[width=7.5cm]{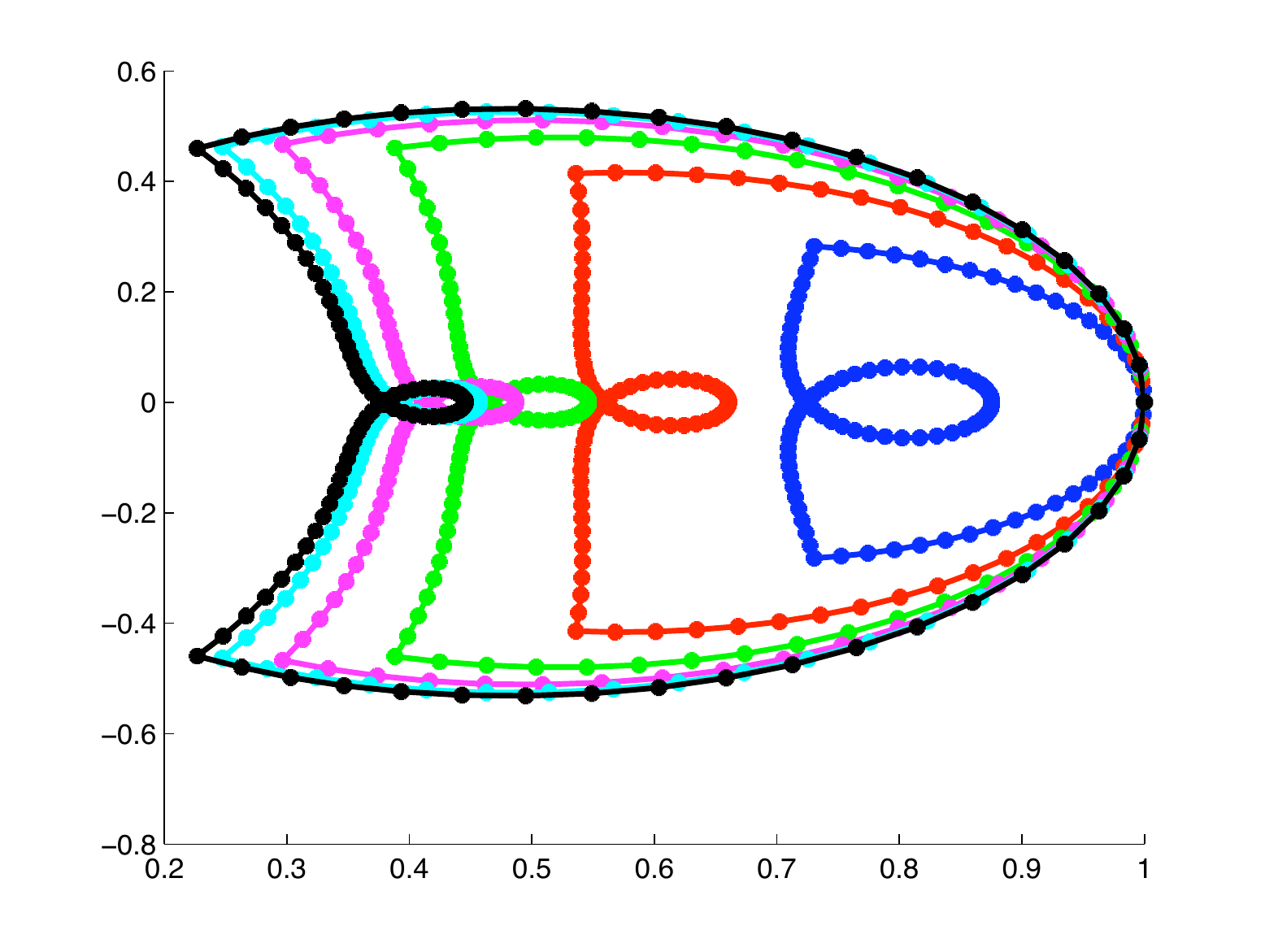}(d)
\end{array}$
\end{center}
\caption{Large amplitude limits, parameters
$K=0.7$, $J=0.5$, and $a=10^{-1}, 10^{-2},\dots , 10^{-k}$, 
getting smaller as necessary to see what appear to be convergence to a 
limit.
(a).  Lax $2$-shock, $v_2$ to $v_1$.
(b). Intermediate Lax $2$-shock, $v_3$ to $v_1$.
(c). overcompressive $1$-$2$ shock, $v_4$ to $v_1$.
(d). Transverse Evans study for (c).
In each case, we appear to obtain convergence at $a=10^{-7}$, corresponding to Mach number $\approx 3,817$.
}
\label{large4}
\end{figure}

\begin{figure}[t]
\includegraphics[width=7.5cm]{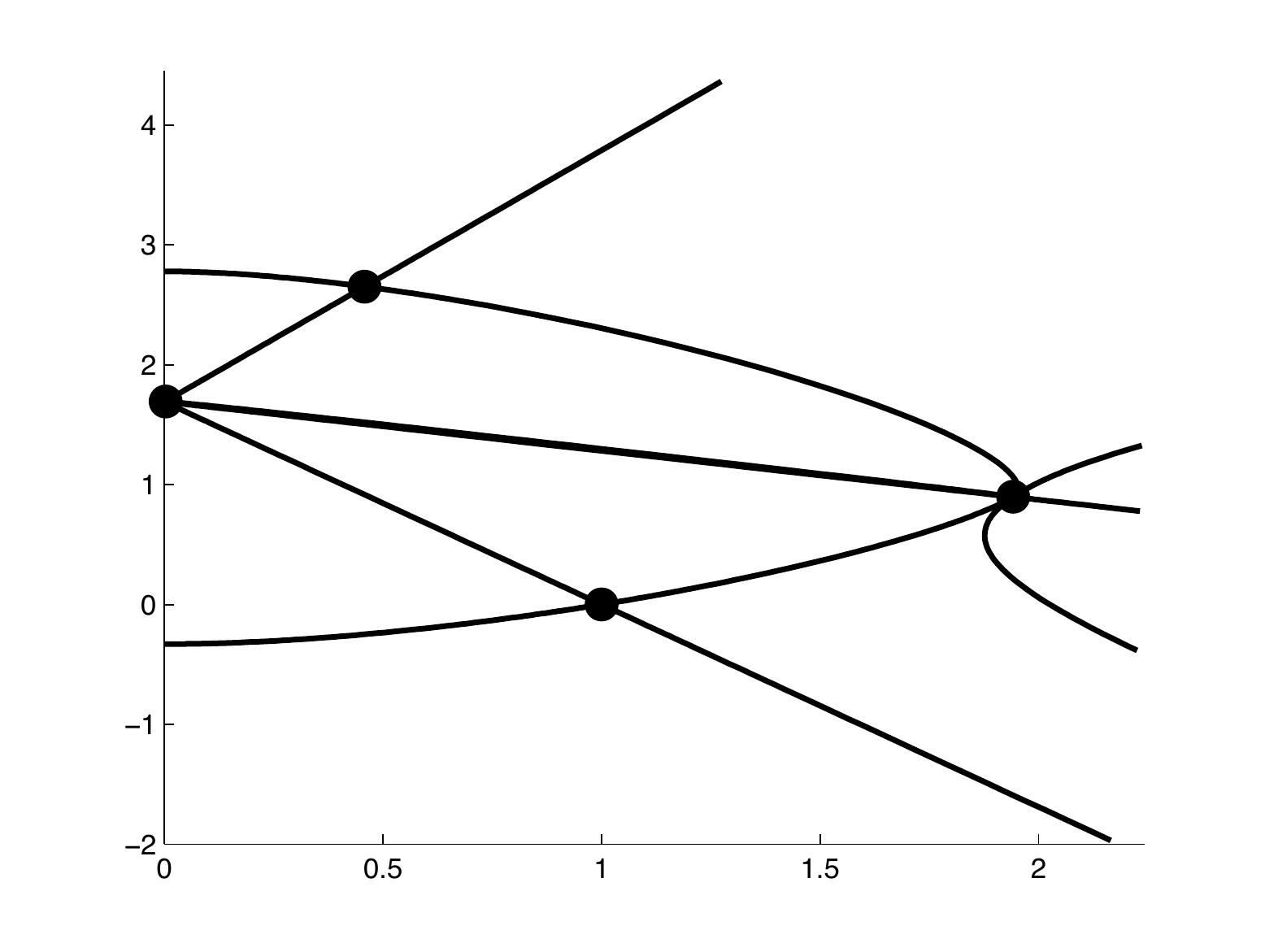}
\caption{
Phase portrait corresponding to Figure \ref{large4}, parameter values
$K=0.7$, $J=0.5$, and $a=10^{-8}\sim 0$.
}
\label{phaselimit}
\end{figure}

\subsection{Three-dimensional stability}\label{s:3dexp}
As discussed in section \ref{s:analytical}, transverse stability holds automatically in the case $\sigma=\infty$ for profiles that are monotone 
decreasing in $\hat v$, so for all the studies described previously, we examined stability of \eqref{3intevalinfty} only in the case of
a nonmonotone profile. All computations were consistent with stability.

\begin{figure}[htbp]
\begin{center}
$\begin{array}{lr}
\includegraphics[width=7.5cm]{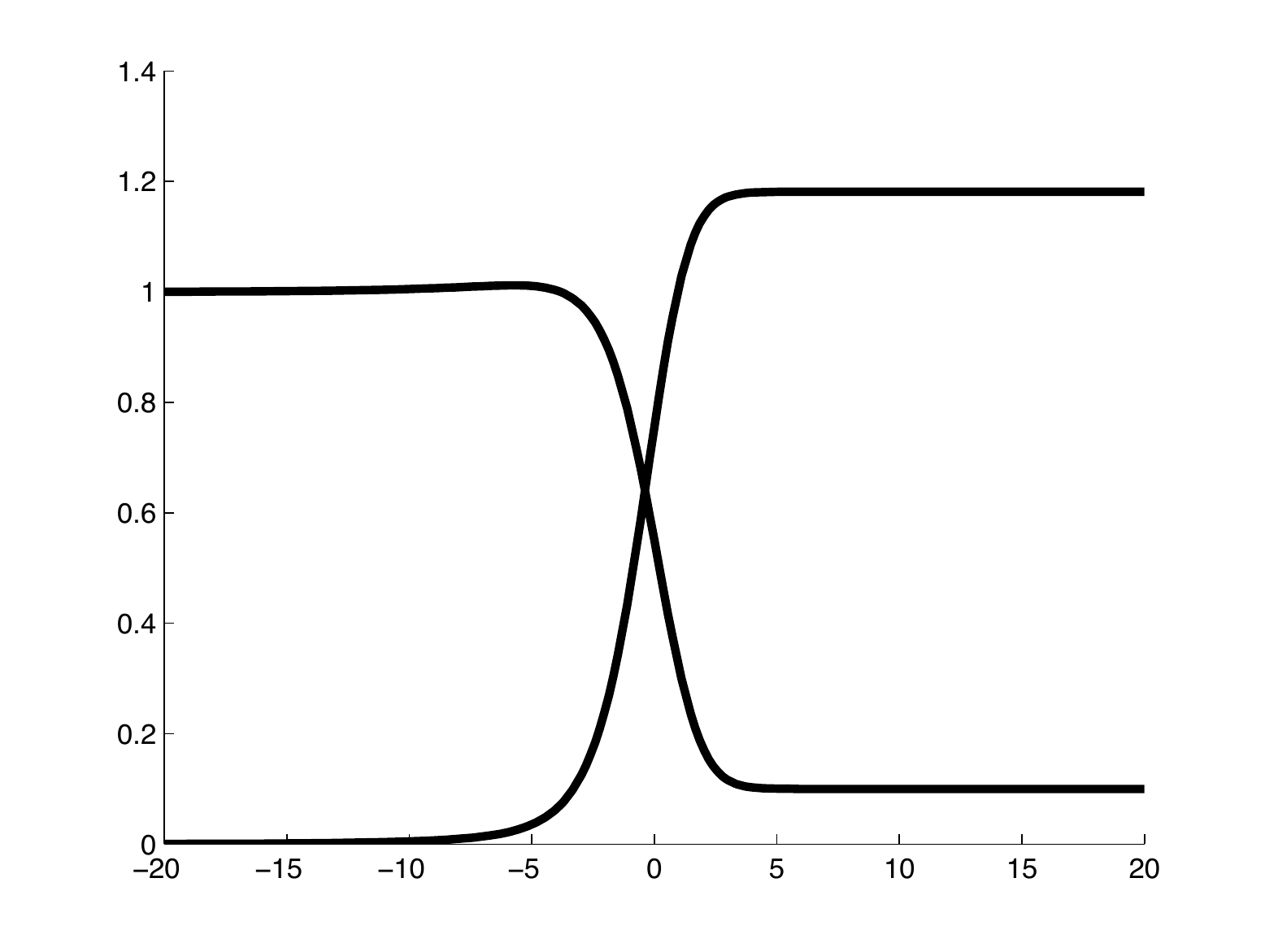}(a) & \includegraphics[width=7.5cm]{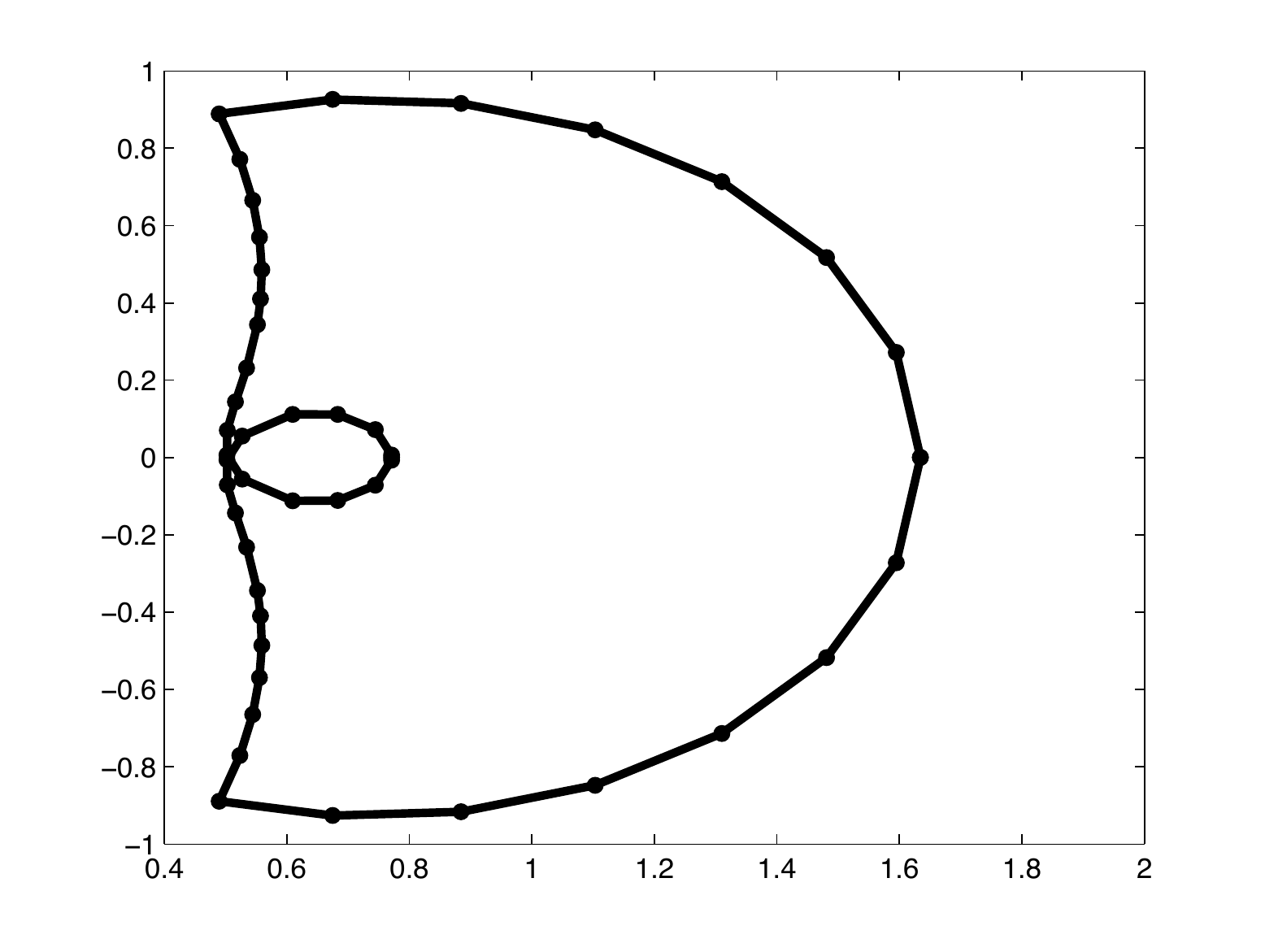}(b)
\end{array}$
\end{center}
\caption{Typical transverse Evans function output, parameter
values $\gamma=5/3$, $I=0.6$, $B_+=1.4$, and $\mu_0=1$.  
In Figure (a) we display the nonmonotone profile.
In Figure (b) we display the winding number computation.}
\label{EvansT}
\end{figure}

\subsection{The undercompressive case}\label{s:ucexp}
For our undercompressive study we, taking care to avoid repetitions, considered the parameter combinations
\begin{align*}
(\gamma,v_+,K,J) &\in \{7/5, 5/3\}\\
&\quad \times \{0.1, 0.2, 0.3, 0.4, 0.5, 0.6, 0.7, 0.8, 0.9\}\\
&\quad \times \{0.1, 0.2, 0.3, 0.4, 0.5, 0.6, 0.7, 0.8, 0.9\}\\
&\quad \times \{0.1, 0.2, 0.3, 0.4, 0.5, 0.6, 0.7, 0.8, 0.9\}
\end{align*}
for which a four rest point configuration exists in phase space. We fixed  $\eta$ and let $\mu$ be a free parameter in the boundary value problem allowing us to solve simultaneously the value of $\mu$ for which a traveling wave connects the saddle points and for the profile itself. We successfully found and examined
 over 250 undercompressive profiles for Evans stability. Since undercompressive profiles are monotone implying stability in the transverse case, we only computed the Evans function associated with \eqref{Amatrix2}.  Because the Evans function output for undercompressive waves has a zero at the origin, we used a small half circle of radius $10^{-3}$ as part of the domain contour to skirt around the origin. All winding number results were consistent with stability. Typical output is displayed in Figure \ref{EvansUC}.

 \begin{figure}[htbp]
\begin{center}
$\begin{array}{lr}
\includegraphics[width=7.5cm]{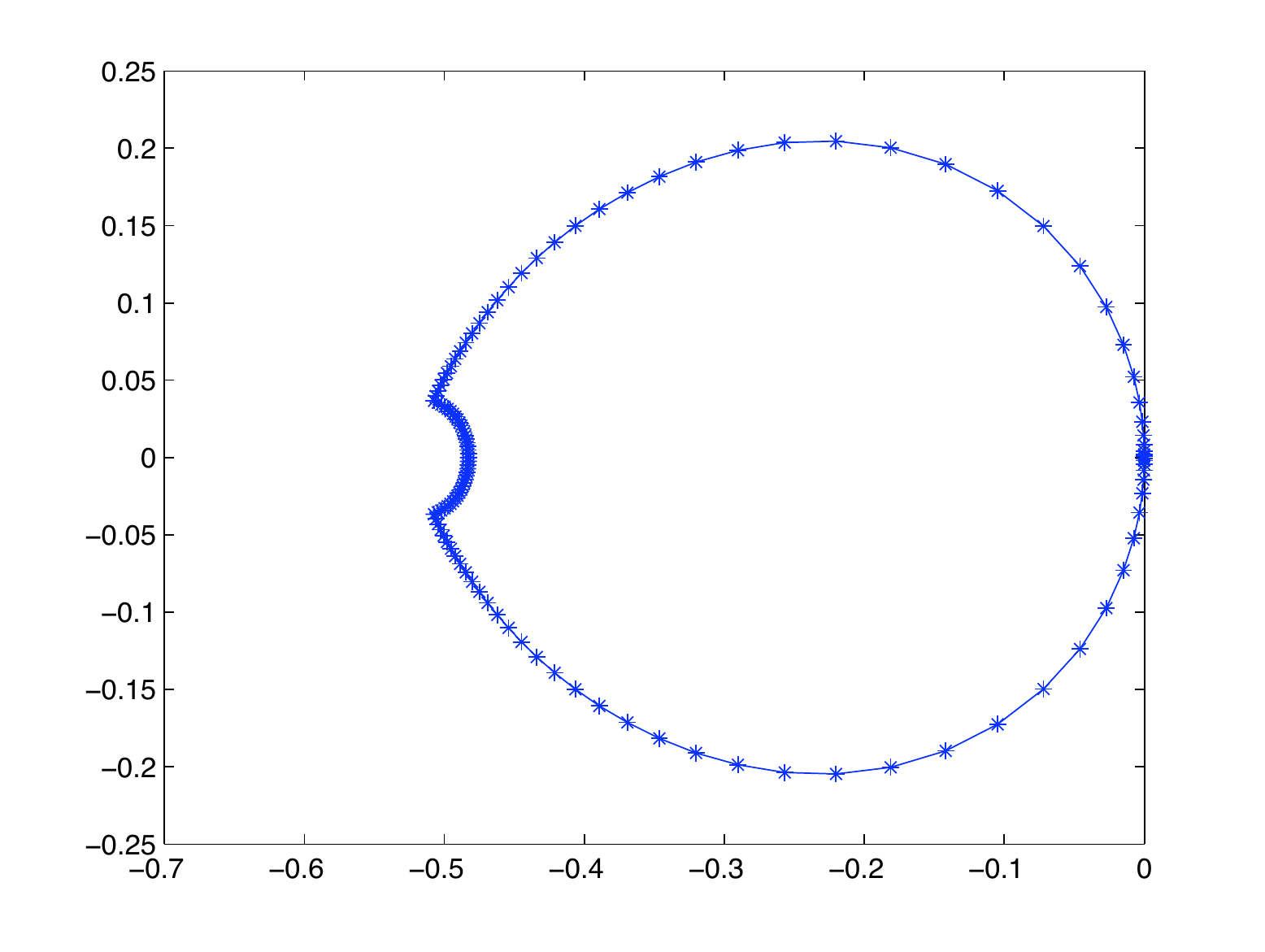}(a) & \includegraphics[width=7.5cm]{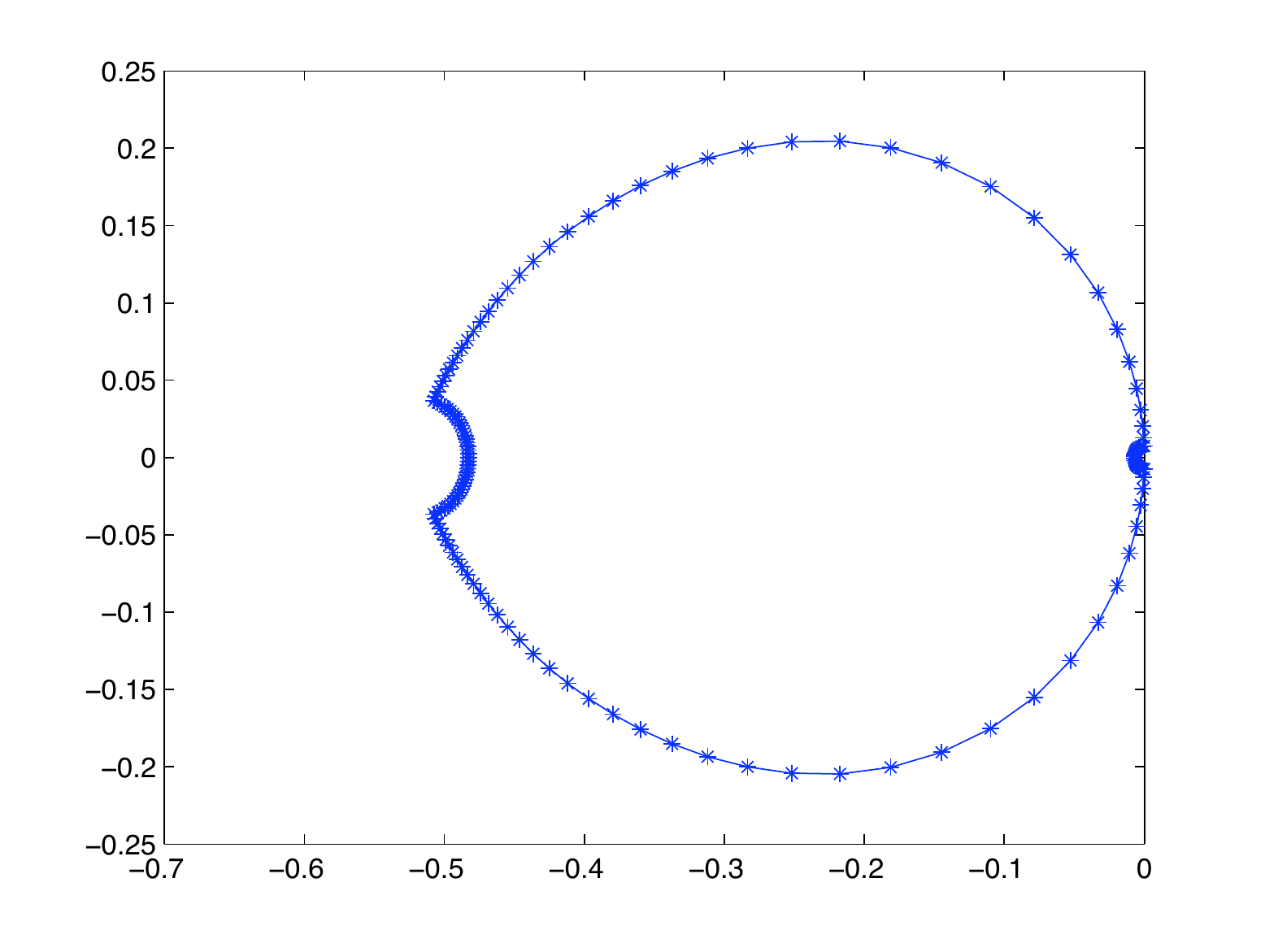}(b)
\end{array}$
\end{center}
	\caption{Typical Evans function output, undercompressive case; parameter
values $v_+=.3809$; $I=0.8$; $B_+=1.1692$; $\gamma=5/3$; $\mu=0.2381$; and $2\mu+\eta=1$.
In Figure (a), our domain contour comes within $10^{-4}$ of the origin, 
with a gap between mesh points $i10^{-4}$ and $-i10^{-4}$,
and our range contour comes within $10^{-6}$ of the origin. 
In Figure (b), our domain contour comes within 
$10^{-3}$ of the origin, then follows a small semicirle around it,
the image of which may be seen at the far righthand side of the figure.
(For undercompressive shocks, the integrated Evans function has a zero
at the origin \cite{ZH,MaZ3}; see Proposition \ref{intcond}.)
}
\label{EvansUC}
\end{figure}

\section{Discussion and open problems}\label{s:discussion}
In this paper, we have carried out by a combination of asymptotic
ODE analysis and numerical Evans function computations
a global existence/stability study for viscous shock profiles of 
two-dimensional isentropic magnetohydrodynamics with infinite
electrical resistivity.
For a monatomic $\gamma$-gas equation of state, and standard
viscosity ration $\eta=-2\mu/3$, we find that Lax and overcompressive
profiles appear but undercompressive profiles do not.
A systematic numerical Evans function investigation indicates
that all profiles are nonlinearly stable both with respect to
two-dimensional and three-dimensional perturbations.
For different viscosity ratios, undercompressive shocks can appear,
and these appear also to be stable with respect to two- and 
three-dimensional perturbations.

Our stability analysis generalizes previous viscous studies of
the viscous stability problem in \cite{FT,BHZ} for
the parallel case.  
See also the investigations of stability in the small-magnetic field limit 
in \cite{MeZ,GMWZ2}.
For analyses of the related inviscid stability problem, 
see, e.g., \cite{T,BT,MeZ} and references therein.

Much of the analysis carries over to the full three-dimensional case;
in particular, the Rankine--Hugoniot analysis is completely general.
It would be very interesting to carry out a systematic analysis in
three dimensions, following the approach laid out here.
Genuinely three-dimensional profiles, having a richer structure
and more degrees of freedom, would appear to be a good place to
look for possible instability or bifurcation.
(Here, three-dimensional refers as in the present paper
to the dependent variables and not the independent variable $x$.)
As noted in \cite{TZ}, instability,
by stability index considerations, would for an ideal gas equation
of state imply the interesting phenomenon of
Hopf bifurcation to time-periodic, or ``galloping'' behavior
at the transition to instability.

Likewise, it would be interesting to carry out the full, nonisentropic
case, building on the Rankine--Hugoniot study of \cite{FR1}.
This should in principle be straightforward using the methods
developed here, but more computationally intensive.
We suspect that, as in the gas-dynamical case \cite{HLyZ1}, 
the large-amplitude limit may in fact be straightforward in this case,
with $v_+$ bounded from the value zero at which the pressure
function has a singularity.
Further interesting generalizations would be to consider a
Van der Waal or other ``real gas'' equation of state,  
or to include third-order dispersive effects modeling
a ``Hall effect'' as in \cite{DR}.

Interesting boundary cases left open in the present analysis
are the $K\to \infty$ singular perturbation problem discussed
in Section \ref{inflim} and the large-amplitude limit $v_+\to 0$
for Lax $2$-shocks and overcompressives.
We conjecture that in each of these cases, the Evans function converges
in the limit to the Evans function for the formal pressureless
gas limit $a=0$.
(As shown in Appendix \ref{s:liim}, all of these limits
coincide with $a\to 0$.)
Our numerics are consistent with this conjecture, while
the related analysis of \cite{HLZ} gives an idea how to prove it.
At the physical, modeling, level, an important problem is to
determine physically interesting values
of $J$, $K$, $a$, and the viscosity ratio $r=\mu/(2\mu + \eta)$,
which strongly affects solution structure as we have seen.

The absence of instabilities in our experiments suggests perhaps
the more general question whether shock
profiles for systems possessing a convex entropy are always stable.
We do not at the moment see why this should be so,
and suspect that an ideal gas equation of state is perhaps
too simple an example on which to base conclusions.
However, to verify or produce a counterexample to this conjecture
takes on a larger importance in light of the growing body of
stable examples.

Another interesting direction for further investigation would be
a corresponding comprehensive study of multi-dimensional stability,
as carried out for gas-dynamical shocks in \cite{HLyZ2}. 
(Here, multi-dimensional refers to the independent variable $x$.)
As pointed out in \cite{FT}, instability results
of \cite{BT,T} for the corresponding inviscid problem imply 
that parallel shock layers become multi-dimensionally unstable for
large enough magnetic field,
by the general result \cite{ZS,Z1,Z2} that inviscid stability
is necessary for viscous stability, so that in multi-dimensions
instability definitely occurs.
The question in this case is whether viscous effects can hasten
the onset of instability, that is, whether
viscous instability can occur in the presence of inviscid stability.

\appendix
\section{Signature of $\nabla^2\check \phi$}\label{signature}
Taking the Hessian of the relative entropy $\check \phi$
defined in \eqref{ischeckphi}, using \eqref{gradcheckphi},
we readily obtain
\be\label{hessian}
\nabla^2_{(v,w,B/\mu_0)}\check \phi=
\bp
p'(v)+1& 0 & \mu_0 (B/\mu_0)\\
0 & 1 & -I\\
\mu_0 (B/\mu_0)& -I & v\mu_0
\ep,
\ee
yielding after a brief computation
$$
\det \nabla^2_{(v,w,B/\mu_0)}\check \phi=
\mu_0(p'(v)+1)(v-K) -|B|^2.
$$
At rest points of the traveling wave equation,
we have $B=J((1-K)^2/(v-K)^2$, by \eqref{rhsoln},
yielding
$\det \nabla^2_{(v,w,B/\mu_0)}\check \phi=
\mu_0(v-K)\tilde f'(v)$,
where $\tilde f$ is as defined in \eqref{tildef}.
This gives an explicit connection between the signature
of $\nabla^2\check \phi$ and the sign of the derivative
of the reduced Rankine--Hugoniot function $\tilde f$.

\section{Computing the Mach number}\label{Machnum}
The Mach number of a Lax shock is defined as
\[
M = \frac{u_- - \sigma}{c_-},
\]
where $u_-$ is the downwind velocity, 
$\sigma$ is the shock speed, and $c_-$ is the downwind 
sound speed (in the characteristic family of the shock), 
all in Eulerian coordinates. 
(Here, the downwind side is at $-\infty$, since we consider a
left-moving shock.)
By considering the conservation of mass equation, we have $\rho_t + (\rho u)_x = 0$, where $\rho=1/v$ is density.
Hence, the jump condition is given by $\sigma [\rho] = [\rho u]$, 
which implies, in the original scaling for \eqref{MHD}, that
\[
\sigma = \frac{u_+ v_- - u_- v_+}{v_- - v_+}.
\]
Hence, 
\[
M = \frac{u_- - \sigma}{c_-} = \frac{v_-(u_- - u_+)}{c_-(v_- - v_+)} 
= \frac{v_- [u]}{c_- [v]} = -s \frac{v_-}{c_-}.
\]
Noting that $0<v_+<v_-=1$, we simplify to get
\[
M = \frac{1}{c_-},
\]
where $c_-$ is the sound speed in Eulerian coordinates
at $v-=1$, $u_-=0$, $w_-=0$.

Here, for Lax $1$-shocks (resp. $2$-shocks),
$$
\begin{aligned}
c_-^2&=
\frac{1}{2}\Big( [c_s^2 + (\rho \mu_0)^{-1} (B_-^2+ I^2) ]
\pm \sqrt{
[c_s^2 + (\rho\mu_0)^{-1}(B_-^2+ I^2)]^2 -4 c_s^2 (\rho\mu_0)^{-1}I^2
}
\Big)\\
&=
\frac{1}{2}\Big( [\gamma a + 2J + K ]
\pm \sqrt{ [\gamma a + 2J + K ]^2 -4 \gamma a K } \Big)
\end{aligned}
$$
by \eqref{phi},
where $c_s=\sqrt{ dp/d\rho}|_{\rho=1}=\sqrt{a\gamma}$ denotes sound speed \cite{A,MeZ}. 
In the parallel case $J=0$, this gives
$c_-= \sqrt{ \gamma a}$,
or $M=\frac{1}{\sqrt { \gamma a}}$, for 
Lax $2$-shocks or Lax $1$-shocks with $\gamma a>K$, 
in agreement with the standard Mach number for gas-dynamical shocks.
However, for Lax $1$-shocks with $\gamma a<K$, it gives the
anomalous value $M=1/K$. (Recall that MHD profiles in the parallel
case reduce to gas-dynamical profiles.)

It is readily verified that the Mach number is invariant under the
rescaling \eqref{scaling},
hence gives a useful measure of shock strength in the original
unrescaled coordinates.
However, the example of the parallel case shows that it can give
anomalous values for other than the simple gas-dynamical case.
Note also that this measure of shock strength involves only the left state
at $-\infty$, so does not distinguish between intermediate vs.
regular types of shocks.  
We therefore use the Mach number only to give a rough idea of the strength
of shocks considered in our studies, and not as a systematic measure of
strength across all parameters.

\section{Limiting cases for Lax and overcompressive shocks}\label{s:liim}

\subsection{$a\to 0$}
One may ask the questions of the limits of positivity of $a$ given by 
\eqref{asoln}. One has
\be\label{again}
a=\frac{1-v_+}{v_+^{-\gamma}-1}\Big(1-J\frac{1+v_+-2K}{(v_+-K)^2}\Big).
\ee
The two roots of the term between brackets are real if and only if 
$\frac{J}{4}>1-K$, the bracket being always positive when $J<4(1-K)$. 
In the former case, there are two roots 
$\underline v< \overline v$, with $a<0$ for $\underline v<v_+<\overline v$.
Moreover, $\underline v(J, K)=K+\frac{J}{2}-\sqrt{\frac{J^2}{4}+J(1-K)}$, 
so $\underline v(J, K)>0$ 
for $K^2-J(1-2K)>0$. If $K>\frac12$, this is always true. 
If $K<\frac12$, it is true for $J< \frac{K^2}{1-2K}$, 
with $\underline v=0$ precisely on the limiting curve
$J=\frac{K^2}{1-2K}$ above which there are no rest points with values
$v_+<K$.

Returning to the discussion,
we have thus, for the two roots between $0$ and $K$, 
that the limiting values of $v_+$ for which $a\to 0$ are $v_+\to 0$,
for which the factor $\frac{1-v_+}{v_+^{-\gamma}-1}$ 
goes to zero while the factor in brackets remains bounded, 
and $v_+\to \underline v\ge 0$, with the two limits coinciding
precisely in the case $J=\frac{K^2}{1-2K}$.
This means, for fixed $(J,K)$, 
that $v_+\to 0$ in the large-amplitude limit for Lax $2$-shocks
or overcompressives (the ones involving the rest point $v_1$ with
smallest $v$-value, and for which $v_1<K$), except on the measure zero boundary
$J=\frac{K^2}{1-2K}$.
On the other hand, noting that $K$ (since the term in brackets
is negative there) lies always between $\underline v$ and $\overline v$,
values $v_+$ for Lax $2$-shocks and overcompressives remain bounded
away from the value $v_+=K$ at which $a$ becomes singular,
so long as $J\ne 0$.

{\it That is, for Lax $2$-shocks and overcompressives,
the boundary of existence for fixed $J$, $K$ is marked
by the singular limit $v_+\to 0$ on one side,
and $v_+\to \underline v(J,K)\ge 0$ on the other side.}
In the measure-zero case $J=\frac{K^2}{1-2K}$, the two
limits agree.
For Lax $1$-shocks on the other hand, 
$K<v_+$, hence $v_+\to 0$ only
in the limit as $J$, $K\to 0$, and so this singular limit
does not arise for $J$, $K$ fixed.
Neither does $v_+\to K$, unless $J=0$.

\subsection{$a\to \infty$?}\label{inflim}
For $K<1$, the upper limit for $a$ is the characteristic boundary
$a\le A(J,K)$ described in Proposition \ref{4desc}.
For $K>1$,
may take without loss of generality $a<a_{min}(J,K)$, 
where $a_{min}$ is defined as the minimum value of $a$ 
at which rest points $v>K$ appear,
since in the latter case one can then rescale to the case
$K<1$ already treated.
Thus, for $J\ge 0$ and $0\le K\ne 1$ bounded, $a$ may be
taken always finite.

In the case $K<1$, we've already seen in Remark \ref{largermk}
that $J$ is finite for fixed $K\ge 0$ bounded from $1$. 
In the case $K>1$, by Remark \ref{Kg1}, we have $J<4(K-1)$
to begin with, once we eliminate four rest point configurations
(as we may do by rescaling so that $K<1$).
So, for fixed $K>1$, we get bounded $J$, $K$, hence bounded
$a$ by \eqref{asoln} and $0\le v_+\le 1$.
Combining these observations, we find for any bounded $K$ that
is also bounded from $1$ that $J$ and $a$ may be taken
bounded as well.
{\it Thus, we need only consider finite parameter values $(a,J)$ for
$K$ bounded and bounded from $1$.}

The sole remaining case is that of a Lax $2$-shock,
$v_+<1<K$, with $K$ going to infinity and $J<4(K-1)$.
Consulting again \eqref{again}, we see that in this
case $a\sim J/K\lesssim 4$ as $K\to \infty$, so that
$a$ is again uniformly bounded.
(On the other hand, this case can certainly occur for
$a$ sufficiently small.)
The conclusion is that, without loss of generality (i.e., rescaling
four rest point configurations to $K<1$ whenever they occur),
$a$ may be taken uniformly bounded, independent of $J$, $K$,
so that {\it $a\to \infty$ does not occur}.
On the other hand, the case $K\to \infty$ can occur, and
even $J\to \infty$, $K\to \infty$ simultaneously, with
$a$ remaining finite.

In this latter case, the profile ODE, and the associated stability
problem, should be treatable by a singular perturbation analysis,
rescaling $x$, in which the pressure term disappears.
However, we do not carry out this analysis here.

\subsection{Characteristic boundaries}\label{s:char}
Other important limits are the parameter values
for which the shock becomes characteristic at
$U_+$ or $U_-$, since the rate
of exponential convergence of the shock profile goes
to zero as they are approached, so that the length of the computational
domain $[-L_-,L_+]$ needed for accurate numerical approximation
goes to infinity.
These are given by the surfaces
$a=a_*(J,K) =\frac{1-K-2J}{\gamma (1-K)}$ (corresponding to 
$v_+=1$) and $a=A(J,K)$ (corresponding to $v_1=v_2$). 
The first is resolved by the analytical result of small-amplitude
stability.  The second requires a refined analysis outside
of the scope of this paper, involving stability of characterstic shocks.
For results in this direction, see \cite{HoZ2,Ho}.

\section{The conjugation lemma}\label{s:conj}

Consider a general first-order system
\be \label{gen}
W'=A^p(x,\lambda)W
\ee
with asymptotic limits $A^p_\pm$ as $x\to \pm \infty$,
where $p\in \RR^m$ denote model parameters.
\begin{lemma} [\cite{MeZ1,PZ}]\label{conjlem}
Suppose for fixed $\theta>0$ and $C>0$ that 
\be\label{udecay}
|A^p-A^p_\pm|(x,\lambda)\le Ce^{-\theta |x|}
\ee
for $x\gtrless 0$ uniformly for $(\lambda,p)$ in a neighborhood of 
$(\lambda_0)$, $p_0$ and that $A$ varies analytically in $\lambda$ 
and smoothly (resp. continuously) in $p$ 
as a function into $L^\infty(x)$.
Then, there exist in a neighborhood of $(\lambda_0,p_0)$
invertible linear transformations $P^p_+(x,\lambda)=I+\Theta_+^p(x,\lambda)$ 
and $P_-^p(x,\lambda) =I+\Theta_-^p(x,\lambda)$ defined
on $x\ge 0$ and $x\le 0$, respectively,
analytic in $\lambda$ and smooth (resp. continuous) in $p$ 
as functions into $L^\infty [0,\pm\infty)$, such that
\begin{equation}
\label{Pdecay} 
| \Theta_\pm^p |\le C_1 e^{-\bar \theta |x|}
\quad
\text{\rm for } x\gtrless 0,
\end{equation}
for any $0<\btheta<\theta$, some $C_1=C_1(\bar \theta, \theta)>0$,
and the change of coordinates $W=:P^p_\pm Z$ reduces \eqref{gen} to 
the constant-coefficient limiting systems
\begin{equation}
\label{glimit}
Z'=A^p_\pm Z 
\quad
\text{\rm for } x\gtrless 0.
\end{equation}
\end{lemma}

\begin{proof}
The conjugators $P^p_\pm$ are constructed by a
fixed point argument \cite{MeZ1}
as the solution of an integral equation
corresponding to the homological equation
\be\label{homolog}
P'=A^pP-A^p_\pm P.
\ee
The exponential decay \eqref{udecay} is needed to 
make the integral equation contractive with respect to $L^\infty[M,+\infty)$ 
for $M$  sufficiently large.  Continuity of $P_\pm$ with respect 
to $p$ (resp. analyticity with respect to $\lambda$) then follow
by continuous (resp. analytic) dependence on parameters of 
fixed point solutions.
Here, we are using also the fact that \eqref{udecay} plus continuity
of $A^p$ from $p\to L^\infty$ together imply continuity of
$e^{\tilde \theta |x|}(A^p-A^p_\pm)$ from $p$ into $L^\infty[0,\pm\infty)$
for any $0<\tilde \theta < \theta$, in order to obtain the needed
continuity from $p\to L^\infty$ of the fixed point mapping.
See also \cite{PZ}.
\end{proof}

%

\begin{definition}[Abstract Evans function]\label{abevans}
Suppose that on the interior of a set $\Omega$ in $\lambda$, $p$, the 
dimensions of the stable and unstable subspaces of $A_\pm^p(\lambda)$
remain constant, and agree at $\pm \infty$ 
(``consistent splitting'' \cite{AGJ}), and that these subspaces
have analytic bases $R_j^\pm$ extending
continuously to boundary points of $\Omega$.  Then, 
the Evans function is defined on $\Omega$ as
\be\label{eq:evans}
\begin{aligned}
D^p(\lambda)&:=
\det( P^+R_1^+,\dots,P^+R_k^+, P^-R_{k+1}^-,\dots, P^-R_N^-)|_{x=0},\\
\end{aligned}
\ee
where $P_\pm^p$ are as in Lemma \ref{conjlem}.
\end{definition}

Evidently, $W_j^+:=P^+R_j^+$, $j=1, \dots, k\}$
and $W_j^-:=P^-R_j^-$, $j=k+1, \dots, N\}$
are bases for the manifolds of solutions $W$ of
$W'=A^pW$ decaying as $x\to +\infty$ and $x\to -\infty$,
respectively, analytic in $\lambda$ and smooth in $x$,
with $W_j^\pm(x)\sim e^{A_\pm x}R_j\pm$ as $x\to \pm \infty$.
(Here, we suppress $p$ for notational convenience.)
Thus, $D^p$ has the alternative representation
$$
D^p(\lambda)=
\det( W_1^+, \dots, W_k^+, W_{k+1}^-, \dots, W_N )_{x=0}
$$
given in \eqref{Ddef1}.

Using the fact that $P^\pm \to I$ as $x\to \pm \infty$,
and the fact that modes $W_j^\pm$ are growing as $x$
goes from $\pm \infty$ to $0$ with undesired modes
exponentially decaying,
it is not difficult to see 
that the Evans function can be well-approximated
by replacing $W_j^\pm$ with solutions $W_j^{\pm, approx}$ 
of \eqref{gen} with data
$$
W_j^{\pm,approx}(\pm L)= e^{A_\pm(\lambda) (\pm L)}R_j^\pm
=
P^\pm(\pm L)^{-1}W_j^\pm(\pm L).
$$
(Compare solutions
$Z_j^\pm=P^\pm W_j^\pm$ and $Z_j^{\pm,approx}=P^\pm W_j^{\pm,approx}$
of the constant-coefficient equations $Z'=A_\pm Z$.)
This can be used as the basis for numerical approximation 
of the Evans function, as described, e.g., in 
\cite{Br1,Br2,BrZ,BDG,HuZ,Z3,Z4}.

\medskip

{\bf Acknowledgement.}
Thanks to Heinrich Freist\"uhler and Christian Roehde for
helpful orienting discussions preliminary to this project
and for generously sharing information regarding their numerical
existence studies in the full nonisentropic case.
We gratefully acknowledge the contribution of Jeffrey Humpherys
in the development of the STABLAB package with which our
numerical Evans studies were carried out and in his collaboration
on the series of parallel investigations \cite{BHRZ,HLZ,CHNZ,HLyZ1,HLyZ2,BHZ}, 
and to Tom Bridges,
Gianne Derks, Jeffrey Humpherys, Bjorn Sandstede, and others,
for ongoing stimulating discussions on numerical Evans function
computations.

\end{document}